%% file: ampleHamming5.tex
\documentclass[11pt,a4paper]{amsart}
\usepackage{amsmath,amsfonts,amsthm,amssymb}
\usepackage[T1]{fontenc}
\usepackage{color}
\usepackage{epsfig}
\usepackage{graphicx}
\usepackage{hyperref}
\usepackage[capitalise]{cleveref}
\usepackage{xspace}
\usepackage{enumitem}
\usepackage[left=1in,right=1in,top=1in,bottom=1in,foot=0.5in]{geometry}
\usepackage{thm-restate}
\usepackage{soul}
\usepackage[square,sort,comma,numbers]{natbib}
\usepackage{etoolbox}

\usepackage{tikz}
\usepackage{tkz-euclide}
\usetikzlibrary{positioning,calc}
\usetikzlibrary{decorations.pathmorphing, arrows.meta}
\usetikzlibrary{shapes.geometric}
\tikzset{harp/.style={arrows = {[harpoon]<->[harpoon]}}} 
\usepackage{color} 
\usepackage{standalone}

\usepackage{amsthm}
\newtheorem{theorem}{Theorem}
\newtheorem{proposition}{Proposition}
\newtheorem{lemma}{Lemma}
\newtheorem{corollary}{Corollary}
\newtheorem{claim}{Claim}
\theoremstyle{definition}
\newtheorem{definition}{Definition}
\newtheorem{remark}{Remark}
\newtheorem{case}{Case}

\AtBeginEnvironment{proof}{\setcounter{case}{0}}
\newcommand{\matthew}[1]{{\color{blue}[Matthew: #1]}}
\newcommand{\victor}[1]{{\color{magenta}[Victor: #1]}}
\newcommand{\toremove}[1]{{\color{gray}[Old text:#1]}}

\newtheorem{conjecture}[theorem]{Conjecture}
\newtheorem{example}{Example}

\numberwithin{equation}{section}

\DeclareMathOperator{\ext}{ex}
\DeclareMathOperator{\conv}{conv}

\DeclareMathOperator{\supp}{supp}

\DeclareMathOperator{\Prod}{Prod}

\DeclareMathOperator{\MP}{MProd}
\DeclareMathOperator{\Part}{Part}
\DeclareMathOperator{\GPart}{GPart}
\DeclareMathOperator{\BPart}{BPart}
\DeclareMathOperator{\Amp}{Amp}

\DeclareMathOperator{\oX}{\overline{X}}
\DeclareMathOperator{\uX}{\underline{X}}
\DeclareMathOperator{\BBox}{Box}

\DeclareMathOperator{\MMP}{\rm{M_{mix}Prod}}
\DeclareMathOperator{\EMP}{\rm{M_{ext}Prod}}
\DeclareMathOperator{\EEMP}{\rm{M_{el}Prod}}
\DeclareMathOperator{\CMMP}{\rm{M_{cel}Prod}}
\DeclareMathOperator{\BMP}{\rm{M_{bin}Prod}}
\DeclareMathOperator{\BcMP}{\rm{M^{\circ}_{bin}Prod}}

\newcommand{\da}{\negthickspace\downarrow}

\newcommand{\cM}{\ensuremath{\mathcal{M}\xspace}}

\newcommand{\cB}{\ensuremath{\mathcal{B}\xspace}}
\newcommand{\cA}{\ensuremath{\mathcal{A}\xspace}}

\renewcommand{\dim}{\operatorname{dim}} 
\DeclareMathOperator{\vcdim}{VC-dim}
\DeclareMathOperator{\dens}{dens}	      
\DeclareMathOperator{\vcdens}{VC-dens}
	
\title{Ample sets in Cartesian products}

\author[V.\ Chepoi and M. Maat]{Victor Chepoi$^{1,3}$ and  Matthew Maat$^2$}

\date{\today}

\begin{document}

\maketitle

	\bigskip
	\centerline{$^{1}$LIS, Aix-Marseille Universit\'e, CNRS, and Universit\'e
	de Toulon}
	\centerline{Facult\'e des Sciences de Luminy, F-13288 Marseille Cedex 9,
	France}
	\centerline{ \sf{victor.chepoi@lis-lab.fr}}

	\bigskip
	\centerline{$^{2}$University of Twente, The Netherlands}
    \centerline{ \sf{m.t.maat@utwente.nl}}
    
   \bigskip
    \centerline{$^{3}$ Institut Universitaire de France (IUF)}

\begin{abstract} Ample sets of hypercubes, introduced by A. Dress in 1995, constitute an interesting combinatorial structure with rich properties and important examples. 
They are the subsets $S$ of the hypercube $\{ 0,1\}^E$ such that any subhypercube $\{ 0,1\}^Y, Y\subseteq E$ shattered by $S$ is strongly shattered by $S$. Ample sets can be characterized in a multitude of combinatorial, graph-theoretical, recursive, and geometrical ways, and they are equivalent to lopsided sets introduced by J. Lawrence in 1983. 

In this paper, we define and investigate ample sets of Cartesian products of finite sets, i.e. of $U=U_1\times\ldots\times U_m$. 
This is done using \emph{minor-subproducts} of $U$, which 
correspond to products of partitions of the factors $U_1,\ldots,U_m$: each minor-subproduct $M$ is obtained by partitioning each  $U_i$ into blocks and contracting each block into a single element. 
For a minor-subproduct $M$ and a set $S\subseteq U$, we define the notions of \emph{shattering} of $M$ by $S$, of \emph{copy} of $M$ in $S$, of \emph{projection} $S_M$ of $S$ on $M$, and of  \emph{strong-projection} $S^M$ of $S$ on $M$. 
We call a set $S\subseteq U$ \emph{ample} if for any minor-subproduct $M$ that is shattered by $S$, there exists a copy of $M$ included in $S$. Using the lattice structure of minor-subproducts, we also define \emph{lopsided} sets. Differently from the binary case, ampleness is no longer equivalent to lopsidedness. 

We prove however that several characterizations of classical ample sets can be extended to ample sets of Cartesian products. In particular, we show that ampleness of $S$ is equivalent to any of the following: ampleness of the complement $S^*=U\setminus S$, isometricity of $S^M$ for any minor-subproduct $M$ (\emph{superisometricity}),  and \emph{commutativity} $(S^M)_{M'}=(S_{M'})^M$ for all minor-subproducts $M,M'$ with disjoint supports. We also provide more efficient characterizations of ampleness, in particular, by showing that $S$ is ample if and only of $S$ is isometric and both $S_e$ and $S^e$ are ample for some elementary minor-subproduct, if and only if the intersection of $S$ with any interval $[u,v]$ with $u,v\in S$ is ample in the classical sense. We also characterize ampleness by push downs and provide a decomposition theorem for ample sets, allowing us to prove that the prism complexes of ample sets are contractible. We provide new examples of ample sets arising from payoff games in graphs, prism-like polyhedra, and quasi-median graphs. Finally, we provide a unified treatment of  various notions of VC-dimension occurring in the literature on multiclass learning in terms of the dimension of shattered minor-subproducts of certain types. 
\end{abstract}

\bigskip

\section{Introduction}

\subsection{Avant-propos} 
Projection is a fundamental mathematical operation. For example, given a set system $S\subseteq 2^E=\{ 0,1\}^E$ and 
a subset $Y\subseteq E$, the \emph{trace} $S_{|Y}=\{ A\cap Y: A\in S\}$ of $S$ on $Y$ is  the projection of $S$ (viewed as a subset of vertices of the hypercube $\{ 0,1\}^E$) on the hypercube $\{ 0,1\}^Y$. If every element of  $\{ 0,1\}^Y$ is in the image of $S$ (i.e., $S_{|Y}=2^Y$), then $Y$ (or the hypercube $\{ 0,1\}^Y$) is said to be \emph{shattered} by $S$. The set $\oX(S)$ of all shattered sets  is a simplicial complex (a set system closed by taking subsets) and the dimension of this complex (the largest size of a set of $\oX(S)$) is the well-known \emph{Vapnik-Chervonenkis dimension} (\emph{VC-dimension} for short) of $S$. On the other hand, the hypercube $\{ 0,1\}^E$ is partitioned into \emph{copies} of $\{ 0,1\}^Y$, i.e., into ``parallel'' cubes of the form $s\times \{ 0,1\}^Y$ with $s\in \{ 0,1\}^{E\setminus Y}$. Then $Y$ (or $2^Y$)  is said to be \emph{strongly shattered} by $S$ if at least one such copy is included in $S$. The set $\uX(S)$ of all strongly shattered sets is again a simplicial complex and $\uX(S)\subseteq \oX(S)$. Dress \cite{Dr-ample} proved that for any set system $S\subseteq \{ 0,1\}^E$, the inequality $|\uX(S)|\le |S|\le |\oX(S)|$ holds, and suggested to called a set $S$ \emph{ample} if $|S|=|\oX(S)|.$ Amplenness of $S$ is equivalent to the equality  $\uX(S)=\oX(S)$ \cite{BaChDrKo}, which can be rephrased as the \emph{shattering$\rightarrow$strong-shattering principle}: whenever a set $Y\subseteq E$ is shattered by $S$, it is also strongly shattered by $S$. 

Ample sets are equivalent to \emph{lopsided sets} of  Lawrence \cite{La}, to \emph{simple sets} of  Wiedemann \cite{Wi}, and  to \emph{extremal sets}  of  Bollob\`as and Ratcliffe \cite{BoRa}. Notice also that the inequality $|S|\le |\oX(S)|$   was proved before by Pajor \cite{Pa} and that it implies the classical  Sauer-Shelah-Perles 
inequality $|S|\le \binom{m}{\le d}$ (where $d$ is the VC-dimension of $S$ and $m=|E|$). Lawrence's motivation for introducing lopsided sets was to investigate the subsets of $\{ 0,1\}^E$ that encode the intersection
pattern of a convex set $K$  with the orthants of  ${\mathbb R}^E$.  He identified a strong combinatorial condition  that he called lopsidedness that is necessary but not sufficient for the existence of such a convex set $K$. 
In our terms, a set $S\subseteq \{ 0,1\}^E$ is \emph{lopsided} if for any $Y\subseteq E$, either $Y$ is strongly shattered by $S$ or $E\setminus Y$ is strongly shattered by the complement $S^*=\{ 0,1\}^E\setminus S$ of $S$ (we refer to this as \emph{Lawrence's lopsidedness principle}). 

Ample/lopsided/extremal/simple sets have numerous characterizations and properties, established in 
\cite{BaChDrKo,BaChDrKo_geometry,BoRa,ChChMoWa,ChKnPh_CUOM,La,Wi}. 
They are of combinatorial, geometrical, recursive, graph-theoretical, metric, and topological nature. These characterizations show that ample sets can be equally dubbed \emph{commutative, superisometric, superconnected,}  or \emph{weakly convex}. Ample sets define cube complexes with strong topological and metric properties like contractibility, collapsibility,  and $\ell_1$-isometricity. Ampleness is preserved by taking complements, push-downs, restrictions, projections, and strong-projections. Ample sets include several important classes of sets arising from completely different research domains: median set systems, convex geometries and their bouquets \cite{BaChDrKo}, maximum classes of given VC-dimension \cite{Mo-thesis}, intersection
patterns of convex sets with the orthants of  ${\mathbb R}^E$ \cite{La}, and sets of strategies in mean payoff games in graphs \cite{maat_strategy_2025} are ample. Together with oriented matroids \cite{BjLVStWhZi}, ample sets were the motivating structures in the definition of \emph{Complexes of Oriented Matroids},  introduced in \cite{BaChKn} and investigated in subsequent papers. As a generalization of maximum classes, ample sets have also found applications in machine learning in relation with the \emph{sample compression conjecture} \cite{MoWa,ChChMoWa,ChChHaMoYe}.  

In this paper, we consider ample sets in Cartesian products. Cartesian products of graphs form a classical subject in graph theory \cite{ImKl}. Isometric subgraphs of hypercubes and Hamming graphs (Cartesian products of cliques), closely related to ampleness, have been investigated in \cite{Ch_Hamming,Dj,Wi-hamming,Wi-products}.  
More recently, subsets of Cartesian products of sets have been intensively studied in machine learning literature in relation with the theory of multiclass learning \cite{BDCeHaLo,RuBaRu,BrCaDiMoYe}. One relevant challenge in this theory is the appropriate  definition of  VC-dimension for subsets of Cartesian products.  Several variants have been suggested: Natarajan dimension \cite{Na}, Daniely and Shalev-Shwartz dimension (DS-dimension) \cite{DaSS}, Pollard dimension \cite{Po}, and graph-dimension \cite{Na}, just to name a few. It was believed that  Natarajan dimension  characterizes the multiclass learnability and only recently it was proved in \cite{BrCaDiMoYe} that this is not the case and that the appropriate notion of dimension is the DS-dimension.

\subsection{Our approach} The main goal of this paper is to define and investigate 
ample sets in Cartesian products, based on the \emph{shattering$\rightarrow$strong-shattering principle}. 
This generalizes the classical setting of ample sets in hypercubes, i.e., in products of two-element sets. Our main contribution is to show that most characterizations of binary ample sets generalize to ample sets in Cartesian products. Compared to the binary case, the required conceptual framework and the proof techniques are much more involved. To define ampleness in Cartesian products, we need to define patterns that can be shattered, and for which there is a notion of copy. While in the binary case, there is no ambiguity to how these concepts are defined, in general Cartesian products there are several ways to define shattered items. In this paper, we undertake a novel but general and systematic approach  by considering the equivalent \emph{generalized partitions}, \emph{box-partitions} and \emph{minor-subproducts} as patterns. The minor-subproducts are also used to define the operations of projection and strong-projection as an analogue to the binary case.

Let $U:=U_1\times\ldots\times U_m$  be the Cartesian product of nonempty finite sets $U_1,\ldots, U_m$. Let also $H(U)$ be the Hamming graph of $U$ defined as the Cartesian product of the cliques $U_1,\ldots,U_m$. For each factor $U_i$, we consider the lattice $\Part(U_i)$ of all partitions of $U_i$, endowed with the partition refinement operation as the partial order.\footnote{For example, the lattice of partitions of $\{ 1,2,3,4\}$ has 15 elements: the minimum, the maximum, 6 atoms and 6 co-atoms, see https://blogs.ams.org/visualinsight/2015/06/15/lattice-of-partitions/ for a picture.} Then $\GPart(U)=\Part(U_1)\times\ldots\times \Part(U_m)$ is a lattice, whose elements $\Lambda=(\alpha_1,\ldots,\alpha_m)$ are called \emph{generalized partitions} since each $\alpha_i=\{ P_1^i,\ldots,P_{\ell_i}^i\}$ is a partition of $U_i$ (see Figure \ref{f:example-gen-part-lattice} for the picture of the lattice of generalized partitions $\GPart(U)=\Part(\{ 1,2,3\})\times \Part(\{ 1,2\})$). 

There are two alternative interpretations of generalized partitions. First, each generalized partition $\Lambda$ defines a partition $\cB(\Lambda)$ of $U=U_1\times\ldots\times U_m$ into \emph{boxes}, where each box of $\cB(\Lambda)$ is the Cartesian product $P^1_{j_1}\times\ldots\times P^m_{j_m}$ of some blocks $P^i_{j_i}$ of the partitions $\alpha_i\in \Lambda, i=1,\ldots,m$. In general, sets of the form $V=V_1\times\ldots\times V_m$ with $\varnothing\ne V_i\subseteq U_i, i=1\ldots,m$ are called \emph{full-dimensional subproducts} (or  \emph{boxes}) of $U$ and sets $V=V_{i_1}\times\ldots V_{i_k}$ with $\varnothing\ne V_{i_j}\subseteq U_{i_j}, j=1,\ldots,k$ are called \emph{subproducts} of $U$.

Second, if we contract each block $P^i_j$ of the partition $\alpha_i$ of $U_i$ into a single vertex $w^i_j$, then $U_i$ (viewed as a complete graph) will be mapped to $M_i=\{ w^i_1,\ldots,w^i_{\ell_i}\}$. We call the Cartesian product $M(\Lambda)=M_1\times\ldots\times M_m$ a \emph{minor-subproduct} of $U$ associated to the generalized partition $\Lambda$. Notice that the Hamming graph of $M(\Lambda)$ can be obtained from the Hamming graph $H(U)$ of $U$ by contracting each box of $\cB(\Lambda)$ into a single vertex. Notice that this establishes a bijection between the boxes of  $\cB(\Lambda)$ and the vertices of $M(\Lambda)$. We denote the set of all  minor-subproducts of $U$ by $\MP(U)$. We also define the supersets  $\MP^*(U)$ and $\MP^{**}(U)$ of $\MP(U)$ consisting of minor-subproducts of all full-dimensional subproducts $V=V_1\times\ldots\times V_m$ and of  minor-subproducts of all subproducts $V=V_{i_1}\times \ldots\times V_{i_k}$ of $U$, respectively.

The generalized partitions $\Lambda\in \GPart(U)$, their associated box-partitions $\cB(\Lambda)$ of $U$, and their minor-subproducts $M(\Lambda)\in \MP(U)$ are our main ingredients in the implementation of the \emph{shattering$\rightarrow$strong-shattering principle} for subsets $S$ of $U$. We say that a minor-subproduct $M(\Lambda)$ is \emph{shattered} by $S$ if $S$ intersects each box of the box-partition $\cB(\Lambda)$. Selecting an element $v^i_j$ from each block $P_j^i$ of each partition $\alpha_i$ of $\Lambda$, we call the Cartesian product $\{ v^1_1,\ldots,v^1_{\ell_1}\}\times\ldots\times \{ v^m_1,\ldots,v^1_{\ell_m}\}$ 
a \emph{copy} of $M(\Lambda)$. Then we say that the minor-subproduct $M(\Lambda)$ is \emph{strongly shattered} by $S$ if $S$ contains at least one copy of $M(\Lambda)$. Given a set $\mathcal M\subseteq \MP(U)$ (or, more generally, $\mathcal M\subseteq \MP^{**}(U)$) of minor-subproducts of $U$, we say that a set $S\subseteq U$ is \emph{$\mathcal M$-ample} if whenever a minor-subproduct $M(\Lambda)\in {\mathcal M}$ is shattered by $S$, then $M(\Lambda)$ is strongly shattered by $S$. Finally, we call $S\subseteq U$ \emph{ample} if $S$ is $\MP(U)$-ample.  

This unified approach to shattering, strong-shattering, and ampleness is general, but at the same time natural and flexible. In particular, many of the definitions we need arise naturally from the lattice structure of $\GPart(U)$ and $\Part(U_i)$ for $i=1,\ldots,m$.
Namely, we define sets of generalized partitions (and the sets of minor-subproducts corresponding to them) by combining  minima, maxima, and atoms of partition lattices. 
This way we distinguish the following sets of generalized partitions $\Lambda$ and their minor-subproducts:  \emph{trivial} ($\Lambda$ is the minimum of $\GPart(U)$, i.e., all partitions of $\Lambda$ are trivial), \emph{co-trivial} ($\Lambda$ is the maximum of $\GPart(U)$), \emph{mixed} (all partitions of $\Lambda$ are minimal or maximal partitions of factors), \emph{elementary} ($\Lambda$ is an atom of $\GPart(U)$), \emph{one-dimensional} (all partitions are trivial except one), \emph{extended} (all partitions  of $\Lambda$  are quasi-trivial, i.e., all blocks of each partition, except at most one, are trivial).  Such classes of minor-subproducts are used to characterize ampleness or to define operations on ample sets. 

The lattices $\GPart(U)$ and $\Part(U_i), i=1,\ldots, m$ are complemented but not uniquely complemented. In the binary case, each $\Part(U_i)$ is the two-element Boolean lattice, and $\GPart(U)$ is their product (and is again a Boolean lattice). In that case, the complement in $\GPart(U)$ is unique, and related to the set-theoretical complement. Then Lawrence's lopsidedness principle can be rephrased as: for any minor-subproduct $M$, either $M$ is strongly shattered by $S$, or its complement $M^{\diamond}$ in $\GPart(U)$ is strongly shattered by $S^*$. The existence of multiple complements in $\GPart(U)$ in general seems to be an obstacle in defining lopsidedness in Cartesian products, using arbitrary minor-subproducts. However the complements of mixed minor-subproducts are unique and are also mixed minor-subproducts, and this allows us to define a notion of lopsidedness. We say that a set $S\subseteq U$ is \emph{lopsided} if for any mixed minor-subproduct $M$ either $M$ is strongly shattered by $S$ or the unique complement $M^\diamond$ of $M$ is strongly shattered by the complement $S^*=U\setminus S$ of $S$. Clearly, $S$ is lopsided if and only if its complement $S^*$ is lopsided. Finally, we say that a set $S$ is \emph{weakly ample} if $S$ is $\MMP(U)$-ample. 

\subsection{Our results}
Our first result establishes an equivalence between several possible notions of ampleness. It also characterizes lopsidedness and relates weak ampleness and lopsidedness. 

\begin{theorem}\label{thm:intro-amplemidample} For a set $S\subseteq U=U_1\times\ldots\times U_m$, the following conditions are equivalent:
\begin{itemize}
\item[(1)] $S$ is ample;
\item[(2)] $S$ is $\MP^*(U)$-ample;
\item[(3)] $S$ is $\MP^{**}(U)$-ample;
\item[(4)] $S$ is $\EMP(U)$-ample. 
\end{itemize}
A set $S$ is lopsided if and only if  $S$ is weakly ample. Ample sets are lopsided. 
\end{theorem}

This theorem  shows that to deal with ample sets it suffices to consider only minor-subproducts of $\MP(U)$, and, in fact, only 
extended minor-subproducts. 
To draw a parallel between the last characterization (4) of ampleness and the binary case, consider that when shattering or strong-shattering a set $Y\subseteq E$, the elements of $E\setminus Y$ are neglected. This can also be achieved by using the partition $\alpha_i$ consisting of a single block $U_i$ in the relevant dimensions. Additionally, in extended minor-subproducts, the elements of each factor $U_i$ either are: included in singleton blocks of $\alpha_i$, and therefore are fully involved in shattering-strong-shattering; or belong to the unique non-trivial block of $\alpha_i$, and all elements of this block act as a single element. Therefore the elements from such blocks are not neglected but have a weaker impact than the elements from singleton blocks. 
The characterizations (1)-(4) allow us to establish the first structural property of ample sets, namely that ample sets $S$ are \emph{isometric}, i.e.,  any $u,v\in S$ can be connected  in $S$ by a shortest $(u,v)$-path of the Hamming graph $H(U)$. Finally, the last result of the theorem shows that in Cartesian products ampleness and lopsidedness are no longer equivalent. As we see in the following, our notion of ampleness maintains many characterizations from the binary case, suggesting that the notion of lopsidedness may be far too general.  

Shattering and strong-shattering are particular instances of projections and strong-projections. The \emph{projection} $S_M$ of a set $S\subseteq U$ on a minor-subproduct $M=M(\Lambda)$ consists of all $t\in M$ whose corresponding boxes in the box-partition $\cB(\Lambda)$ intersect $S$. Analogously, the \emph{strong-projection} 
$S^M$ consists of all $t\in M$ whose corresponding boxes are included in $S$. To combine several operations of projection and strong-projection, we define $S_M$ and $S^M$ not only for subsets of $U$ but also for sets $S\subseteq M'$, where $M$ and $M'$ are arbitrary minor-subproducts of $U$ (for this we use the meet operation in the lattice  $\GPart(U)$). This allows us to define 
commutativity, superisometricity, and superconnectivity of subsets of Cartesian products, which are the main characterizing features of ampleness in binary products \cite{BaChDrKo}. Our first main characterization of ample sets shows that this is also the case for general ampleness: 

\begin{theorem} \label{thm:intro-maincharacterization}
For a set $S\subseteq U=U_1\times\ldots\times U_m$, the following conditions are equivalent:
\begin{enumerate}
    \item \label{item:ample} $S$ is ample;
    \item \label{item:superisometric} $S^M$ is isometric for all $M\in\MP(U)$ (superisometricity);
    \item \label{item:box-superisometric} $S$ is box-superisometric;
    \item \label{item:superconnected}  $S^M$ is connected for all $M\in\MP(U)$ (superconnectivity);
    \item \label{item:commutative} $(S^M)_{M'}=(S_{M'})^M$ for all $M,M'\in \MP(U)$ with disjoint supports (commutativity); 
    \item \label{item:complement} $S^*$ is ample.
\end{enumerate}
\end{theorem}

Box-superisometricity is a refinement of superisometricity and requires that any pair of parallel boxes $B',B''\subseteq S$ can be connected in $S$ by a geodesic gallery (a shortest path consisting of boxes parallel to $B'$ and $B''$). 

Our second main characterization of ampleness provides the most efficient and economical characterizations of ampleness  and generalizes a similar result of \cite{BaChDrKo}. Instead of considering projections and strong-projections with respect to all minor-subproducts, we consider these operations with respect to elementary minor-subproducts only. Elementary minor-subproducts are exactly the atoms of $\GPart(U)$. They are the generalized partitions $e=(\alpha_1,\ldots,\alpha_m)$, where all partitions $\alpha_i$ are trivial (consist only of singleton blocks), except one partition $\alpha_j$, which consists of singleton blocks and exactly one block $\{ a,b\}$ of size 2. Therefore, we can identify each such generalized partition $e$ with the block $\{ a,b\}$, and denote the projection and the strong-projection by $S_e$ and $S^e$. 

\begin{theorem}\label{thm:intro-efficientcharacterization}
\setcounter{claim}{0} 
For a set $S\subseteq U=U_1\times\ldots\times U_m$,  the following conditions are equivalent:
\begin{enumerate}
    \item \label{item:ample} $S$ is ample;
    \item \label{item:contraction-strongcontr} $S$ is isometric and both $S^e$ and $S_e$ are ample for some 
    elementary minor-subproduct $e$; 
    \item \label{item:strongcontractions} $S$ is connected and $S^e$ is ample for every elementary minor-subproduct $e$;
    \item $S\cap[u,v]$ is ample in $[u,v]$ for all $u,v\in S$.
  \end{enumerate}
\end{theorem}

Condition (4) establishes a strong link between ampleness in Cartesian products and classical ampleness.  Recall that the \emph{interval} $[u,v]$ between $u,v\in U$ consists of all $w\in U$ on shortest $(u,v)$-paths of 
$H(U)$. Each such interval is a hypercube, thus the ampleness of $S\cap [u,v]$ coincides with the classical definition of ampleness. This last characterization also allows to algorithmically recognize if $S\subseteq U$ is ample in time polynomial in the size of $S$ and the number of factors $m$. 

\medskip
In the second part of the paper, we present further properties of ample sets and of their box and prism complexes. For a set $S\subseteq U$, the \emph{box complex} $\BBox(S)$ of $S$ consists of all  boxes $B\in \BBox(U)$ included in $S$.    A box complex  $\BBox(S)$  is  a \emph{bouquet} if there exists a vertex $v_0\in S$ belonging to all maximal boxes of $\BBox(S)$.  Each box $B$ is a Cartesian product of cliques.  
Interpreting each clique as a simplex, each box $B\in \BBox(S)$ is realized by a Cartesian product of simplices, i.e., by a  a prism $\Pi(B)$.  The \emph{prism complex}  $||\BBox(S)||$ of $S$ is obtained by replacing each box $B\in \BBox(S)$ by the prism $\Pi(B)$.  

Our first result in this direction is a characterization of ample sets via push downs. Classical push down (also called shifting or stabilization) is an operation on set families with numerous combinatorial applications \cite{FuPa}. One of them is the elegant proof by Haussler \cite{Ha} of the fundamental lemma of \cite{HaLiWa} (with numerous applications in Machine Learning) that the density of the 1-inclusion graph of any set system $S\subseteq 2^E=\{ 0,1\}^E$ does not exceed its VC-dimension. The \emph{push down}  $S[e\negthickspace\downarrow]$ of  $S\subseteq 2^E$ with respect to  $e\in E$ is obtained from $S$ by replacing  every set $s\in S$ such that $s\setminus \{ e\}\notin S$ by the set $s\setminus \{ e\}$.
From definition,  $S[e\negthickspace\downarrow]$ and $S$ have the same size.  Furthermore, if $E=\{ e_1,\ldots,e_m\}$, then after performing 
the push downs of $S$ with respect to  $e_1,\ldots,e_m$, the set $S[e_1,\ldots e_m\negthickspace\downarrow]=S[e_1\negthickspace\downarrow][e_2\negthickspace\downarrow] \ldots [e_m\negthickspace\downarrow]$ will no longer change if further push downs are applied. The cube complex of  $S[e_1,\ldots e_m\negthickspace\downarrow]$ is a bouquet of cubes with origin at $(0,\ldots,0)$.  Performing serial push downs of $S$ along two different permutations of the same set of elements  may lead to  two different results.  This does not happen if $S$ is ample: a set  $S\subseteq \{ 0,1\}^E$  is ample if and only if 
all series of push downs commute, i.e., $S[e_{i_1},\ldots e_{i_k}\negthickspace\downarrow]=S[e_1,\ldots e_k\negthickspace\downarrow]$  \cite{BaChDrKo}. In this case, the bouquet of cubes 
$S[e_1,\ldots e_m\negthickspace\downarrow]$ coincides with the simplicial complex $\overline{X}(S)=\underline{X}(S)$. 

In this paper, we define a notion of push down for subsets of Cartesian products based on one-dimensional minor-subproducts.  Let $S\subseteq U=U_1\times\ldots\times U_m$ and assume that $U_i=\{0,1,\ldots,|U_i|-1\}$ for each $i$. Let $M=M(\Lambda)$ be a one-dimensional minor-subproduct of $U$. Then every box $B\in \cB(\Lambda,U)$ is a clique whose elements only differ in their $i$-coordinate. The \emph{push down} operation $S[M\negthickspace\downarrow]$  is defined boxwise: for each $B\in \cB(\Lambda,U)$, $S[M\negthickspace\downarrow ]\cap B$ is equal to the $|S\cap B|$ elements of $B$ with the smallest $i$-coordinates.  Informally, the push down operation shifts the elements of $S$ as far down as possible, subject to the condition that the elements of every box $B$ of $\cB(\Lambda,U)$ stay within $B$.\footnote{If the one-dimensional minor-subproduct $M$ is mixed, then 
every box $B\in \cB(\Lambda,U)$ is a clique  defined by some factor $U_i$. In this case, our definition of push down  coincides with the definition of shifting used in \cite{BrCaDiMoYe}.}

\begin{theorem} \label{intro-thm:intro-pushdown} A set $S\subseteq U$ is ample if and only if every 
serial push down commutes on $S$, i.e., for any  sequence of one-dimensional minor-subproducts $M^1,M^2,\ldots, M^k$ with pairwise distinct supports and any permutation $M^{i_1},M^{i_2},\ldots, M^{i_k}$ 
we have $S[M^1,M^2,\ldots, M^k\negthickspace\downarrow]=S[M^{i_1},M^{i_2},\ldots, M^{i_k}\negthickspace\downarrow]$.
\end{theorem}

The \emph{dimension} of a box $B=B_1\times\ldots\times B_m$  is $\dim(B)=\sum_{i=1}^m (|B_i|-1)$, which is  the topological dimension of the prism $\Pi(B)$. Let $f_i(S)$ denotes the number of faces of $\BBox(S)$ of dimension $i$. As usually, the vector $f(S)=(f_0(S),f_1(S),\ldots,)$ is called the \emph{$f$-vector} of  $\BBox(S)$ (and of $S$). 
The \emph{Euler characteristic} of  box-complex  $\BBox(S)$ (or of the prism complex  box-complex  $||\BBox(S)||$ of $S\subseteq U$ is $\chi(S)=\sum_{i=0}^{\infty}(-1)^if_i(S)$. 

 We show that the push down operation preserves the $f$-vectors of ample sets, and that by applying repeated push down operations we end up with a bouquet of prisms. This allows us to prove the following result, generalizing the analogous result for the binary case of \cite{Wi} and \cite{BaChDrKo}:

\begin{theorem} \label{Euler-characteristic}  A set $S\subseteq U$ is ample if and only if $\chi(S\cap V)=1$ for every full-dimensional subproduct $V$ with $S\cap V\neq \varnothing$.
\end{theorem}

One of the main results of the paper is a decomposition theorem for ample sets and their box/prism complexes. It asserts that any ample set $S\subseteq U$ can be obtained from its maximal boxes by successive ample amalgams. Roughly speaking, an \emph{ample amalgam} is an operation of gluing  together two ample sets $S_1$ and $S_2$ along an ample subset $S_0=S_1\cap S_2$ (see Definition \ref{def:amalgam}). Then each maximal box of $\BBox(S)$ is either a maximal box of $\BBox(S_1)$, or a maximal box of $\BBox(S_2)$, or a maximal box of both, in which case it is a maximal box of $\BBox(S_0)$. To prove the decomposition theorem, we establish and use the ampleness of several types of sets, which are defined by a factor $U_i$ and an element $a\in U_i$, and have a clear geometric meaning (see Definition \ref{def:sector}): sectors, cosectors, extended  sectors, boundaries and neighborhoods of sectors, hyperplanes. This result may be useful in other settings and provides strong geometric and recursive properties of ample sets. A consequence of the decomposition theorem  is that prism complexes of ample sets are contractible topological spaces. This shows that ample sets have not only strong combinatorial, metric, and geometric properties, but also important topological properties. These results can be summarized as follows: 

\begin{theorem}\label{intro:ample-amalgams-pseudo-boxes} For each ample set 
$S\subseteq U$, the following holds:
\begin{itemize}
    \item[(1)] $S$ can be obtained from the set of its maximal boxes  by a sequence of ample amalgams;
    \item[(2)]  the prism complex  $||\BBox(S)||$  of $S$ is contractible. 
\end{itemize}
\end{theorem}

We already mentioned several important examples 
of classical ample sets. 
Most of those examples  arise as set systems. Each ample set of $2^E$ is an ample set in our sense 
due to the bijection between the Boolean cube $2^E$ and the binary Cartesian product $\{ 0,1\}^E$. In this paper, we present several natural examples of ample sets in Cartesian products that are not binary. The first new example of ample sets of Cartesian products is that of \emph{quasi-median graphs}. These graphs are the Hamming analogs of median graphs (in particular, they are the retracts of the Hamming graphs \cite{Wi-qm}, while median graphs are the retracts of hypercubes). They have been introduced by Mulder  \cite{Mu} and investigated in numerous papers, in particular in \cite{BaMuWi}.
Median graphs are important in geometric group theory because they are exactly the 1-skeleta of CAT(0) cube complexes. Similarly, quasi-median graphs also found applications in this theory \cite{Ge}. 
Secondly, we present a class of polyhedra that we call prism-like polyhedra, as their combinatorial structure resembles that of a partial prism. We show that the vertex sets of these polyhedra can be represented by an ample set of a Cartesian product.

Our principal example concerns payoff games in graphs and actually initiated our work on ample sets in Cartesian products. 
Games on graphs are played on a weighted directed graph $G=(V_{\max}\cup V_{\min},E,w)$, by one or two players, where a pebble is moved around the graph and players collect rewards. It turns out that winning strategies can be represented by a vector $\sigma\in N^+(v_1)\times N^+(v_2)\times\ldots\times N^+(v_m)$, where $\{v_1,\ldots,v_m\}$ are the player-controlled vertices, and $N^+(v_i)$ is the set of successors of $v_i$. The second author proved in \cite{maat_strategy_2025} that, if $|N^+(v_i)|\leq 2$ for $i=1,2,\ldots, m$, then the set of winning strategies $\Sigma$ is a binary ample set for multiple classes of games on graphs. We extend this result by removing the degree constraint. In particular, we show that for any mean payoff game, $\Sigma$ is an ample set of $N^+(v_1)\times N^+(v_2)\times\ldots\times N^+(v_m)$. 

We conclude the paper with a unified treatment of various notions of shattering and VC-dimension, occurring in the literature on multiclass learning. We interpret these in terms of shattering and  dimension of minor-subproducts of specific types.

\subsection{Organization} 
We provide the necessary background in \cref{sec:preliminaries}. Then in \cref{sec:shatterstrongshatter} we provide formal definitions of our main new concepts, minor-subproduct, box-partition, and generalized partition, and we generalize the definitions of shattering, strong-shattering, ampleness and lopsidedness to Cartesian products. After proving the main properties of minor-subproducts in \cref{section:propertiesminorsubproducts}, we prove \cref{thm:intro-amplemidample} in \cref{sec:ample-extample}. We discuss projection, strong-projections, and restrictions in \cref{sec:projections}, which we use for our main characterizations (\cref{thm:intro-maincharacterization,thm:intro-efficientcharacterization}) in  \cref{sec:maincharacterization}.  Then in \cref{sec:pushdowns} we explore the relation between ampleness and push down operations, and in \cref{sec:boxcomplexes} we consider the box complexes of ample sets. Finally, we present some classes of ample sets in \cref{sec:examples}, and discuss the relation of minor-subproducts to VC-dimension in \cref{sec:VC-dimension}. To illustrate the main new notions, 
we use a running example of an ample set and we also relate them with analogous notions in the binary case. 

\section{Preliminaries}\label{sec:preliminaries}

\subsection{Cartesian products of sets} 
The \emph{Cartesian product} of $m$ finite \emph{nonempty} sets $U_1,\ldots,U_m$ is the set $U:=U_1\times\ldots\times U_m$ of all tuples $(u_1,\ldots,u_m)$ with  $u_{i} \in U_i$ for $i=1,\ldots,m$. 
A factor $U_i$ which is a singleton is called \emph{trivial}. If all factors $U_1,\ldots, U_m$ have size $2$, then $U$ is called a \emph{binary product}. If $U_1=\ldots=U_m=Y$ and $X=\{1,\ldots,m\}$, then $U$ coincides with the set $Y^X$ of all maps $f:X\rightarrow Y$. 
%
For a set $A\subseteq X=\{1,\ldots,m\}$ with $A=\{ i_1,\ldots, i_k\}$,  an $A$-\emph{tuple}  is any element $t=(t_{i_1},\ldots,t_{i_k})$ of the product $U_{i_1}\times\ldots\times U_{i_k}$. For a tuple $u=(u_1,\ldots,u_m)\in U$ and $A\subseteq X$, $u|_A$ denotes the \emph{trace} (or the \emph{restriction}) of $u$ on $A$: $u|_A$ is the $A$-tuple whose coordinates are the $u_i$ with $i\in A$. Conversely, if $t$ is an $A$-tuple, then any tuple $u\in U$ such that $u|_A=t$ is called an \emph{extension} of $t$. The set of all extensions of $t$ is called the \emph{fiber} of $t$ and is denoted by $F(t)$.

\begin{definition} [Subproducts]  Given a Cartesian product $U=U_1\times\ldots\times U_m$ with $X=\{ 1,\ldots,m\}$ and some nonempty sets $V_{i_j}\subseteq U_{i_j}$ for $j=1,\ldots,k$. If $A=\{ i_1,\ldots,i_k\}$, then the \emph{subproduct} $V$ of $U$ defined by the sets $V_{i_1},\ldots,V_{i_k}$ is the set of all $A$-tuples $v=(v_{i_1},\ldots,v_{i_k})$ such that $v_{i_j}\in V_{i_j}$.  
Then $V_{i_1},\ldots,V_{i_k}$ are called the \emph{subfactors} of the subproduct $V=V_{i_1}\times\ldots\times V_{i_k}$ and $A\subseteq X$ is called the \emph{support} of $V$ and is denoted by $\supp(V)$. 
A subproduct  $V$ is called \emph{full-dimensional} (or a \emph{box}) if $\supp(V)=X$. 
The \emph{dimension} $\dim(V)$ of a subproduct $V=V_{i_1}\times\ldots\times V_{i_k}$ is the sum $\sum_{j=1}^k (|V_{i_j}|-1)$.
\end{definition}

A subproduct $V$ is not always a subset of the product $U$. 
In fact, $V$ is a subset of $U$ if and only if $V$ is  full-dimensional. However, $U$ hosts copies of each subproduct $V$: 

\begin{definition} [Copies of subproducts] Let $A$ and $B$ be two disjoint sets of $X$, say $A=\{i_1,\ldots, i_k\}$
and $B=\{j_1,\ldots, j_\ell\}$. Let $t=(t_{i_1},\ldots, t_{i_k})$ be an $A$-tuple and $s=(s_{j_1},\ldots, s_{j_{m-k}})$ be a $B$-tuple. 
We denote by $s\times t$ the $(A\cup B)$-tuple $u=(u_{i_1},\ldots,u_{i_k})$ such that $u_i=t_i$ if $i\in A$ and $u_i=s_i$ if $i\in B$.  
If $V$ is an $A$-subproduct and $t$ is a $(X\setminus A)$-tuple, then we call the full-dimensional subproduct $V\times t\subseteq \Prod(X)$ a \emph{copy of} $V$ in $\Prod(X)$. 
If $V$ is full-dimensional, then $V$ has a unique copy in $\Prod(X)$, which is $V$ itself. 
\end{definition}

\subsection{Partition lattice} \label{partitionlattice} (following \cite[Chapter IV]{Gr}) A \emph{partition} of a set $Z$ with $n$ elements is a set  $\alpha=\{ P_1,\ldots,P_k\}$ of nonempty pairwise disjoint subsets of $Z$ whose union is $Z$.  The members $P_1,\ldots, P_k$ of $\alpha$  are called the \emph{blocks} of $\alpha$. A singleton as a block is called \emph{trivial}. A \emph{trivial partition} is a  partition whose all blocks are trivial. Analogously, a \emph{co-trivial partition} is the partition with a single block $Z$. There is a one-to-one correspondence between the partitions of $Z$ and the equivalence relations on $Z$. The set of all partitions of $Z$ is denoted by  $\Part(Z)$. 

Given two partitions $\alpha,\beta\in \Part(Z)$, $\alpha$ is a \emph{refinement} of  $\beta$ (in which case  we say that $\alpha$ is \emph{finer} than $\beta$ and that $\beta$ is \emph{coarser} than $\alpha$) 
if every block of $\alpha$ is a subset of some block of $\beta$. In that case, we write $\alpha\preceq \beta$. Equivalently, the blocks of $\beta$ are unions of blocks of $\alpha$. $\Part(Z)$ with the partial order $\preceq$  forms a complete lattice, 
which is called the \emph{partition lattice} of $Z$.  The meet and the join of partitions $\alpha$ and $\beta$ in the partition lattice $\Part(Z)$ are defined as follows. The \emph{meet} $\alpha\wedge \beta$ is the partition $\gamma$ whose blocks are the intersections of a block of $\alpha$ and a block of $\beta$ that are not disjoint from each other. The \emph{join} $\alpha\vee \beta$ is the partition $\rho$  whose blocks are the equivalence classes of the transitive closure $\sim^*$ of the following binary relation  $\sim$ on the blocks of $\alpha=\{ P_1,\ldots,P_k\}$ and $\beta=\{ Q_1,\ldots,Q_{\ell}\}$: $P_i\sim Q_j$ 
if $P_i\cap Q_j\ne\varnothing$. The \emph{atoms} of the partition lattice  $\Part(Z)$  are the partitions with $n-2$ singleton blocks and one block with two elements.  The minimal element of the lattice $\Part(Z)$ is the trivial partition $\{ \{ z\}: z\in Z\}$ and the maximal element of $\Part (Z)$ is the co-trivial partition $\{ Z\}$ with a single block. 
We say that a partition  $\alpha$ is \emph{quasi-trivial} if  $\alpha$  contains at most one non-trivial block. The \emph{restriction} of a partition $\alpha=\{ P_1,\ldots,P_k\}$ 
of $Z$ on a subset $Z'$ of $Z$ is the partition $\alpha'$ of $Z'$ whose blocks are the non-empty intersections $P_i\cap Z'$.  
Last, notice that $\Part(Z)$ is a \emph{complemented lattice} \cite{Gr}, i.e., for
each element $\alpha\in \Part(Z)$ there exists an   element $\beta\in \Part(Z)$ (called a \emph{complement} of $\alpha$) such that the meet $\alpha\wedge \beta$ is the minimal partition $\alpha_{\bot}=\{ \{ z\}: z\in Z\}$ and the join $\alpha\vee\beta$ is the maximal partition $\alpha^{\top}=\{ Z\}$ with a single block. In general, $\alpha$ may have many complements. We will denote any complement of $\alpha$ by $\alpha^{\diamond}$.

\subsection{Cartesian products of graphs and minors} In this subsection, we  endow $U=U_1\times\ldots\times U_m$ with the structure of a Hamming graph (Cartesian product of cliques). We also present some notions related to graphs. 

\subsubsection{Basic notions about graphs}
Let $G=(V,E)$ be a finite simple undirected graph with the set of vertices $V$ and the set of edges $E$. For two distinct
vertices $v,w$ of $G=(V,E)$  we write $v\sim w$ when there is an edge
connecting $v$ with
$w$. 
The subgraph of $G$ \emph{induced by} a subset $S\subseteq V$ is the graph
$G[S]=(S,E')$ such that $uv\in E'$ if and only if $uv\in E$.
The \emph{distance}
$d(u,v)=d_G(u,v)$ between two vertices $u$ and $v$ of a connected graph $G$ is the
length of a shortest $(u,v)$--path (measured by number of edges). The \emph{interval}
$[u,v]$ between $u$ and $v$ of $G$ consists of all vertices on shortest
$(u,v)$--paths, that is, of all vertices (metrically) \emph{between} $u$
and $v$: $[u,v]=\{ x\in V: d(u,x)+d(x,v)=d(u,v)\}$.  
A set $S\subseteq V$ (or the subgraph $G[S]$ induced by $S$)  is called \emph{isometric} if $d_{G[S]}(u,v)=d_G(u,v)$ for any
 $u,v\in S$ and  \emph{weakly isometric} if $d_{G[S]}(u,v)=d_G(u,v)$ for any
 $u,v\in S$ with $d_G(u,v)=2$. Equivalently, $S$ is weakly isometric if $S\cap [x,y]\ne\{x,y\}$ for all $x,y\in S$ 
 with $d_G(x,y)=2$. Finally, a set $S$ (or the subgraph $G[S]$) is called \emph{convex}
if $[u,v]\subseteq S$ for any $u,v\in S$. A \emph{halfspace} of $G$ is a convex set $S$ with a convex complement $V\setminus S$. An \emph{isometric embedding} of a connected graph $G=(V,E)$ into a connected graph $G'=(V',E')$ is a map $\varphi: V\rightarrow V'$ such that $d_{G'}(\varphi(u),\varphi(v))=d_G(u,v)$ for any two vertices $u,v\in V$. The image $\varphi(G)$ of $G$ is an isometric subgraph of $G$. 

\subsubsection{Cartesian product of graphs}
Let $G_{1}=(U_1,E_1),\ldots,G_{m}=(U_m,E_m)$ be a set of $m$ graphs.  The \emph{Cartesian product} $G:=\prod_{i=1}^m G_i=G_1\times\ldots\times G_m$ is a graph defined on the set $U=U_1\times\ldots\times U_m$ of all tuples $(x_1,\ldots,x_m)$, $x_{i} \in U_i$, where two vertices $x=(x_1,\ldots,x_m)$ and $y=(y_1,\ldots,y_m)$ are adjacent in $G$ if and only if there exists an index $j$ such that $x_{j} y_{j} \in E(G_{j})$ and $x_{i} = y_{i}$ for all $i\neq j$.
A \emph{subproduct} $G'$ of a Cartesian  product $G=\prod_{i=1}^m
G_i$ is a product
such that $G'=\prod_{j=1}^k G'_{i_j}$, where each $G'_{i_j}$ is a nonempty subgraph of $G_{i_j}$.

\subsubsection{Minors in graphs} We will consider the following slightly modified version of the notion of minor in graphs. 
A graph $M$ is called a \emph{minor} of a graph $G$ if there exists a
partition of vertices  of $G$ into connected subgraphs ${\mathcal P}=\{
P_1,\ldots,P_t\}$ and a
bijection $f: V(M)\rightarrow {\mathcal P}$ such that if $uv\in E(M)$ then
there exists an edge of
$G$ running between the subgraphs $f(u)$ and $f(v)$ of $\mathcal P$, i.e.,
after contracting each subgraph $P_i\in {\mathcal P}$ into a single vertex  we
will obtain a graph containing $M$ as a spanning subgraph. An \emph{induced minor} of $G$ is a minor of an induced subgraph $G'$ of $G$. 

In case of Cartesian products of graphs, the following definition of minor takes into consideration the product structure:

\begin{definition}\label{minor-subproduct} [Minor-subproduct] \cite{ChLaRa} A Cartesian product $M=M_1\times\ldots\times M_m$ of graphs $M_1,\ldots,M_m$ is called a \emph{minor-subproduct} of a Cartesian product $G:=\prod_{i=1}^m G_i$ if $M_{i}$ is a minor of $G_{i}$ for each $i=1,\ldots,m$.
\end{definition}
However, in this paper we use an adapted version of this notion using partitions. This is worked out in more detail in \cref{sec:minor-subproducts}.


\subsubsection{Hamming graphs} A \emph{Hamming graph}  is the Cartesian product $K_{n_1}\times\ldots\times K_{n_m}$  of $m$ complete graphs $K_{n_1},\ldots, K_{n_m}$. If $U_1,\ldots,U_m$ are the vertex-sets of the factors $K_{n_1},\ldots, K_{n_m}$ and $U=U_1\times\ldots\times U_m$, then we denote this Hamming graph by $H(U)$ and call it the \emph{Hamming graph} of the product $U$. 
The $m$-\emph{dimensional hypercube} $Q_m$ is the Cartesian product of $m$ copies of $K_2$: $Q_m=K_2\times \ldots \times K_2$. Equivalently, $Q_m$ is the Boolean cube defined by a set $X$ of size $m$: the elements $Q_m$ are vertices of the Boolean cube, and they can also be viewed as subsets of $X$ by considering them indicator vectors of sets: then two sets $A,B\subseteq X$ are adjacent in $Q_m$ if and only if $|A\Delta B|=1$. We refer to a 2-dimensional cube as a square. A sequence $(x_1,x_2,x_3,x_4)$ of vertices of a graph 
$G$ is called a \emph{square} if the vertices $x_1,x_2,x_3,x_4$ induce a square of $G$, where $x_i,x_j$ are adjacent if and only if $i$ and $j$ differ by 1 (modulo 4). 

The graph-distance  $d(x,y)$ (or $d_{H(U)}(x,y)$ if we need to avoid an ambiguity) between two vertices  $x=(x_1,\ldots,x_m)$ and $y=(y_1,\ldots,y_m)$ of a Hamming graph $H(U)=K_{n_1}\times\ldots\times K_{n_m}$  coincides with the \emph{Hamming distance} between the $m$-tuples $(x_1,\ldots,x_m)$ and $(y_1,\ldots,y_m)$, and thus is equal to the number of indices $i$ such that $x_i\ne y_i$. 
The intervals $[x,y]$ in Hamming graphs $H(U)$ have a special form: if $x=(x_1,\ldots,x_m)$ and $y=(y_1,\ldots,y_m)$, then $[x,y]$ is the binary subproduct $\{x_1,y_1\}\times\ldots\times\{ x_m,y_m\}$, where each pair $\{x_i,y_i\}$ is an edge or a single vertex. Thus $[x,y]$ is isomorphic to the $k$-hypercube, where $k$ is the distance between $x$ and $y$ in $H$. 
We will call a set $S\subseteq U=U_1\times \ldots\times U_m$ \emph{isometric}, \emph{weakly isometric},  or \emph{convex} if $S$ induces an isometric, weakly isometric, or convex subgraph  of the Hamming graph $H(U)$. For a set $S\subseteq U$, we denote by $H(S)$ the subgraph of 
$H(U)$ induced by $S$ (in literature, $H(S)$ is often called the \emph{1-inclusion graph} of $S$). An example of an isometric set is shown in \cref{fig:runningexample}.

A graph $G$ is called a \emph{partial cube} if $G$ admits an isometric embedding into a hypercube. Analogously, a graph $G$ is called a \emph{partial Hamming graph} if $G$ admits an isometric embedding into a Hamming graph. For an edge $uv$ of a graph $G=(V,E)$ the vertex-set $V$ can be partitioned into the following three sets 
$W(u,v)=\{ x\in V: d(x,u)<d(x,v)\}, W(v,u)=\{ x\in V: d(x,v)<d(x,u),$ and $W_=(u,v)=\{ x\in V: d(x,u)=d(x,v)$. 
Notice that for all edges $uv$ the set $W_=(u,v)$ is empty of and only if $G$ is bipartite (i.e., all cycles of $G$ are even). Partial hypercubes have been nicely characterized by Djokovi\'c \cite{Dj} and his characterization was generalized in \cite{Ch_Hamming} to partial Hamming graphs: 

\begin{theorem} \label{thm:partial-Hamming} (1) \cite{Dj} A connected graph $G$ is a partial cube if and only if $G$ is bipartite and for any edge $uv$ the sets $W(u,v)$ and $W(v,u)$ are complementary halfspaces;

\medskip\noindent
(2) \cite{Ch_Hamming} A connected graph $G$ is a partial Hamming graph if and only if for any edge $uv$ the sets $W(u,v),W(v,u),$ and $W_=(u,v)$ are complementary halfspaces, i.e., $W(u,v),W(v,u),W_=(u,v)$ and their pairwise unions are convex. 
\end{theorem}

In case of a Hamming graph $H(U)$ and an edge $e=uv$ of $H(U)$, we will denote the halfspaces 
$W(u,v)\cup W(v,u)$ and $W_=(u,v)$ by $W(e)$ and $W_=(e)$, respectively. If $u,v\in U_i$, then 
$W(e)$ coincides with the full-dimensional subproduct $U_1\times\ldots\times U_{i-1}\times \{ u,v\}\times U_{i+1}\times\ldots\times U_m$ and $W_=(e)$ coincides with the full-dimensional subproduct $U_1\times\ldots\times U_{i-1}\times (U_i\setminus\{ u,v\})\times U_{i+1}\times\ldots\times U_m$. Analogously, if $S$ is an isometric subset of $U$  (and thus $H(S)$ is a partial Hamming graph), we denote by $S(e)$ the halfspace $W(u,v)\cup W(v,u)$ and by $S_=(e)$ the halfspace $W_=(u,v)$ of $H(S)$. Notice that $S(e)=S\cap W(e)$ and $S_=(e)=S\cap W_=(e)$.

\begin{figure}[h]
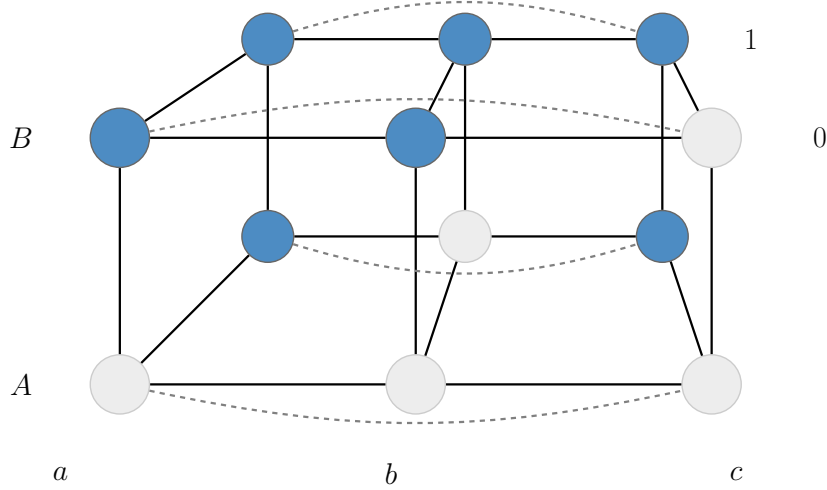

    \centering
    \includestandalone[width=0.7\linewidth]{img/runningexample}
    \caption{The Hamming graph corresponding to the Cartesian product $\{a,b,c\}\times \{A,B\}\times \{0,1\}$. Coordinates are shown next to the graph for clarity. The seven colored vertices induce an isometric subgraph. We will use this set as a \bf{running example}. 
    }
    \label{fig:runningexample}
\end{figure}

\section{Shattering and strong-shattering in Cartesian products}\label{sec:shatterstrongshatter}
The main goal of this section is to define a notion of shattering, which generalizes the classical notion of shattering and is appropriate for Cartesian products of 
sets. In the existing definitions of shattering in binary products $\{ a,b\}^X$ or in 
products $Y^X$, used in various definitions of VC-dimension, a subset $S$ of a product always shatters a subset of coordinates $Y\subseteq X$. 
In this paper, we consider a more fine-grained notion of shattering by considering shattering not only of subproducts of a Cartesian product $U=U_1\times\ldots\times U_m$, but also of the minor-subproducts of $U$, when viewed as a Hamming graph.
For this, we adapt and refine \cite[Definition 2]{ChLaRa} of shattering the minor-products in Cartesian products of graphs to our setting of Hamming graphs. Roughly speaking, minor-subproducts of Hamming graphs are obtained by partitioning each $U_i$ into a collection of non-empty sets, contracting each such set into a single vertex, resulting into the contraction of the clique $U_i$ to a smaller clique $U^+_i$, and taking the Cartesian products of the resulting $U^+_i$s. Similarly to extensions and fibers for tuples in subproducts of $U_1\times\ldots\times U_m$, we define extensions and fibers for tuples in minor-subproducts. Using this, the definition of shattering minor-subproducts becomes similar to the classical one.

\subsection{Minor-subproducts}\label{sec:minor-subproducts}
We now adapt the notion of minor-subproduct from \cref{minor-subproduct} to our needs, using partitions. 
Let $U=U_1\times \ldots\times U_m$, where all $U_i$ are nonempty. For each factor $U_i$ we denote by  $\Part(U_i)$ the lattice of partitions of the sets $U_i$ endowed with the partial order $\preceq_i$ 
and the lattice operations $\wedge_i$ and $\vee_i$ (where $\preceq_i$, $\wedge_i$ and $\vee_i$ correspond to $\preceq$, $\wedge,$ and $\vee$ in the definition of the partition lattice from Subsection \ref{partitionlattice}). 

While the minors of graphs in general can take many forms, the minors and the induced minors of complete graphs are also complete subgraphs. 
Namely, each minor $M_i$ of the complete graph $G_i$ with the vertex-set $U_i$ can be viewed as a partition $\alpha_i=\{ P^i_1,\ldots,P^i_k\}\in \Part(U_i)$ and a contraction of  each block $P^i_j$ of $\alpha_i$ into a single vertex $w^i_j$. Then  $M_i=\{ w^i_1,\ldots,w^i_\ell\}$ is called a \emph{minor} of the set $U_i$ defined by the
partition  $\alpha_i$  of  $U_i$. Analogously, each induced minor of $G_i$ can be viewed as a partition $\alpha'_i=\{ Q^i_1,\ldots,Q^i_k\}\in \Part(V_i)$ of a nonempty subset $V_i$ of $U_i$ and the contraction of each block of $\alpha'_i$ into a single vertex. Implicitly, we remember which vertices of the original graph are contracted to which vertex of the minor. 
Therefore, the  minors of $U_i$, viewed as a complete graph $G_i$, are in bijection with the elements of the partition lattice $\Part(U_i)$. We extend this bijection to a bijection between minor-subproducts of Hamming graphs and generalized partitions, which we define next. 

\begin{definition} [Lattice of generalized partitions] For a Cartesian product  $U=U_1\times\ldots\times U_m$, let $\GPart(U)=\GPart(U_1\times\ldots\times U_m)=\Part(U_1)\times\ldots\times \Part(U_m)$ be the lattice which is the direct product of the lattices $\Part(U_i), i=1,\ldots,m$. Then $\GPart(U)$ has  the $m$-tuples $\Lambda=(\alpha_1,\ldots,\alpha_m)$ with $\alpha_i\in \Part (U_i), i\in X$, as elements, which we call \emph{generalized partitions}. For two  generalized partitions $\Lambda=(\alpha_1,\ldots,\alpha_m)$ and $\Upsilon=(\beta_1,\ldots,\beta_m)$, we set $\Lambda\preceq \Upsilon$ if $\alpha_i\preceq_i \beta_i$ for $i=1,\ldots,m$, $\Lambda\wedge \Upsilon=(\alpha_1\wedge_1 \beta_1,\ldots,\alpha_m\wedge_m \beta_m)$,  and  $\Lambda\vee \Upsilon=(\alpha_1\vee_1 \beta_1,\ldots,\alpha_m\vee_m \beta_m)$. We denote by $\Lambda_{\bot}$ and $\Lambda^{\top}$ the minimal and the maximal elements of the lattice $\GPart(U)$. See \cref{f:example-gen-part-lattice} for an example.
\end{definition}

\begin{figure}[htbp]
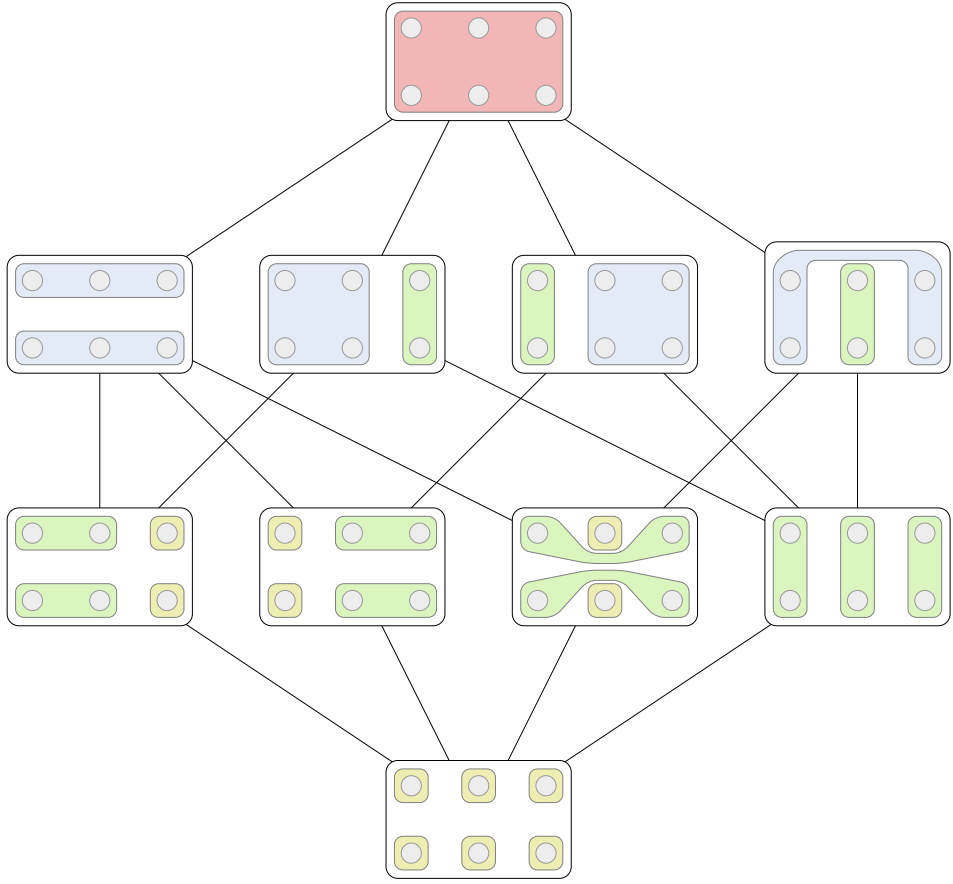

    \centering
    \includestandalone[width=0.8\linewidth]{img/partitionlattice}
    \caption{Hasse diagram of the lattice of generalized partitions of $U=U_1\times U_2$, where $|U_1|=3$ and $|U_2|=2$. Its elements are represented by the corresponding box-partitions of $U$, see Definition \ref{d:box-partition}.}
    \label{f:example-gen-part-lattice}
\end{figure}

 Since the lattice $\GPart(U)$ is the direct product of complemented lattices $\Part(u_i)$, $\GPart(U)$ is also complemented. 
We will denote by $\Lambda^{\diamond}$ any complement of $\Lambda\in \GPart(U)$.  If $\Lambda=(\alpha_1,\ldots,\alpha_m)$, then $\Lambda^{\diamond}=(\alpha_1^{\diamond},\ldots,\alpha_m^{\diamond})$, where $\alpha_i^{\diamond}$ is a complement of the partition $\alpha_i$ in $\Part(U_i), i\in X$. The atoms of $\GPart(U)$ are the generalized partitions $\Lambda=(\alpha_1,\ldots,\alpha_m)$ in which all blocks of all partitions are trivial, except that one block in one partition has size 2. 
Analogously, the co-atoms of $\GPart(U)$ are  the generalized partitions $\Lambda=(\alpha_1,\ldots,\alpha_m)$ in which all partitions except one are co-trivial 
and one partition has two blocks. The \emph{support} of $\Lambda=(\alpha_1,\ldots,\alpha_m)\in \GPart(U)$ is the set of indices $i$ such that the partition $\alpha_i$ is not trivial. 

\begin{lemma}\label{lem:atomistic}
    The lattice $\GPart(U)$ is atomistic and co-atomistic, i.e. every element can be written as the join of a finite number of atoms and as the meet of a finite number of co-atoms. Moreover, if $A\cup B=\supp(\Lambda)$ with $A\cap B=\varnothing$, then there exist $\Lambda',\Lambda''\in \GPart(U)$ with $\Lambda'\vee \Lambda''=\Lambda$, $\supp(\Lambda')=A$ and $\supp(\Lambda'')=B$.
\end{lemma}
\begin{proof}
    Let $\Lambda=(\alpha_1,\ldots,\alpha_m)\in \GPart(U)$, with $\alpha_i=(P_1^i,\ldots,P_{\ell_i}^i)$. Consider any $P_j^i=\{p_1,p_2,\ldots,p_{k}\}$ with $|P_j^i|\geq 2$, and let $\Lambda_{j}^i$ be the generalized partition in which every block is trivial, except that the $i$-partition has a block $P_j^i$. For $t=1,2,\ldots,k-1$, let $\Lambda_t'$ be an atom of $\GPart(U)$ whose only non-trivial block is $\{p_1,p_{t+1}\}$. Clearly $\Lambda_j^i=\bigvee_{t=1}^{k-1}\Lambda_{t}'$. Since we can write any $\Lambda_j^i$ as a join of atoms in this manner and since $\Lambda=\bigvee_{i,j:j\leq \ell_i} \Lambda_j^i$, it follows that $\Lambda$ can be written as a join of atoms. For the last part of the lemma, take $\Lambda'=\bigvee_{i\in A,j\leq \ell_i}\Lambda_j^i$ and $\Lambda''=\bigvee_{i\in B,j\leq \ell_i}\Lambda_j^i$. Then clearly $\Lambda'\vee \Lambda''=\Lambda$, $\supp(\Lambda')=A$ and $\supp(\Lambda'')=B$. Finally, to prove the lattice is co-atomistic, for any $\Lambda$ we take $\Lambda_{j}^i$ to be the co-atom that has partition $\{P_j^i,U_i\backslash P_j^i\}$ in its $i$-th component, and then clearly $\Lambda=\bigwedge_{i,j:j\leq \ell_i}\Lambda_j^i$.
\end{proof}

\begin{definition}\label{minor-subproduct-revised} [Minor-subproduct revised] Let  $U=U_1\times\ldots\times U_m$ and $\Lambda=(\alpha_{1},\ldots,\alpha_m)$ be an element of the lattice $\GPart(U)$ with $\alpha_i=\{ P^i_{1},\ldots,P^i_{\ell_i}\}\in \Part(U_i)$. Then a full-dimensional Cartesian product $M:=M(\Lambda,U)=M_1\times\ldots\times M_m$ is called a \emph{minor-subproduct} of $U$ defined by  $\Lambda$ if  $M_i=\{ w^i_1,\ldots,w^i_{\ell_i}\}$, where  $w^i_j$ is a new element to which the block $P^i_j$ of $\alpha_i$ is contracted. When the set $U$ is clear from context, we also use the shorthand notation $M=M(\Lambda)$. The \emph{support} $\supp(M)$ of  $M$ consists of all $i\in X$ such that $\alpha_i$  is a non-trivial partition of $U_i$ (i.e. the support of the underlying generalized partition $\Lambda$). Let $\MP(U)$ denote the set of all minor-subproducts $M=M(\Lambda,U)$ with $\Lambda\in \GPart(U)$. 
\end{definition}

\begin{remark}
Note that, whenever we consider a minor-subproduct, we implicitly remember which elements were contracted to each element of $M$. To this end, we sometimes refer to elements of minor-subproducts by their boxes, and to coordinates of minor-subproducts by their corresponding blocks.
\end{remark}

For $M,M'\in \MP(U)$ with $M=M(\Lambda,U)$ and 
$M'=M(\Upsilon,U)$ we will set $M\preceq M'$ if $\Lambda\preceq \Upsilon$ in $\GPart(U)$ (note: if $M\preceq M'$ then $M'$ will have \emph{smaller} or equal cardinality compared to $M$). We will also set $M\vee M'=M(\Lambda\vee \Upsilon,U)$ and $M\wedge M'=M(\Lambda\wedge \Upsilon, U)$. 
We set $M_{\bot}=M(\Lambda_{\bot},U)$ and $M^{\top}=M(\Lambda^{\top},U)$. If $M=M(\Lambda,U)$ and $\Lambda^{\diamond}$ is a complement of $\Lambda$ in the lattice $\GPart(U)$, then we set $M^{\diamond}=M(\Lambda^{\diamond},U)$ and call $M^{\diamond}$ a \emph{complement} of $M$. A subfactor $M_i$ of size 1 of $M=M(\Lambda,U)\in \MP(U)$ is called  \emph{co-trivial}. A subfactor $M_i$ is called \emph{trivial} if the partition $\alpha_i\in \Part(U_i)$ of $\Lambda$ is trivial. 

Using the one-to-one correspondence  between the minor-subproducts $M=M(\Lambda, U)=M_1\times\ldots\times M_m\in \MP(U)$  and generalized partitions $\Lambda=(\alpha_1,\ldots,\alpha_m)\in \GPart(U)$, we distinguish the following types of minor-subproducts and generalized partitions:   

\begin{itemize}
\item \emph{trivial} if all partitions $\alpha_1,\ldots\alpha_m$ of $\Lambda$ are trivial. Then $M=M(\Lambda,U)$ coincides with $U$ and $\Lambda$ coincides with the  minimum $\Lambda_\bot$ of $\GPart(U)$;
\item \emph{co-trivial} if all partitions $\alpha_1,\ldots\alpha_m$ of $\Lambda$ are co-trivial. Then $M$ is a single $m$-tuple and $\Lambda$ coincides with the maximum $\Lambda^\top$ of  $\GPart(U)$; 
\item \emph{mixed} if each of the  partitions $\alpha_1,\ldots\alpha_m$ is either trivial or co-trivial;
\item \emph{elementary} if $\Lambda$ is an atom of the lattice $\GPart(U)$; 
\item \emph{one-dimensional} if $|\supp(\Lambda)|=1$;
\item \emph{co-elementary}  if $\Lambda$ is a co-atom of the lattice $\GPart(U)$; 
\item \emph{binary} if all partitions $\alpha_1,\ldots\alpha_m$ of $\Lambda$ have either two blocks or one block;
\item \emph{extended subproduct} if all partitions $\alpha_1,\ldots\alpha_m$ of $\Lambda$ are quasi-trivial.
\end{itemize}

We denote by $\MMP(U),\EMP(U), \EEMP(U)$ and $\BMP(U)$ the sets of all mixed, extended, elementary, and binary  minor-subproducts of $U$, respectively. 

 
\subsection{Boxes and box-partitions} We now introduce a third framework besides generalized partitions and minor-subproducts. The convex sets of any Cartesian product of graphs $G=G_1\times\ldots\times G_m$ are exactly the Cartesian products $B=B_1\times\ldots\times B_m$ of nonempty convex subsets $B_i$ of the factors $G_i$ \cite{VdV}. In case of the Hamming graph $H(U)$ of a Cartesian product $U=U_1\times \ldots \times U_m$, the factors of $H(U)$ are cliques and any nonempty subset of a clique is convex, thus the nonempty convex sets  of $U$ are exactly the full-dimensional subproducts of $U$: 

\begin{lemma} \label{convex-Hamming} A set $B\subseteq U=U_1\times \ldots\times U_m$ is a convex subset of $U$  if and only if $B$ is a full-dimensional subproduct.
\end{lemma}

In view of this lemma, we call the full-dimensional subproducts of $U$ \emph{boxes}. Let $\BBox(U)$ denote the set of all boxes of $U$ plus the empty set. Since $\BBox(U)$ is closed by intersections and $U\in \BBox(U)$, by \cite[Theorem 2.21]{DaPr}, $\BBox(U)$ is a  lattice with respect to inclusion: the meet of two boxes is their intersection and the join is the smallest box containing them. Although we will not need the lattice structure of $\BBox(U)$, we will use the following relation on boxes: 

\begin{definition} [Parallel boxes and galleries] \label{def:parallel} Two boxes $B',B''\in \BBox(U)$  are called \emph{parallel} if there exist a subproduct $V=V_{i_1} \times\ldots\times V_{i_k} $ of $U$ and two $(X\setminus \supp(V))$-tuples $t',t''$ such that $B'=V\times t'$ and $B''=V\times t''$, i.e., if $B'$ and $B''$ are copies of $V$. Denote by $\Theta(V)$ all copies of $V$. With some abuse of notation, for a box $B$ we denote by $\Theta(B)$ all boxes of $U$ parallel to $B$. Obviously, $B\in \Theta(V)$ if and only of $\Theta(B)=\Theta(V)$. 

Let $B',B''\in \Theta(V)$ be two parallel boxes, say  $B'=V\times t',B''=V\times t''$. Then $B',B''$ are \emph{adjacent} if the tuples $t',t''$ differ in a single coordinate. A \emph{gallery of length $\ell$} between  $B',B''$ is a sequence $\gamma(B',B'')=(B'=B_0,B_1,\ldots,B_{\ell-1},B_\ell=B'')$ of boxes of $\Theta(V)$ such that any two consecutive boxes $B_i,B_{i+1}, i=0,\ldots,\ell-1$ are adjacent. 
A \emph{geodesic gallery} between $B',B''$ is a gallery of length equal to the Hamming distance between the tuples $t'$ and $t''$. 
\end{definition}

\begin{example}\label{parallel-edges}
Each pair of tuples $u',u''$ of $U$ (i.e., vertices of $H(U)$) are parallel and a gallery between them is any path in $H(U)$ connecting $u'$ and $u''$. Now, consider the case of two parallel edges $e'=u'v'$ and $e''=u''v''$. This means that $e',e''\in \Theta(W)$ for some subproduct $W$, which has $\supp(W)=\{ i\}$ for some $i$, and $W=\{ u_i,v_i\}$; and then $\{u_i',v_i'\}=\{u_i'',v_i''\}=\{u_i,v_i\}$, so we can pick $(X\setminus \{ i\})$-tuples $t',t''$ such that $e'=W\times t'$ and $e''=W\times t''$. 
\end{example} 

$\BBox(U)$ is associated with its own notion of partition: 

\begin{definition} [Box-partitions] \label{d:box-partition} A \emph{box-partition} of $U=U_1\times \ldots \times U_m$ is a partition 
$\cB=\{ B_1\ldots, B_k\} $ of $U$ into boxes, i.e., $B_1, \ldots, B_k$ are boxes of $U$, 
$\bigcup_{i=1}^k B_i=U$, and $B_i\cap B_j=\varnothing$ for any $i\ne j$. We denote by $\BPart(U)$ the set of all box-partitions $\cB$ of $U$. 
\end{definition}

The set $\BPart(U)$ of box-partitions defines a lattice. For two box-partitions $\cB,\cB'\in \BPart(U)$, we set  ${\cB}\preceq {\cB}'$ if $\cB$ refines $\cB'$, i.e., each box of $\cB'$ is a disjoint union of boxes of $\cB$. The minimal element is the box-partition consisting of boxes of size 1 and the maximal element is the box-partition consisting of a single box $U$. The meet $\cB'\wedge \cB''$ of two box-partitions $\cB',\cB''$ is the box-partition $\cB$ whose boxes are the non-empty intersections of boxes from $\cB'$ and $\cB''$. By \cite[Theorem 2.16]{DaPr}, $\BPart(U)$  is indeed a lattice. The join $\cB'\vee \cB''$ of  $\cB',\cB''$ is the box-partition obtained by applying to the set $\mathcal S$ of all boxes of $\cB'$ and $\cB''$ the following iterative procedure: at each iteration, pick any pair $B',B''\in \mathcal S$ of intersecting boxes and replace them in $\mathcal S$ by the smallest box $B$ containing $B'$ and $B''$, and continue until $\mathcal S$ becomes a partition of $U$. Since at each iteration the number of boxes strictly decreases, this procedure finishes with a partition. 

Generalized partitions $\Lambda\in \GPart(U)$ (and their minor-subproducts $M(\Lambda)\in \MP(U)$) can be viewed as box-partitions. Namely, to  each generalized partition $\Lambda=(\alpha_1,\ldots,\alpha_m)$ with $\alpha_i=(P^i_{1},\ldots, P^i_{\ell_i})$ we associate the set ${\mathcal B}(\Lambda,U)$  of all  full-dimensional subproducts of the form $P^1_{j_1}\times\ldots\times P^m_{j_m}$. Then ${\mathcal B}(\Lambda,U)$ is a box-partition of $U$. 
If $\Lambda,\Upsilon$ are distinct generalized partitions, then 
${\mathcal B}(\Lambda,U)$  and ${\mathcal B}(\Upsilon,U)$ are distinct box-partitions of $U$ and $\Lambda\preceq \Upsilon$ if and only if $\cB(\Lambda,U)\preceq \cB(\Upsilon, U)$. The meet operation in the lattice $\GPart(U)$ corresponds to the meet operation in the lattice $\BPart(U)$: for $\Lambda,\Upsilon\in \GPart(U)$ we have $\cB(\Lambda\wedge \Upsilon)=\cB(\Lambda)\wedge \cB(\Upsilon)$. 

Not every box-partition $\cB\in \BPart(U)$ is the box-partition of some generalized partition, for example, the partition of $\{ a,b\}\times \{ A,B,C\}$ into 3 boxes $B_1=\{ (a,A),(a,B),(b,A),(b,B)\}, B_2=\{ (a,C)\}, B_3=\{ (b,C)\}$. Therefore,  we only have an injection from $\GPart(U)$ to $\BPart(U)$. Nevertheless, in our definitions and proofs it will be convenient to use the bijection between $\GPart(U)$ and  $\MP(U)$ and their injection into $\BPart(U)$ because each minor-subproduct $M=M(\Lambda)$ can be fully described by its box-partition $\cB(\Lambda)$. 

Each box $B$ arises as a box in the box-partition $\cB(\Lambda,U)$ of some generalized partition 
$\Lambda$. Indeed, suppose that $B=B_1\times\ldots\times B_m$ for $B_i\subseteq U_i, i=1,\ldots,m$. Let $A_i=U_i\setminus B_i$ for each $i$ such that $B_i\ne U_i$. Consider the binary generalized partition $\Lambda=(\alpha_1,\ldots,\alpha_m)$, where $\alpha_i=\{ A_i,B_i\}$ or $\alpha_i=\{ U_i\}$. Then clearly, $B=B_1\times\ldots\times B_m$ is a box of the box-partition $\cB(\Lambda,U)$. Note that the same box $B$  may occur in different box-partitions $\cB(\Lambda,U)$ and $\cB(\Upsilon,U)$ and thus in the description of minor-subproducts $M(\Lambda)$ and $M(\Upsilon)$. In this case, in the box-partition descriptions of $M(\Lambda)$ and $M(\Upsilon)$ we will identify such boxes $B$. Likewise, for convenience of notation, we often label elements of minor-subproducts by their respective box, and consider two elements of $M(\Lambda)$ and $M(\Upsilon)$ equal if they induce the same box in the box partition. In particular, since there can be boxes with one element, we identify such boxes with the element they contain, such that statements like $M(\Lambda_{\bot},U)=U$ are well-defined. 

\begin{figure}[htbp]
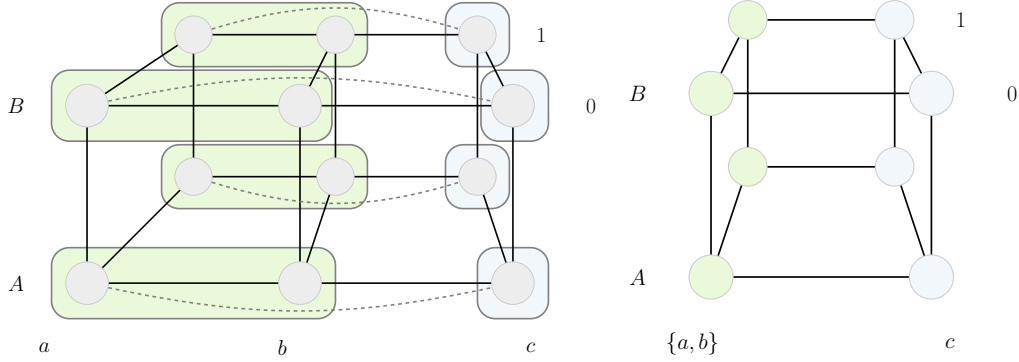

    \centering
    \includestandalone[height=0.3\linewidth]{img/boxpartition}
    \includestandalone[height=0.3\linewidth]{img/minorsubproduct}
    \caption{A box partition (left) and its related minor-subproduct (right).}
    \label{fig:boxpartitionexample} 
\end{figure}

\begin{example}
    Consider again the Cartesian product $\{a,b,c\}\times \{A,B\}\times \{0,1\}$, and take the generalized partition $\Lambda=\left(\{\{a,b\},\{c\}\},\{\{A\},\{B\}\},\{\{0\},\{1\}\}\right)$. This generalized partition induces a box partition (\cref{fig:boxpartitionexample}, left), and also a minor-subproduct (\cref{fig:boxpartitionexample}, right). As shown by the coloring, there is a one-to-one correspondence between the vertices of the minor-subproduct and the boxes of the box partition.
\end{example}
 Box-partitions arising from generalized partitions can be characterized in the following way:

\begin{lemma} \label{box-partition} For a box-partition $\cB=\{ B_1,\ldots,B_k\}\in \BPart(U)$ there exists a generalized partition $\Lambda\in \GPart(U)$ such that $\cB=\cB(\Lambda,U)$ if and only if: for each edge $uv$ of $H(U)$ such that the tuples $u$ and $v$ belong to distinct boxes of $\cB$, and for every edge $u'v'$ parallel to $uv$, the tuples $u'$ and $v'$ also belong to distinct boxes of $\cB$. 
\end{lemma}

\begin{proof} First suppose that $\cB=\cB(\Lambda,U)$ for $\Lambda=(\alpha_1,\ldots,\alpha_m)\in \GPart(U)$. Pick two parallel edges $uv$ and $u'v'$ with 
$u$ and $v$ belonging to distinct boxes of $\cB$, say $u\in B'$ and $v\in B''$. Since $B',B''\in \cB=\cB(\Lambda,U)$, 
$B'$ and $B''$ can be written as $B'=P'_1\times\ldots\times P'_m$ and $B''=P''_1\times\ldots\times P''_m$, where $P'_i,P''_i$ are blocks of the partition $\alpha_i$, $i=1,\ldots,m$. Since $uv$ is an edge of $H(U)$, 
the tuples $u=(u_1,\ldots,u_m)$ and $v=(v_1,\ldots,v_m)$ differ in a single coordinate, say $u_1\ne v_1$ and $u_i=v_i$ for $i>1$. Since $u\in B'$ and $v\in B''$, this implies  that $u_1\in P'_1$ and $v_1\in P''_1$. Since the edges 
$uv$ and $u'v'$ are parallel, the tuples  $u=(u'_1,\ldots,u'_m)$ and $v=(v'_1,\ldots,v'_m)$ also differ only in the first coordinate, in which $u'_1=u_1$ and $v'_1=v_1$. This implies that $u'_1\in P'_1$ and $v'_1\in P''_1$, thus the vertices $u'$ and $v'$ also belong to different boxes of $\cB$. 

Conversely, suppose that $\cB$ is a box-partition of $U$ whose parallel edges satisfy the condition of the lemma. By definition (see Lemma \ref{convex-Hamming}), each box $B$ of $\cB$ is a full-dimensional subproduct $P_1\times\ldots\times P_m$. For any $i$, let $\alpha_i$ be the set consisting of all blocks $P_i$ that occur in the subproduct representation $P_1\times\ldots\times P_m$ of a box $B$ of $\cB$. 
We assert that each $\alpha_i$ is a partition of $U_i$. Suppose this is not true. Then there exists $i$ and two boxes $B',B''$ such that their $i$th factors $P'_i$ and $P''_i$ are distinct but they intersect. Let $u_i\in P'_i\cap P''_i$. Since 
$P'_i$ and $P''_i$ are distinct, there exists $v\in P'_i\Delta P''_i$, say $v_i\in P''_i\setminus P'_i$. Since $u_i\in P'_i$ and $v_i\notin P'_i$, we can find an edge $uv$ of $H(U)$ with $u\in B'$ and $v\notin B'$ such that the $i$th coordinate of $u$ is $u_i$ and the $i$th coordinate of $v$ is $v_i$. Analogously, since $u_i,v_i\in P''_i$ and the box $B''$ is a full-dimensional subproduct of $U$ with $P''_i$ as a factor, 
we can find an edge $u'v'$ of $H(U)$ with $u',v'\in B''$  such that the $i$th coordinate of $u'$ is $u_i$ and the $i$th coordinate of $v'$ is $v_i$. But then the edges $uv$ and $u'v'$ are parallel and they violate the condition of the lemma. 
\end{proof}

For box-partitions of $U$ we apply the same terminology as for generalized partitions and their minor-subproducts: we call  a box-partition  $\cB(\Lambda,U)$ trivial, co-trivial, mixed, etc. if its related generalized partition $\Lambda$ is respectively trivial, co-trivial, mixed, etc. In language  of box-partitions $\cB(\Lambda,U)$, these terms can be interpreted as follows:

\begin{itemize}
\item \emph{trivial}: all boxes of $\cB(\Lambda,U)$ are singletons; 
\item \emph{co-trivial}: $\cB(\Lambda,U)$ consists of a single box $U$;
\item \emph{mixed} if $\Lambda$ has $k$ co-trivial factors $U_{i_1},\ldots,U_{i_k}$  and $m-k$ trivial factors, then  $\cB(\Lambda,U)$ is a partition of 
$U$ into parallel boxes from $\Theta(U_{i_1}\times\ldots\times U_{i_k})$;
\item \emph{elementary}:  all boxes of $\cB(\Lambda,U)$ are singletons except a set of boxes of size 2 corresponding to the class $\Theta(uv)$ of parallel edges of $H(U)$. They are defined by the unique block $\{u,v\}$ of size 2 of $\Lambda$;
\item \emph{one-dimensional}: each box of $\cB(\Lambda,U)$ is a clique whose elements only differ in their $i$-coordinate, where $\supp(\Lambda)=\{ i\}$;
\item \emph{co-elementary}:  $\cB(\Lambda,U)$ contains two boxes, i.e., the box partitions of co-elementary generalized partitions are the pairs of proper complementary halfspaces of $H(U)$.
\end{itemize}

\subsection{Shattering and strong-shattering of minor-subproducts} 
In this subsection, we define shattering and strong-shattering of minor-subproducts of $U$ by subsets of $U$. For this, we generalize the notions of extension, fiber, and copy. 

\begin{definition}\label{expansion-fibers-revised} [Expansions and fibers for minor-subproducts]
Let $M=M(\Lambda,U)=M_1\times\ldots\times M_m\in \MP(U)$, where $\Lambda=(\alpha_1,\ldots,\alpha_m)\in \GPart(U)$ and for $i=1,\ldots,m$, $\alpha_{i}=\{ P^{i}_1,\ldots,P^{i}_{\ell_{i}}\}$ is a partition of $U_i$ and $M_i=\{w^i_1,\ldots,w^i_{\ell_i}\}$. A tuple $u=(u_1,\ldots,u_m)\in U$ is called an \emph{expansion} of a tuple $t=(t_1,\ldots,t_m)\in M$ if each coordinate $t_i$ corresponds to the block containing $u_i$, i.e. for each $i\in X$ and $j=1,\ldots,\ell_i$, we have $u_i\in P^i_{j}$ iff $t_i=w^i_{j}$. The \emph{fiber} $F(t)$ of $t\in M$ consists of all expansions of $t$ in $U$.
\end{definition}

From the definition it follows that: 

\begin{lemma} \label{fibers-as-boxes} If $M=M(\Lambda,U)=M_1\times\ldots\times M_m\in \MP(U)$,  then the following hold:

\begin{itemize}
\item[(1)]  each fiber $F(t)$ of $t\in M$ is a box of $U$ defined by $\Lambda$ and each box defined by $\Lambda$  is a fiber of a unique tuple $t\in M$;
\item[(2)] If $A=\supp(M)$ and $B=X\setminus A$, then each fiber $F(t), t\in M$ has the form $F(t)=(\prod_{i\in A}\prod_{j=1}^{\ell_i} P^i_{j})\times (\prod_{i\in B} U_i)$;
\item[(3)] the set $\{ F(t): t\in M\}$ defines a box-partition of  $U=U_1\times\ldots\times U_m$. 
\end{itemize}
\end{lemma}

In previous sections we defined shattering, strong-shattering and copies for subproducts. Now that we introduced minor-subproducts, we are ready to define our main notions of shattering and strong-shattering for minor-subproducts. An illustration of these notions is given in \cref{fig:shatterstrongshatter}. 

\begin{definition}[Shattering minor-subproducts]
Let  $U=U_1\times\ldots\times U_m$ and $S$ be a subset of $U$.
A minor-subproduct $M\in \MP(U)$   is \emph{shattered} by $S$ if for any $t\in M$ the fiber $F(t)$ intersects $S$, i.e., $t$ has an expansion belonging to $S$. Equivalently, if $M=M(\Lambda)$ for $\Lambda\in \GPart(U)$, then $M$ is shattered by $S$ if $S$ intersects each box $B$ of the box-partition $\cB(\Lambda)$. 
\end{definition}


\begin{definition}[Copies of minor-subproducts]
Let $M=M(\Lambda,U)=M_1\times\ldots\times M_m$ be a minor-subproduct of $U=U_1\times\ldots\times U_m$ defined by $\Lambda=(\alpha_1,\ldots,\alpha_m)\in \GPart(U)$, where $\alpha_{i}=\{ P^{i}_1,\ldots,P^{i}_{\ell_{i}}\}$ is a partition of $U_i$, $i=1,\ldots,m$. For each $i=1,\ldots,m$ and
each $w^i_{j}\in M_i$, pick one element $u^i_j\in P^i_j, j=1,\ldots, \ell_i$. Set $W_{i}=\{ u^{i}_{j_1},\ldots,u^{i}_{\ell_i}\}$ and  consider the full-dimensional subproduct $W=W_{1}\times\ldots\times W_{m}$ of $U$. Any subproduct $W$ of this form is called a \emph{copy} of $M$. 
\end{definition}

\begin{definition}[Strong-shattering of minor-subproducts]\label{def:strongshatter} A minor-subproduct $M\in \MP(U)$  is \emph{strongly-shattered} by a set $S\subseteq U=U_1\times\ldots\times U_m$ if
$S$ contains a copy of $M$.
\end{definition}

\begin{example}\label{ex:shatterbinary}
    In the binary case $U=\{a,b\}^X$, in previous literature (e.g. \cite{BaChDrKo,ChChMoWa}) a set $Y\subseteq X$ is shattered by $S\subseteq U$ if for all $Y$-tuples $t\in \{a,b\}^Y$ there is an extension of $t$ contained in $S$. In our definition, this is equivalent to stating that $S$ shatters the unique minor-subproduct $M(Y)\in \MP(U)$ with $\supp(M(Y))=X\setminus Y$ (note: uniqueness follows since in the binary case every minor-subproduct is mixed). Likewise, in previous literature strong-shattering of $Y$ by $S$ meant that there exists an $(X\setminus Y)$-tuple $t$ such that all extensions of $t$ belong to $S$. And that in turn is equivalent to saying that $S$ strongly shatters the minor-subproduct $M(Y)$.
\end{example}

\begin{figure}[htbp]
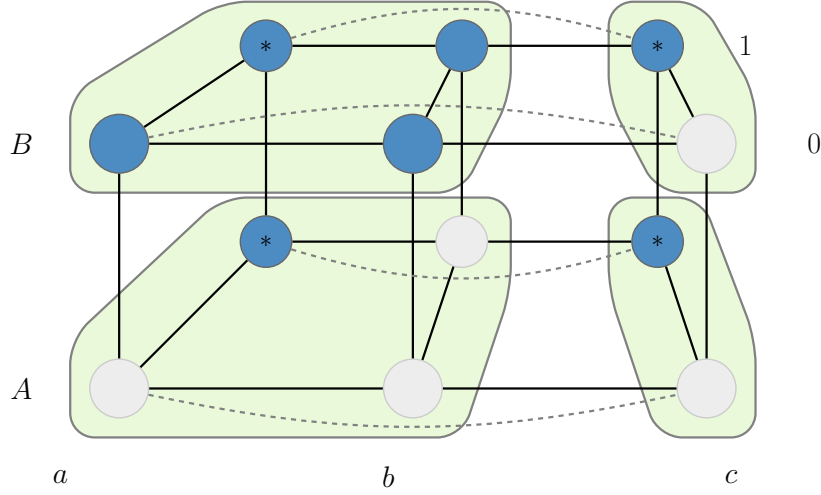

    \centering
    \includestandalone[width=0.7\linewidth]{img/shatterstrongshatter}
    \caption{The set $S$ in the figure shatters the minor-subproduct $M=M(\Lambda,U)$, where $\Lambda=\left(\{\{a,b\},\{c\}\},\{A\},\{B\},\{\{0,1\}\}\right)$. This can be observed from the figure, since every box contains an element of $S$. The set $S$ also strongly shatters $M$, since $S$ contains a copy of $M$, namely $\{a,c\}\times \{A,B\}\times\{1\}$ (marked by $*$ in the figure).}
    \label{fig:shatterstrongshatter}
\end{figure}

The notion of minor-subproduct can be extended in a straightforward way to all subproducts $V=V_{i_1}\times\ldots\times V_{i_k}$ of $U$ and therefore the set $\MP(V)$ of minor-subproducts of $V$ is well-defined. However, $\MP(V)$ is not a subset of $\MP(U)$, even if $V$ is full-dimensional (recall that in that case, $V$ is a subset of $U$). Nevertheless, we can define the notion of shattering and strong-shattering of minors $M\in \MP(V)$ for all subproducts $V$ of $U$ and all subsets $S$ of $U$. 
First, if $V=V_1\times\ldots\times V_m\subseteq U$ is a full-dimensional product of $U$, then $M=M(\Lambda,V)\in \MP(V)$ is 
\emph{shattered} by $S\subseteq U$ if $M$ is shattered by $S\cap V$ in $V$, i.e., if each box $B$ of the box-partition $\cB(\Lambda,V)$ of $V$ intersect $S$. Analogously, $M$ is \emph{strongly shattered} by $S$ if $S\cap V$ contains a copy of $M$. 
Now, suppose that $V=V_{i_1}\times\ldots\times V_{i_k}$ is not full-dimensional and $A=\{ i_1,\ldots,i_k\}$. Let $M=M_{i_1}\times\ldots\times M_{i_k}$ be a minor-subproduct of $V$ defined by the generalized partition $\Lambda=(\alpha_{i_1},\ldots,\alpha_{i_k})$ of $V$. Consider the full-dimensional subproduct $V'=V'_1\times\ldots\times V'_m$, where $V'_i=V_i$ if $i\in A$ and $V'_i=U_i$ if $i\notin A$. Consider the generalized partition $\Lambda'=(\alpha'_1,\ldots,\alpha'_m)$ of $V'$, where $\alpha'_i=\alpha_i$ if $i\in A$ and $\alpha'_i=\{ U_i\}$ if $i\notin A$. Finally, let $M'=M(\Lambda',V')$ be the minor-subproduct of $V'$ corresponding to $\Lambda'$.  Then we say that the minor $M$ of $V$ is \emph{shattered} by a set $S\subseteq U$ if the minor $M'$ of the full-dimensional subproduct $V'$ is shattered by $S$. Analogously, $M$ is 
\emph{strongly shattered} by $S$ if $M'$ is strongly shattered by $S$. We denote by $\MP^*(U)$ the set of all minor-subproducts $M\in \MP(V)$ over all full-dimensional subproducts $V$ of $U$, and by $\MP^{**}(U)$ the set of all minor-subproducts $M\in \MP(V)$ over all subproducts $V$ of $U$. Notice that $\MP(U)\subset \MP^*(U)\subset \MP^{**}(U)$.

\subsection{Ample and lopsided sets} We continue with the definitions of ample and lopsided  sets of Cartesian products of sets, which generalize the notions of ample and lopsided sets for binary products. We provide several versions of ample sets by  applying the \emph{``shattering$\rightarrow$strong-shattering'' principle} to some  subsets of minor-subproducts of $U$ and of its subproducts.  

\begin{definition}[Ample sets of Cartesian products]\label{def:basicample}
Let $S\subseteq U=U_1\times\ldots\times U_m$ and ${\mathcal M}$ be a subset of minor-subproducts of $\MP^{**}(U)$. Then $S$ is called: 
\begin{itemize}
\item  ${\mathcal M}$-\emph{ample} if  each   $M\in {\mathcal M}$ which is shattered by $S$ is strongly shattered by $S$;
\item \emph{ample} if $S$ is $\MP(U)$-ample. Let $\Amp(U)$ denote the set of all ample subsets of $U$; 
\item \emph{weakly ample} if $S$ is $\MMP(U)$-ample. 
\end{itemize}
We will also consider \emph{$\MP^{**}(U)$-ample}, \emph{$\MP^*(U)$-ample} and \emph{$\EMP(U)$-ample} sets. 
\end{definition}

By definition, if ${\mathcal M'}\subseteq {\mathcal M}''\subseteq \MP^{**}(U)$, then each ${\mathcal M}''$-ample set is  ${\mathcal M}'$-ample. Consequently, each $\MP^{**}(U)$-ample set is $\MP^{*}(U)$-ample, each  $\MP^{*}(U)$-ample set is ample, and each ample set is $\EMP(U)$-ample and weakly ample. Theorem \ref{thm:ample=weaklyample} in \cref{sec:ample-extample} establishes that all these variations of ampleness, except weak ampleness,  are equivalent. On the other hand, we prove in Proposition \ref{lopsider=wample} that weak ampleness is equivalent to lopsidedness (which we define below). 
Notice also that $\MMP(U)$-ampleness can be defined via shattering/strong-shattering of subproducts instead of minor-subproducts. The following simple example shows that weakly ample sets are not necessarily ample:

\begin{example} \label{weak-ample-nonample} Let $U=\{ a,b,c\}\times \{ A,B,C\}$, and consider the set $S$ shown in blue in \cref{fig:amplenotisometric}. We show that $S$ is not ample or $\EMP$-ample. Take the extended subproduct $M=\{a,b,c\}\times\{\{A,B\},C\}$ (for ease of notation, we denote this minor-subproduct by its boxes, and one-element blocks by their element). $M$ has two copies in $U$, namely $W^1=\{a,b,c\}\times\{A,C\}$ and $W^2=\{a,b,c\}\times\{B,C\}$.

The extended subproduct $M$ is shattered but
not strongly shattered by $S$. Indeed, one can check that the fiber $F(v)$ of each of the six tuples $v$ of $M$ intersects $S$. However, none of the possible two copies of $M$ is contained in $S$, as $(c,A)\in W^1\backslash S$ and $(b,B)\in W^2\backslash S$.
So $S$ is not ample or $\EMP$-ample. However, the set $S$ is weakly ample, as one can easily verify that $M\in \MMP(U)$ is both shattered and strongly shattered if and only if $\supp(M)\geq 1$. Thus weak ampleness is a weaker notion than ampleness.


    \begin{figure}[htbp]
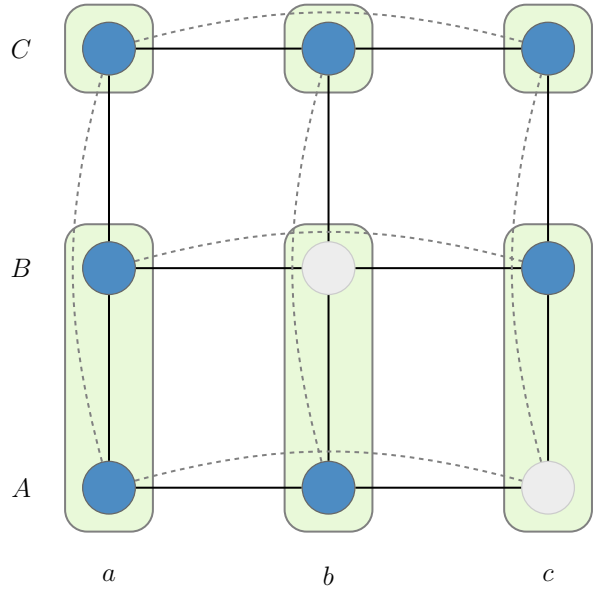

        \centering
        \includestandalone[width=0.5\linewidth]{img/amplenotisometric2}
        \caption{The Cartesian product $U=\{a,b,c\}\times\{A,B,C\}$, its Hamming graph, and the box-partition defined by $M$ from Example \ref{weak-ample-nonample}. The set $S$ is shown in blue.}
        \label{fig:amplenotisometric}
    \end{figure}
\end{example}

\begin{example}\label{ex:amplebinary}
The classical notion of ample/lopsided set \cite{BaChDrKo,La} corresponds to ample sets in \emph{binary products}, i.e., in  Cartesian products of sets of size 2. In this case, each factor can be identified with e.g. $\{ 0,1\}$, $\{ a,b\}$, or with $K_2$. The subsets $S$ of $U$ correspond to sets of binary vectors, or equivalently, to subsets of $2^X$, corresponding to families of subsets of $X$. In this case, the three definitions of ample sets coincide with the classical definition: this is since we saw in \cref{ex:shatterbinary} that shattering and strong-shattering can be described in terms of minor-subproducts, and since in the binary case every minor-subproduct is mixed.
\end{example}

\begin{remark} The notions of ampleness with respect to the sets $\EEMP(U)$ and $\CMMP(U)$ of elementary and co-elementary minor-subproducts are not considered because they are too permissive. For example, all elementary minor-subproducts define box-partitions with many boxes, therefore any sparse enough set $S\subseteq U$ will not shatter any such minor and thus $S$ will be  $\EEMP(U)$-ample. On the other hand, any connected set $S$ will satisfy the ``shattering$\rightarrow$strong-shattering'' principle for all co-elementary minor-subproducts and thus will be $\CMMP(U)$-ample. Nevertheless, elementary minor-subproducts will be very useful in recursive characterizations of ample sets. 
\end{remark}

For a set $S\subseteq U$, denote by $S^*=U\setminus S$ its complement in $U$. In case of binary products $U=\{ a,b\}^X$, Lawrence's \cite{La} definition of lopsided sets 
can be stated in the following way in terms of strong-shattering: $S$ is lopsided if for any partition of $X$ into the sets $A$ and $B$, 
either $A$ is strongly shattered by $S$ or $B$ is strongly shattered by $S^*$. In case of $U=\{ a,b\}^X$, the lattice of generalized 
partitions $\GPart(U)$ is isomorphic to the Boolean lattice, which is a uniquely complemented lattice: the unique complement of $M$ with $\supp(M)=A$ is given by the minor-subproduct $M^{\diamond}$ with $\supp(M)=X\backslash A$.
For general Cartesian products $U$,  the lattices $\GPart(U)$ and $\MP(U)$ are still complemented but are no longer uniquely complemented. If we require Lawrence's dichotomy \emph{``either $M$ is strongly shattered by $S$ or $M^\diamond$ is strongly shattered by $S^*$''} to any $M\in \MP(U)$ {and to any complement $M^{\diamond}$, then one can show that, for example, if all factors have size $>2$, then $S=\varnothing$ or $S=U$.}

To get a more insightful notion of lopsidedness, instead of all minor-subproducts we will consider only mixed minor-subproducts, which are the minor-subproducts where each subfactor is either trivial or co-trivial. 
Each mixed minor-subproduct $M$ has a unique complement $M^\diamond$, which is also a mixed minor-subproduct. 

\begin{definition}[Lopsided sets of Cartesian products]\label{def:basicample}
A set $S\subseteq U=U_1\times\ldots\times U_m$ is called 
\item \emph{lopsided} if for any mixed minor-subproduct $M$ of $U$ either $M$ is strongly shattered by $S$ or the complement $M^\diamond$ of $M$  is strongly shattered by $S^*$. 
\end{definition}

From the definition it follows that the class of lopsided sets is closed by taking complements. 
The next result establishes an equivalence between lopsided  and weakly ample sets: 

\begin{proposition} \label{lopsider=wample} A set $S\subseteq U=U_1\times\ldots\times U_m$ is lopsided if and only if $S$ is weakly ample. Consequently, the complement $S^*$ of a  weakly ample set $S$ is  weakly ample. 
\end{proposition}

\begin{proof} First, suppose that $S$ is lopsided. To prove that $S$ is weakly ample, pick any mixed minor-subproduct $M=M_1\times\ldots\times M_m$ shattered by $M$. Let $M=M(\Lambda,U)$, where $\Lambda=(\alpha_1,\ldots,\alpha_m)\in \GPart(U)$ and assume $M_1,\ldots,M_i$ are the trivial subfactors and $M_{i+1},\ldots,M_m$ the co-trivial subfactors of $M$. Then the boxes of the box-partition $\cB(\Lambda,U)$ of $U$ have the form $\{ u_1\}\times\ldots\times \{ u_i\}\times U_{i+1}\times\ldots U_m$, where $u_1\in U_1,\ldots, u_i\in U_i$. 
Notice also that the complement of $M$ in $\MP(U)$ is the mixed minor-subproduct $M^\diamond=M^\diamond_1\times\ldots\times M^\diamond_m$  whose factors $M^\diamond_1,\ldots,M^\diamond_i$ are co-trivial and the factors $M^\diamond_{i+1},\ldots,M^\diamond_m$ are trivial. 

Suppose by way of contradiction that $M$ is not strongly shattered by $S$. Since $S$ is lopsided, the complement $M^\diamond$ of $M$ in $\MP(U)$ is strongly shattered by $S^*$ and thus $S^*$ contains a copy $W$ of $M^\diamond$.  Together with the structure of $M^\diamond$, this implies that there exist $u^*_1\in U_1,\ldots, u^*_i\in U_i$ such that for any choice of the elements $u_{i+1}\in U_{i+1},\ldots, u_m\in U_m$, the tuple $(u^*_1,\ldots,u^*_i,u_{i+1},\ldots,u_m)$ belongs to $S^*$. Equivalently, the box $B:=\{ u^*_1\}\times\ldots\times \{ u^*_m\}\times U_{i+1}\times\ldots\times U_m$ is included in $S^*$. But $B$ is a box of $\cB(\Lambda,U)$, contrary to the assumption that $M$ is shattered by $S$ and thus each box of $\cB(\Lambda,U)$ must contain an element of $S$. This contradiction shows that $M^\diamond$ cannot be strongly shattered by $S^*$ and by lopsidedness of $S$, $M$ must be strongly shattered by $S$. Consequently, any lopsided set is weakly ample. 

Conversely, suppose that $S$ is  weakly ample. Pick any mixed minor $M\in \MMP(U)$  and let 
$M^\diamond$ be its complement in $\MP(U)$. We use the same notation as before for the 
factors of $M$ and $M^\diamond$. We assert that either $M$ is strongly shattered by $S$ or 
$M^\diamond$ is strongly shattered by $S^*$. First suppose that $M$ is strongly shattered by $S$. 
Then $S$ contains a copy of $M$, which implies that there exist 
$u^*_{i+1}\in U_{i+1},\ldots, u^*_m\in U_m$ such that for any choice of the 
elements $u_{1}\in U_{1},\ldots, u_i\in U_i$, the tuple $(u_1,\ldots,u_i,u^*_{i+1},\ldots,u^*_m)$ 
belongs to $S^*$. This implies that the box $B':=U_1\times\ldots\times U_i\times \{ u^*_{i+1}\}\times\ldots\times \{ u^*_m\}$ is included in $S$. As showed above, any copy of 
$M^\diamond$ is a box of the form $B'':=\{ u''_1\}\times\ldots\times \{ u''_i\}\times U_{i+1}\times\ldots\times U_m$ for some choice of the elements $u''_1\in U_1,\ldots,u''_i\in U_i$. 
Each such box $B''$ intersects $B'$ in the tuple $(u''_1,\ldots,u''_i,u^*_{i+1},\ldots, u^*_m)$, and therefore cannot be included in $S^*$. Consequently, $S^*$ does not contain any copy of $M^\diamond$, hence $M^\diamond$ is not strongly shattered by $S^*$. 

Now suppose that $M\in \MMP(U)$ is not strongly shattered by $S$. Since $S$ is  weakly ample, this implies that $M$ is not shattered by $S$. Therefore, the box-partition $\cB(\Lambda,U)$ contains a box $B$ not intersecting $S$. Then $B$ has the form $B=\{ u^*_1\}\times\ldots\times\{ u^*_i\}\times U_{i+1}\times\ldots\times U_m$ for some choice of   $u^*_1\in U_1,\ldots,u^*_i\in U_i$. Consequently,  $B$ is a copy of $M^\diamond$. Since $B\subseteq S^*$, $M^\diamond$ is strongly shattered by $S^*$. This concludes the proof that each  weakly ample set is lopsided.  
\end{proof}

\section{Properties of minor-subproducts}\label{section:propertiesminorsubproducts} In this section, we establish some auxiliary properties of minor-subproducts. 

\begin{lemma} \label{lem:refinement1} If $\Lambda,\Upsilon\in \GPart(U)$  with  $\Lambda\preceq \Upsilon$ and $M=M(\Lambda,U), M'=M(\Upsilon,U)$, then $M'\in \MP(M)$.
\end{lemma}

\begin{proof} Let $\Lambda=\{ \alpha_1,\ldots,\alpha_m\}$ and $\Upsilon=\{ \beta_1,\ldots,\beta_m\}$. Since $\Lambda\preceq \Upsilon$, each partition $\alpha_i$ of $U_i$ is a refinement of the partition $\beta_i$. By definition, $M=M_1\times\ldots\times M_m$ and $M'=M'_1\times\ldots\times M'_m$, where each $M_i$ has the blocks of $\alpha_i$ as vertices and $M'_i$ has the blocks of $\beta_i$ as vertices. Since $\beta_i$ is a coarsening of $\alpha_i$, $M'_i$ can be viewed as a minor of $U_i$ first obtained by contracting each block of $\alpha_i$ to a single element and then contracting all such elements belonging to the same block of $\beta_i$ into a single element. Consequently, each $M'_i$ is a minor of $M_i$ and therefore $M'=M'_1\times\ldots\times M'_m$ is a minor-product of $M=M_1\times\ldots\times M_m$. 
\end{proof}

The following lemma is a consequence of the previous result:

\begin{lemma} \label{lem:refinement2} If $\Lambda,\Upsilon\in \GPart(U)$  and $M=M(\Lambda,U), M'=M(\Upsilon,U)$, then 
$M,M'$ are minor-subproducts of $M\wedge M'$, and $M\vee M'$ is a minor-subproduct of $M$ and $M'$. 
\end{lemma}


\begin{lemma} \label{minors-of-minors} If $M\in \MP(U)$ and $M'\in \MP(M)$, then $M'\in \MP(U)$. 
\end{lemma}

\begin{proof} Let $M=M(\Lambda, U)=M_1\times\ldots\times M_m$ and $M'=M(\Upsilon,M)=M'_1\times\ldots\times M'_m$. Let $\Lambda=(\alpha_1,\ldots,\alpha_m)$, where $\alpha_i=\{ P^i_{1},\ldots,P^i_{\ell_i}\}, i=1,\ldots, m$. 
Then $M_i=\{ w^i_1,\ldots, w^i_{\ell_i}\}$ for $i=1,\ldots,m$. The generalized partition $\Upsilon$ has the form $\Upsilon=(\beta_1,\ldots,\beta_m),$ where each $\beta_i$ is a partition $\{ R^i_1,\ldots,R^i_{k_i}\}$ of $M_i$. If we  set $Q^i_j=\bigcup_{w^i_t\in R^i_j} P^i_t$, then we conclude that $\beta'_i=\{ Q^i_1,\ldots,Q^i_{k_i}\}$ is a partition of the factor $U_i$ that is coarser than $\alpha_i$ of $U_i$. Consequently, if we set $\Lambda'=(\beta'_1,\ldots,\beta'_m)$, then $\Lambda'\in \GPart(U)$ and $\Lambda\preceq \Lambda'$. Since $M'=M(\Lambda',U)$ (because each $z^i_j\in M'_i$ can be obtained by the contraction of the block $Q^i_j$ of $\beta'_i$), we get $M'\in \MP(U)$, as required. 
\end{proof}

The \emph{restriction} of a generalized partition $\Lambda=(\alpha_1,\ldots,\alpha_m)$ of $U$  to a full-dimensional subproduct 
$V=V_1\times\ldots\times V_m$ is the generalized partition $\Lambda'=(\alpha'_1,\ldots,\alpha'_m)$ of $V$, where  $\alpha'_1,\ldots,\alpha'_m$ of $\Lambda'$ are the restrictions of the partitions  $\alpha_1,\ldots,\alpha_m$ to $V_1,\ldots,V_m$, respectively. Then $M(\Lambda',V)\in \MP(V)$ is called the \emph{restriction} of $M(\Lambda,U)\in \MP(U)$ to $V$.

\begin{lemma} \label{minors-versus-subproducts1} If $V=V_1\times\ldots\times V_m$ is a full-dimensional subproduct of  $U=U_1\times\ldots\times U_m$, then any minor-subproduct  $M$ of $V$ is the restriction on $V$ of a minor-subproduct of $U$. Furthermore, if $M$ is an extended minor-subproduct of $V$, then $M$ is  the restriction on $V$ of an extended minor-subproduct of $U$. 
\end{lemma}

\begin{proof} Let $M=M(\Upsilon,V)=M_1\times\ldots\times M_m$ with 
$\Upsilon=(\beta_1,\ldots,\beta_m)\in \GPart(V)$, where $\beta_i=\{ Q^i_1,\ldots Q^i_{\ell_i}\}$ is a partition of $V_i$. From the partition $\beta_i$ of $V_i$ we derive a partition $\alpha_i=\{ P^i_1,\ldots,  P^i_{\ell_i}\}$ 
of $U_i$ by assigning the elements of $U_i\setminus V_i$ to the blocks of $\beta_i$ and without creating new blocks. If 
$M$ is an extended minor-subproduct of $V$ and $Q^i_j$ is the unique non-trivial block of the partition $\beta_i$, then 
in $\alpha_i$ we assign all elements of $U_i\setminus V_i$ to the block $Q^i_j$  of $\beta_i$ (or if $\beta_i$ is trivial, we assign all these elements to the first block). Then each $\alpha_i$ is a quasi-trivial partition of $U_i$. In all cases, let $\Lambda=(\alpha_1,\ldots,\alpha_m)$, then clearly the restriction of $M(\Lambda,U)$ to $V$ is $M(\Gamma,V)$. Moreover, if $M$ is an extended minor-subproduct of $V$, then $M(\Lambda,U)$ is an extended minor-subproduct of $U$. 
\end{proof}

\begin{lemma} \label{minors-versus-subproducts2} Let $V=V_1\times\ldots\times V_m$ be a full-dimensional subproduct of $U=U_1\times\ldots\times U_m$, $S\subseteq U$, and $M\in \MP(V)$ such that $M$ is the restriction of $M'\in\MP(U)$. If $M$ is shattered by $S\cap V$ in $V$, then $M'$ is shattered by $S$ in $U$. 
\end{lemma}

\begin{proof} Pick any $t\in M$. Since $M$ is shattered in $V$ by $S\cap V$, there exists an expansion $u$ of $t$ in $V$ belonging to $S\cap V$. Since $V$ is a full-dimensional subproduct of $U$, $u$ is a tuple of $U$, thus $u$ belongs to $S$ and to the fiber $F(t')$ in $U$, where $t'\in M'$ is the element corresponding to $t$. Consequently, $M$ is shattered by $S$ in $U$.
\end{proof}

\begin{lemma} \label{shattering-fulld-minor-subproducts} Let $V=V_1\times\ldots\times V_m$ be a full-dimensional subproduct of $U=U_1\times\ldots\times U_m$ and let $M=M(\Lambda,V)\in \MP(V)$. For a set $S\subseteq U$ the following conditions are equivalent:
\begin{itemize}
\item[(1)] $M$ is shattered by $S$;
\item[(2)] for each box $B$ of $\BPart(\Lambda,V)$, we have $B\cap S\ne\varnothing$;
\item[(3)] $M$ is shattered by $S\cap V$.
\end{itemize}
\end{lemma}

\begin{proof} The equivalence between (1) and (2) follows from Lemma \ref{fibers-as-boxes} and the definition of shattering, and the implications (2)$\Rightarrow$(3)$\Rightarrow$(1) are trivial.
\end{proof}

\begin{lemma} \label{ample-in-V-ample-in-U} If $V$ is a full-dimensional subproduct of $U=U_1\times\ldots\times U_m$ and $S\in \Amp(V)$, then $S\in \Amp(U)$. 
\end{lemma}

\begin{proof} It is enough to consider the case when   $V=(U_1\backslash\{x\})\times U_2\ldots\times U_m$ and $|U_1|>1$ 
and apply induction in remaining cases. Let $M=M(\Lambda,U)$  
be any minor-subproduct of $U$ that is shattered by $S$. 
Let $\Lambda=(\alpha_1,\ldots,\alpha_m)$ and suppose that 
$P^1_1$ is the block of $\alpha_1=\{P^1_1,\ldots,P^1_{\ell_1}\}$ containing the element $x$. Pick any 
tuple $t=(t_1,\ldots,t_m)\in M$. Since $M$ is shattered by $S$, $t$ has an expansion $u=(u_1,\ldots,u_m)$ belonging to $S$. Since $S\subseteq V=(U_1\backslash\{x\})\times U_2\ldots\times U_m$ it follows that $u_1\ne x$. We deduce that  the block $P^1_1$ is non-trivial. Consequently, $\alpha'_1=\{P^1_1\setminus \{ x\},\ldots,P^1_{\ell_1}\}$ is a partition of $U_1\setminus \{ x\}=V_1$, and $\Lambda'=(\alpha'_1,\alpha_2,\ldots,\alpha_m)$ is a generalized partition of $V$ and $M=M(\Lambda',V)$. Therefore $S$ shatters $M$ in $V$. Since $S\in \Amp(V)$, $M$ is strongly shattered 
by $S$ in $V$, hence $S$ has a copy $W$ of $M$. Since $W\subseteq S\subseteq V\subset U$, we are done. 
\end{proof}

Finally, in the Cartesian product $U=U_1\times\ldots\times U_m$ all factors $U_i$ are nonempty, however $U$ may contain trivial factors (of one element). 
For a trivial factor $U_i$, let $U^{-i}=U_1\times\ldots\times U_{i-1}\times U_{i+1}\times\ldots U_m$. Analogously, for a minor-subproduct 
$M=M_1\times\ldots\times M_m$ of $U$ with a trivial factor $M_i$, let $M^{-i}=M_1\times M_{i-1}\times M_{i+1}\times\ldots M_m$. 
Clearly, $M^{-i}$ is a minor-subproduct of $U^{-i}$. Finally, for a set $S\subseteq U$, the set $S^{-i}\subseteq U^{-i}$ consists of the traces $s|_{X\setminus \{ x\}}$ of all  $s\in S$ to $X\setminus \{ i\}$. The following lemma allows us to ignore 
trivial factors, and the proof of each its assertions directly follows from their respective definitions:

\begin{lemma}\label{lem:trivial-factor} Let $U_i$ be a trivial factor of $U=U_1\times\ldots\times U_m$. For a 
set $S\subseteq U$ and a minor-subproduct $M$ of $U$, the following equivalences hold:
\begin{itemize}
\item[(1)] $S$ is isometric in $U$ if and only $S^{-i}$ is isometric in $U^{-i}$;
\item[(2)] $M$ is shattered by $S$ if and only if $M^{-i}$ is shattered by $S^{-i}$;
\item[(3)] $M$ is strongly shattered by $S$ if and only if $M^{-i}$ is strongly shattered by $S^{-i}$;
\item[(4)] $S$ is ample (weakly ample) in $U$ if and only $S^{-i}$ is ample (resp. weakly ample) in $U^{-i}$.
\end{itemize} 
\end{lemma}


\section{Ampleness and its variations}\label{sec:ample-extample}

The goal of this section is to prove that ampleness is equivalent to other stronger and weaker versions. 
We also show that ample sets are isometric. 

\begin{theorem}\label{thm:ample=weaklyample} For a subset $S$ of $U=U_1\times \ldots \times U_m$, the following conditions are equivalent: 

\begin{itemize}
    \item[(1)] $S$ is ample; 
    \item[(2)] $S$ is $\EMP(U)$-ample;
    \item[(3)] $S$ is $\EMP(U)$-ample in $U$  and for any proper full-dimensional subproduct $V=V_1\times\ldots\times V_m\subset U$, $S\cap V$ is $\EMP(V)$-ample;
    \item[(4)] $S$ is $\MP^*(U)$-ample;
    \item[(5)] $S$ is $\MP^{**}(U)$-ample. 
\end{itemize}
\end{theorem}

\begin{proof} In the subsequent proofs, each time when we use the induction on the size of $U$, if $U$ has a trivial 
factor, then we can use the induction hypothesis on $U^{-i}$ and $S^{-i}$, ignoring trivial factors this way. The implication (1)$\Rightarrow$(2) is obvious. Now we will prove the implication (2)$\Rightarrow$(3)  by induction on   $|U_1|+\ldots+|U_m|$. By previous remark, we can suppose that all factors of $U$ are non-trivial. Pick any $x\in U_1$ and consider the full-dimensional subproduct $V=(U_1\backslash\{x\})\times U_2\times\ldots\times U_m$. Let $S'=S\cap V$. If we show that $S'$ is $\EMP(V)$-ample in $V$, then the statement of the theorem would follow by the induction hypothesis.  

To show the $\EMP(V)$-ampleness of $S'$, consider any  $M=M(\Upsilon,V)\in \EMP(V)$ that is shattered by $S'$ in $V$. Let $M(\Lambda,U)\in \EMP(U)$ be an extended minor-subproduct whose restriction to $V$ equals $M$, as defined in Lemma \ref{minors-versus-subproducts1}. Then 
by Lemma \ref{minors-versus-subproducts2}, $M(\Lambda,U)$ is shattered by $S$ in $U$. Since $S$ is $\EMP(U)$-ample, $M$ is strongly shattered by $S$, i.e., $S$ contains a copy $W$ of $M$.
To be precise, let $M=M(\Upsilon,V)=M_1\times\ldots\times M_m$ be defined in $V$ by the generalized partition $\Upsilon=(\beta_1,\ldots,\beta_m)\in \GPart(V)$, where  $\beta_{i}=\{ P^{i}_1,\ldots,P^{i}_{\ell_{i}}\}$ is a partition of $V_i, i=1,\ldots,m$.  
Since each partition $\beta_i$ is quasi-trivial, all blocks of $\beta_i$ are trivial, except for at most one block, which we will denote by  $P^i_1$. Then $M(\Lambda,U)$ is defined by the generalized partition $\Lambda=(\alpha_1,\ldots,\alpha_m)$ of $U$,  where $\alpha_i=\beta_i$ if $i>1$ and $\alpha_1=\{ P^{1}_1\cup \{ x\},P^1_2,\ldots,P^{1}_{\ell_{1}}\}$, and $M(\Lambda,U)$ is an extended minor-subproduct of $U$. We know that $S$ contains a copy $W=W_1\times\ldots\times W_m\subseteq S$ of $M(\Lambda,U)=M_1\times\ldots\times M_m$, whose factors are of the form $W_i=\{ u^{i}_1,\ldots,u^{i}_{\ell_i}\}$, where $u^1_1\in P^1_1\cup \{ x\}$, and where $u^i_j\in P^i_j$ if $i>1$ or $j>1$. 
We distinguish two cases.

First suppose that $x\notin W_1$, i.e., $u^1_1\ne x$. Then $W\subseteq V$, since $x$ is the only coordinate from $U$ not occurring in $V$. Since $W\subseteq S$, necessarily $W\subseteq S'=S\cap V$, thus $S'$ contains $W$, that is a copy of $M$. 

Now, suppose that $u^1_1=x$, 
this means $W_1=\{ x,\ldots,u^{i}_{\ell_i}\}$. Consider the minor-subproduct $M'=M(\Phi,U)$ of $U$ defined by the generalized partition $\Phi=(\phi_1,\ldots,\phi_m)$,  where $\phi_i=\beta_i$ if $i>1$ and
$\phi_1$ is obtained from $\beta_1$ by adding the trivial block $\{ x\}$: $\phi_1=\{ P^{1}_1,P^1_2,\ldots,P^{1}_{\ell_{1}}, \{ x\}\}$. Consequently,  $\phi_1,\ldots,\phi_m$ are quasi-trivial partitions of $U_1,\ldots,U_m$, respectively. Then $M'$ is an extended minor-subproduct of $U$, and we have $M'=M(\Phi,U)=M_1'\times M'_2\times\ldots\times M'_m$, where $M'_i=M_i$ if $i\ge 2$ and $M'_1=M_1\cup \{ w^1_{{\ell_1}+1}\}$, with $w_{\ell_1+1}^1$ corresponding to the coordinate $x$. 

\begin{claim}
 $M'$ is shattered by $S$ in $U$.   
\end{claim}

To prove the claim, pick any $t\in M'$. If $t\in M$, then $t$ has an expansion $v$ in $V$ belonging to $S'$, since $M$ is shattered in $V$ by the set $S'$. But then $v$ is also an expansion of $t$ in $U$ belonging to $S$, thus in $U$ we have  $F(t)\cap S\ne\varnothing$. Now suppose that $t\in M'\setminus M$. By the definition of $M'$ we conclude that $t$ is an $m$-tuple of the form $(w^1_{\ell+1},w_{j_2}^2,\ldots,w_{j_m}^m)$ for some $j_2,\ldots,j_m$. Since $S$ contains $W$ and since $x\in W_1$, we have $(x,u_{j_2}^2,\ldots,u_{j_m}^m)\in S$, where $(x,u_{j_2}^2,\ldots,u_{j_m}^m)\in F(t)$ is an expansion of $t$.   
Consequently, $M'$ is shattered by $S$ in $U$. 

Since $S$ is $\EMP(U)$-ample and $M'$ is an extended minor-subproduct of $U$,  $M'$ is strongly shattered in $U$ by $S$. Therefore $S$ contains a copy $W'=W'_1\times\ldots\times W'_m$ of $M'$. 
We set $W'_i=\{ u^{i}_1,\ldots,u^{i}_{\ell_i}\}$ for $i>1$ and  $W'_1=\{ u^{i}_1,\ldots,u^{i}_{\ell_i},x\}$, where $u^i_j\in P^i_j$ for all $i,j$. 
Consider the full-dimensional subproduct $W''=W''_1\times\ldots\times W''_m$, where $W''_i=W'_i$ if $i>1$ and $W''_1=\{ u^{i}_1,\ldots,u^{i}_{\ell_i}\}$. Then $W''\subset W'\subseteq S$. On the other hand, $W''\subseteq V$, hence $W''$ is a copy of  $M$ in $V$. Consequently, the set $S'=S\cap V$ contains a copy of $M$, which completes the proof that $S'$ is $\EMP(V)$-ample and the proof of the implication (2)$\Rightarrow$(3). 

We continue with the proof of (3)$\Rightarrow$(1).  For contradiction, assume that for $U=U_1\times\ldots\times U_m$  there exists $S\subseteq U$ for which (3) holds but not (1). Assume that we have picked a set $U$ with this property with the least number of elements. By the remark at the beginning of the proof of the theorem, we can suppose that all factors of $U$ are non-trivial. Since $S$ is not ample, there exists  $M\in\MP(U)$ that is shattered but not strongly shattered by $S$. 
Now suppose that there is a subfactor $M_i$ of $M$ that is neither trivial nor co-trivial, say without loss of generality  that this applies to $M_1$. Let $\Lambda=(\alpha_1,\ldots,\alpha_m)$ and $\alpha_1=\{P_1^1,P_2^1,\ldots,P_{\ell_1}^1\}$. Our assumption on $M_1$ tells us that $\ell_1>1$ and that there is some $P_i^1$ with $|P_i^1|>1$. Assume without loss of generality that $|P_1^1|>1$, and let $x,y\in P_1^1$ with $x\neq y$.

Now we remove $x$ from our universe. More precisely, let $U(x)=(U_1\backslash\{x\})\times U_2\times U_3\times\ldots\times U_m$, and let $\Lambda(x)\in \GPart(U(x))$ with $\Lambda(x)=(\beta_1,\ldots,\beta_m)$, where $\beta_1=\{P_1^1\backslash \{x\},P_2^1,P_3^1,\ldots,P_{\ell_1}^1\}$ and $\beta_i=\alpha_i$ for $i=2,3,\ldots,m$. Furthermore, let $M(x)=M(\Lambda(x),U(x))$ and $S(x)=S\cap U(x)$. Since $U(x)$ has strictly less elements than $U$ and by the choice of $U$, every $\EMP(U(x))$-ample subset of $U(x)$ is ample (since for $U(x)$ the implications $(2)\Rightarrow (3)\Rightarrow(1)$ hold). We distinguish three cases, and show that $M$ is strongly shattered by $S$ in all cases.

     \begin{case}\label{case:Mx}
         $M(x)$ is shattered by $S(x)$ in $U(x)$.
     \end{case}
     Since $U(x)$ is a full-dimensional subproduct of $U$, $S(x)$ is $\EMP(U(x))$-ample,  and therefore ample (by what we just observed). It follows that $M(x)$ is strongly shattered by $S(x)$, and thus $S(x)$ contains a copy $W=W_1\times\ldots\times W_m$ of $M(x)$. Any tuple $u=(u_1,\ldots,u_m)\in W$ is an expansion in $U(x)$ of a tuple $t'=(v^1_{j_1},\ldots v^m_{j_m})$ of $M(x)$, where each $u_i$  belongs to $P^i_{j_i}$ (or, if $i=1$ and $j_1=1$, to $P^1_1\setminus \{ x\}\subset P^1_1$). It follows that $u=(u_1,\ldots,u_m)$ is also an expansion in $U$ of the tuple $t=(w^1_{j_1},\ldots,w^m_{j_m})\in M$ corresponding to the tuple $t'\in M(x)$. Consequently, $W$  is a copy of $M$ in $U$, and we deduce that $M$ is strongly shattered in $U$. 

    Analogously to $U(x)$, we create $U(y)$ by removing $y$ from $U_1$; $M(y)$ and $S(y)$ are defined similarly to $M(x)$ and $S(x)$. 
    \begin{case}
        $M(y)$ is shattered by $S(y)$ in $U(y)$.
    \end{case}
    This case is completely analogous to \cref{case:Mx}.
    \begin{case}
        $M(x)$ is not shattered by $S(x)$ in $U(x)$, and $M(y)$ is not shattered by $S(y)$ in $U(y)$.
    \end{case}
    In this case, we create another full-dimensional subproduct of $U$, by removing all elements of $U_1\backslash P_1^1$. Let $U'=P_1^1\times U_2\times U_3\times \ldots \times U_m$. Let $\Lambda'=(\gamma_1,\ldots,\gamma_m)$, where $\gamma_1=\{P_1^1\}$, and $\gamma_i=\alpha_i$ for $i=2,3,\ldots,m$. Let $M'=M(\Lambda',U')$ and $S'=S\cap U'$.
    Note that the set of boxes of $M'$ is a strict subset of the set of boxes of $M$. Recall that $M$ being shattered by $S$ means that every box of $M$ contains an element of $S$. Therefore every box of $M'$ contains an element of $S'$, so $M'$ is shattered by $S'$. Furthermore,  $S'$ is $\EMP(U')$-ample  by condition (3), and since $U'$ has less elements than $U$, $S'$ is ample in $U'$. Thus, $M'$ is strongly shattered by $S'$. Let then $W'=W_1'\times W_2'\times \ldots \times W_m'$ be a copy of $M'$ in $U'$, where $W_i'=\{u_1^i,u_2^i,\ldots,u_{\ell_i}^i\}$. Since $P_1^1=U_1'$, $W_1'$ consists of a single element $u_1^1$. Since the labels of $x$ and $y$ could be interchanged, we may assume without loss of generality that $u_1^1\neq x$. Now we will establish a contradiction, by showing that $M(x)$ is shattered by $S(x)$ in $U(x)$.

    Pick any box $B=P_{j_1}^1\times P_{j_2}^2\times \ldots\times P_{j_m}^m$ of $M(x)$. If $j_1\neq 1$, then the entire box $B$ is contained in $U(x)$. In particular, since $M$ is shattered by $S$, $B$ contains an element of $S$, which is also in $S(x)$. On the other hand, if $j_1=1$, then the entire box $B$ is contained in $U'$. This means that $B$ contains an element $u$ of $W'$, namely $(u_1^1,u_{j_2}^2,u_{j_3}^3,\ldots,u_{j_m}^m)\in B$. In particular, since $u_1^1\neq x$, $u\in U(x)$ and $u\in S(x)$. So in both cases there is an element of $S(x)$ in $B$, hence $S(x)$ shatters $M(x)$. This proves that Case 3 is not possible.

    We conclude that our initial assumption on $M_1$ was incorrect. Therefore,  $M_1$ must be either trivial or co-trivial, which implies that $M_i$ is trivial or co-trivial for all $i$. In particular, this means that $M$ is an extended minor-subproduct. 
    Since $S$ is $\EMP(U)$-ample and $S$ shatters $M$, this implies that $M$ is strongly shattered by $S$. Since this holds for all $M$, we conclude $S$ is ample. Consequently, $U$ and $S$ that we defined at the start of the proof of this implication do not exist. This proves that $\EMP(V)$-ampleness in all full-dimensional subproducts $V$ implies ampleness of $S$ in $U$, hence (3)$\Rightarrow$(1). 

    The implications (4)$\Rightarrow$(1) and (5)$\Rightarrow$(1) are trivial. To prove the implication (1)$\Rightarrow$(4), let $S\subseteq U$ be ample and let $V$ be a full-dimensional subproduct of $U$. Then $S\cap V$ is $\EMP(V)$-ample  by the implication (1)$\Rightarrow$(3). But then $S\cap V$ is ample in $V$ by the implication (3)$\Rightarrow$(1) applied to $V$. Therefore $S$ is $\MP^*(U)$-ample. 
    
    It remains to prove the implication (4)$\Rightarrow$(5), which essentially follows from the definition of shattering and strong-shattering of arbitrary subproducts. Let $S$ be a $\MP^*(U)$-ample subset of $U$ and let $M\in \MP^{**}(U)$ be shattered by $S$. We assert that   $M$ is strongly shattered by $S$. By the definition, let $M\in \MP(V)$ for a subproduct $V$ of $U$. If $V$ is full-dimensional, then $M$ is strongly shattered because $S$ is $\MP^*(U)$-ample and we are done. Therefore, we can suppose that $V=V_{i_1}\times\ldots\times V_{i_k}$ is not full-dimensional, i.e. its support  $A=\{ i_1,\ldots,i_k\}$ is a proper subset of $X$. Let $M=M_{i_1}\times\ldots\times M_{i_k}$ be defined by the generalized partition $\Lambda=(\alpha_{i_1},\ldots,\alpha_{i_k})$ of $V$.  Consider the full-dimensional subproduct $V'=V'_1\times\ldots\times V'_m$, where $V'_i=V_i$ if $i\in A$ and $V'_i=U_i$ if $i\notin A$ and consider the generalized partition $\Lambda'=(\alpha'_1,\ldots,\alpha'_m)$ of $V'$, where $\alpha'_i=\alpha_i$ if $i\in A$ and $\alpha'_i=\{ U_i\}$ if $i\notin A$. Finally, let $M'=M(\Lambda',V')$ be the minor-subproduct of $V'$ corresponding to $\Lambda'$. Recall that $M$ is shattered by $S$ if and only if $M'\in \MP(V')$ is shattered by $S$. Since $V'$ is a full-dimensional subproduct and $S$ is $\MP^*(V')$-ample, we conclude that $M'$ is strongly shattered by $S$. But this implies that $M$ is strongly shattered by $S$, hence each $\MP^*(U)$-ample set $S$ is $\MP^{**}(U)$-ample. This concludes the proof of the theorem. 
\end{proof}

\begin{proposition} \label{prop:ample-is-isometric} Every ample set $S\subseteq U$ is isometric. 
\end{proposition} 

\begin{proof}  Assume for contradiction that there exists an ample set $S\subseteq U$ that is not isometric. Take the product $U$ with the smallest number of elements with this property. We may assume each factor of $U$ is non-trivial by \cref{lem:trivial-factor}. Since $S$ is not isometric in $U$, there exist $x,y\in S$ such that $d_H(x,y)>1$ and $[x,y]\cap S=\{ x,y\}$. Suppose that there is a coordinate $i$ such that $U_i\neq \{x_i,y_i\}$, say we have $x_m'\in U_m\setminus\{x_m,y_m\}$. Let  $U'=U_1\times U_2\times\ldots\times U_{m-1}\times (U_m\setminus \{ x_m'\})$. Then the set $S'=S\cap U'$ is ample in $U'$ and $U$ (by \cref{thm:ample=weaklyample}). By the minimality choice of $U$, $S'$ is isometric in $U'$. Since $x,y\in S'$, there exists $z\in [x,y]\cap S', z\ne x,y$. But this is impossible, since $U'$ is a convex subset  of $U$: $z$ must belong to $S$ and to the interval between $x$ and $y$ in $H(U)$, contrary to our assumption that $[x,y]\cap S=\{ x,y\}$. So there is no such $i$. Then there is also no coordinate $i$ where $x_j=y_j$, as this would result in a trivial factor $U_i$. Hence each factor $U_i$ is binary, i.e., $U_i=\{ x_i,y_i\}, i=1,\ldots,m$. 

Therefore $[x,y]=U$, whence $S=\{ x,y\}$. Let $M=M(\Lambda,U)$, where $\Lambda=(\alpha_1,\ldots,\alpha_m)$, with $\alpha_1=\{ \{ x_1\},\{ y_1\}\}$ and $\alpha_i=\{ \{U_i\}\}$ if $i>1$. 
Then $M=M(\Lambda,U)$ is the Cartesian product $\{ w^1_1,w^1_2\}\times \{ w^2\}\times\ldots \{ w^m\}$ (i.e., $M$ is the product of an edge and one-vertex factors), which consists of the $m$-tuples $t_1=(w^1_1,w^2,\ldots,w^m), t_2=(w^1_2,w^2,\ldots,w^m)$. Then $t_1$ has $x$ as expansion and $t_2$ has $y$ as expansion, thus  $S$ shatters the minor-subproduct $M$. On the other hand, $S$ does not contain any copy of $M$, because any such copy will be an edge  with endpoints differing only in the first coordinate. This shows that any ample set $S\subseteq U$ is isometric.  
\end{proof}

\section{Restrictions, projections, and strong-projections}\label{sec:projections}
Several operations on subsets of binary products preserve ampleness \cite{BaChDrKo}. We generalize these operations to subsets of arbitrary products, and prove some of their basic properties.

\begin{definition}[Complements, restrictions, projections, and strong-projections] Let  $S$ be a subset of  $U=U_1\times\ldots\times U_m$. First, set $S^*=U\setminus S$ and call $S^*$ the \emph{complement} of $S$. For a full-dimensional subproduct $V\subseteq U$, we call the intersection $S\cap V$ the \emph{restriction} of $S$ to $V$. For a minor-subproduct $M=M(\Lambda,U)$, let  $S_M=\{ t\in M: F(t)\cap S\ne \varnothing\}$ and $S^M=\{ t\in M: F(t)\subseteq S\}$ (recall that the fiber $F(t)$ of $t\in M$ consists of all expansions of $t$ in $U$ and coincides with a box of $U$). The sets $S_M$ and $S^M$ are called the \emph{projection} and the \emph{strong-projection} of $S$ on $M$, respectively. 
\end{definition}

\begin{example}\label{ex:binaryprojection} In case of subsets $S$ of binary products $U=\{ a,b\}^X$, each minor-subproduct $M$ is binary, say with support $X\setminus Y$. Then the sets $S_M$ and $S^M$ are defined uniquely by the set $Y$, and can be denoted by $S_Y$ and $S^Y$, respectively, which coincides with the definition from  \cite{BaChDrKo}. 
\end{example} 

\begin{example}\label{ex:projection}
An example of the projection and strong-projection operators is shown in \cref{fig:projections}. 
\end{example}

\begin{figure}[htbp]
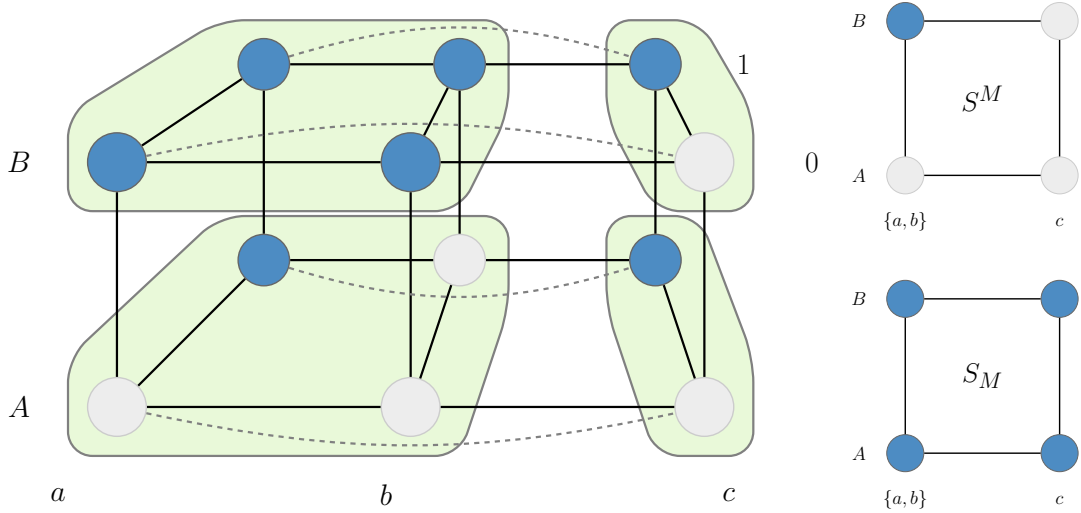

    \centering
    \includestandalone[width=0.7\linewidth]{img/projections1}
    \includestandalone[width=0.2\linewidth]{img/projections2}
    \caption{Left: box partition related to the minor-subproduct $M$ from \cref{fig:shatterstrongshatter}. Right: the sets $S^M$ and $S_M$. Since all vertices have their third coordinate equal to $\{0,1\}$, this coordinate is not shown.}
    \label{fig:projections}
\end{figure}

To be able to compose the operators  $S^M$ and $S_M$, the operators need to be well-defined not only for subsets $S$ of $U$, but also for subsets $S$ of minor-subproducts $M'$ of $U$. For this, we need to adapt the notion of fibers. Recall that a fiber in $U$ of an element $t$ of $M$ consists of the elements of $U$ that are in the box corresponding to $t$. Likewise, the fiber of an element of a minor-subproduct $M$ in a minor-subproduct $M'$ that refines $M$  consists of the blocks of $M'$ that together form a block of $M$. This is made more formal in the following definition.
\begin{definition}[Fibers relative to minor-subproducts]
    Let $M=M(\Lambda,U)$ and $M'=M(\Lambda',U)$ where $M'\preceq M$. Let $M=M_1\times\ldots\times M_m$ with $M_i=\{w_1^i,\ldots,w_{\ell_i}^i\}$, and let $M'=M_1'\times\ldots\times M_m'$ with $M_i'=\{v_1^i,\ldots,v_{\ell_i'}^i\}$. Suppose $\Lambda=(\alpha_1,\ldots,\alpha_m)$ with $\alpha_i=\{P_1^i,\ldots,P_{\ell_i}^i\}$, and $\Lambda'=(\beta_1,\ldots,\beta_m)$ with $\beta_i=\{Q_1^i,\ldots,Q_{\ell_i'}^i\}$. A tuple $u=(u_1,u_2,\ldots,u_m)\in M'$ is called an $M'$-\emph{expansion} of a tuple $t=(t_1,t_2,\ldots,t_m)\in M$ if for each $i\in X$, if $t_i=w_j^i$ and $u_i=v_{k}^i$, then $Q_{k}^i\subseteq P_{j}^i$. The $M'$-\emph{fiber} of $t$ is denoted by $F_{M'}(t)$, and is the set of all $M'$-expansions of $t$. The $U$-fiber $F_U(t)$ coincides with the fiber $F(t)$.
\end{definition}

This allows us to extend the definition of the operators $S^M$ and $S_M$ to subsets of minor-subproducts $M'$ and arbitrary pairs $M,M'$. 
\begin{definition}[Projections and strong-projections, bis]\label{def:operatorsextended} 
    Let $M,M'\in \MP(U)$ and let $S\subseteq M'$. Then the set  $S_M=\{t\in M\vee M':F_{M'}(t)\cap S\neq \varnothing\}$ is called the \emph{projection} of $S$ on $M$ and the set 
    $S^M=\{t\in M\vee M':F_{M'}(t)\subseteq S\}$ is called the \emph{strong-projection} of $S$ on $M$.
\end{definition}

\begin{remark} In Definition \ref{def:operatorsextended}  we make no assumptions on how $M,M'$ relate with respect to $\preceq$. Also, note that the operation implicitly depends on $M'$, which we can consider the relative universe of $S$. For our purposes, the set $M'$ is always clear from the context, hence we do not write it down.
\end{remark}


\begin{lemma}\label{lem:projections-and-complements}  Let $M,M'\in \MP(U)$ and  $S\subseteq M'$. Then 
    $(S^*)^M=(S_M)^* \mbox{ and } (S^*)_M=(S^M)^*.$ 
\end{lemma}
\begin{proof}
    We have 
    \begin{eqnarray*}
    (S^*)^M&=&\{t\in M\vee M':F_{M'}(t)\subseteq S^*\}\\
    &=&\{t\in M\vee M':F_{M'}(t)\cap S = \varnothing\} = (S_M)^*\\
    (S^*)_M&=&\{t\in M\vee M':F_{M'}(t)\cap S^*\neq \varnothing\}\\
    &=& \{t\in M\vee M':F_{M'}(t)\nsubseteq S\} = (S^M)^*.
\end{eqnarray*}
This concludes the proof. 
\end{proof}

With \cref{def:operatorsextended}, we can compose arbitrary sequences of projections and strong-projections. For example, if we have $S\subseteq U$, then $(S^M)^{M'}$, $(S_M)^{M'}$, $(S^M)_{M'}$, and $(S_M)_{M'}$ are all well-defined and are subsets of $M\vee M'$.

\begin{lemma}\label{lem:operatorsselfcommutative}
    Let $M,M',M''\in \MP(U)$ and  $S\subseteq M''$. Then $(S^{M})^{M'}=S^{M\vee M'}=(S^{M'})^{M}$, and $(S_{M})_{M'}=S_{M\vee M'}=(S_{M'})_{M}$.
\end{lemma}
\begin{proof}
    First of all, if we have an element $r\in M\vee M'\vee M''$, then we can expand it to an element $t$ of $M\vee M''$, and then expand it to an element of $M''$. In this way, we can get all expansions of $r$ in $M''$. So we have $\bigcup_{t\in F_{M\vee M''
    }(r)} F_{M''}(t)=F_{M''}(r)$. We get the following equalities, where the equation marked with $*$ uses the above observation:
    \begin{eqnarray*}
        (S_{M})_{M'}&=&(\{t\in M\vee M'':F_{M''}(t)\cap S\neq \varnothing\})_{M'}\\
        &=&\{r\in M\vee M'\vee M'':\exists t\in F_{M\vee M''}(r):F_{M''}(t)\cap S\neq \varnothing\}\\
        &\stackrel{*}{=}& \{r\in M\vee M'\vee M'':F_{M''}(r)\cap S \neq \varnothing\} = S_{M\vee M'}
    \end{eqnarray*}
    Analogously, $(S_{M'})_M=S_{M\vee M'}$. Applying \cref{lem:projections-and-complements} then tells us
    \[
    (S^{M})^{M'}=((S^{**})^M)^{M'}=(((S^*)_M)^*)^{M'}=(((S^*)_M)_{M'})^*=((S^*)_{M\vee M'})^*=(S^{M\vee M'})^{**}=S^{M\vee M'}
    \]
    and likewise $(S^{M})^{M'}=S^{M\vee M'}$.
\end{proof}

We also need a more specific definition of weak isometricity. Recall that a set $S\subseteq U$ is weakly isometric if $[x,y]\cap S\ne \{ x,y\}$ for all $x,y\in S$ with $d_{H(U)}(x,y)=2$. Then, given $k,\ell\in X$ with $k\neq \ell$, we say that the set $S\subseteq U$ is \emph{weakly $(k,\ell)$-isometric} if for all $x,y\in S$ with $d_{H(U)}(x,y)=2$ and which differ only in their $k$- and $\ell$-coordinates, we have $[x,y]\cap S\ne \{ x,y\}$.

\begin{lemma} \label{lem:convexity-preimages} Let $M\in \MP(U),$ $S\subseteq U$, and $C\subseteq M$. Let $k,\ell\in X$ with $k\neq \ell$. Then the following properties hold:
\begin{itemize}
\item[(1)] if $C$ is  convex in $M$, then $C^+=\bigcup_{t\in C} F(t)$ is convex in $U$;
\item[(2)] if $S$ is  isometric in $U$, then $S_M$ is isometric in $M$; 
\item[(3)] if $S$ is weakly $(k,\ell)$-isometric and $\supp(M)\subseteq \{k,\ell\}$, then $S_M$ and $S^M$ are weakly $(k,\ell)$-isometric.
\end{itemize}
Consequently, if  $M,M'\in \MP(U),$  $S\subseteq M'$ is isometric (weakly $(k,\ell)$-isometric) in $M'$, and $C\subseteq M\vee M'$ is convex in $M\vee M'$, then $S_M$ is isometric (weakly $(k,\ell)$-isometric) in $M\vee M'$ and $C^+=\bigcup_{t\in C} F_{M'}(t)$ is convex in $M'$.
\end{lemma}

\begin{proof} In all parts of the proof, suppose that $M=M(\Lambda, U)$, where $\Lambda=(\alpha_1,\ldots,\alpha_m)$ and $\alpha_i=\{P^i_1,\ldots,P^i_{\ell_i}\}$ for $i=1,\ldots,m$. 

To prove (1), let $C$ be a convex set of $M$. By Lemma \ref{convex-Hamming} applied to $M$, $C$ is a full-dimensional product of $M$, say $C=C_1\times\ldots\times C_m$. Then the union 
$C^+=\bigcup_{t\in C} F(t)$ coincides with the full-dimensional subproduct $P^1\times\ldots\times P^m$ of $U$, where $P^i$ is the union of all blocks of the partition $\alpha_i$ that correspond to the $i$th coordinate $t_i$ of some $t=(t_1,\ldots,t_m)\in C$. By Lemma \ref{convex-Hamming} applied to $U$, $C^+$ is a convex set of $U$ and we are done. 

To prove (2), let $S$ be an isometric set of $U$ and pick any $u,v\in S_M$ with
$k:=d_{H(M)}(u,v)\geq 2$. Let $u=(u_1,\ldots,u_m), v=(v_1,\ldots,v_m)$ be the coordinates of $u$ and $v$ in $M$. Suppose without loss of generality that that $u$ and $v$ differ in the first $k$ coordinates, i.e., $u_{k+1}=v_{k+1},\ldots,u_m=v_m$. This implies that $C:=[u,v]=\{u_1,v_1\}\times\ldots\times \{ u_k,v_k\}\times\{ u_{k+1}\}\times\ldots\times \{ u_m\}$. 
By Lemma \ref{convex-Hamming},  the set $C:=[u,v]$ is convex in $M$, thus by assertion (1)  its pre-image $C^+$ in $U$ is also convex. Pick any  $u'\in F(u)\cap S$ and $v'\in F(v)\cap S$ (they exist because $u,v\in S_M$). Since $S$ is isometric in $U$, there exists a shortest $(u',v')$-path $P$ included in $S$ (i.e., a path of length $d(u',v')=d_{H(U)}(u',v')$). Since $u',v'\in C^+$, the path $P$ starts in $F(u)$ and ends in $F(v)$. Since the fibers $F(u)$ and $F(v)$ are not adjacent (because $u$ and $v$ are not adjacent in $M$), and the fibers of vertices of $M$ define a partition of $U$, the path $P$ traverses a fiber $F(x)$ with $x\ne u,v$. Let $x'\in P\cap F(x)$. Since $C^+$ is convex and $P$ is a shortest $(u',v')$-path, necessarily 
$P\cap F(x)\subseteq C^+$. Since the fibers of  $M$ define a partition of $U$, this implies that $x\in C=[u,v]$. Since 
$x\ne u,v$ and $x'\in P\cap F(x)\subseteq S\cap F(x)$, we deduce that $x\in S_M\cap ([x,y]\setminus \{ u,v\})$. Consequently, for any pair $u,v\in S_M$ there exists $x\in S\cap [u,v], x\ne u,v$ and this easily implies 
isometricity by applying induction on $k=d_{H(M)}(u,v)$.

Now we prove assertion (3). Suppose by way of contradiction that we have some $S\subseteq U$ that is weakly $(k,\ell)$-isometric, and that $S^M$ or $S_M$ is not weakly ($k,\ell$)-isometric. Because of \cref{lem:atomistic}, we know that $M$ can be split up into atoms of $\GPart(U)$. Combined with \cref{lem:operatorsselfcommutative} this means that we can obtain the projection $S_M$ by projecting $S$ onto each of its atoms.  Similarly, $S^M$ is the result of a number of strong-projections of atoms. Thus, we may assume without loss of generality that $M$ is elementary, since if the lemma holds for all elementary $M$, then the lemma holds by induction for all $M$ with $\supp(M)\subseteq\{k,\ell\}$. 

Assume that $S^M$ is not weakly $(k,\ell)$-isometric. Then there must exist $u,v\in S^{M}$ differing in exactly two coordinates, say in their $k$ and $\ell$-coordinates, such that both their common neighbors $w,z$ in $M$ do not belong to $S^M$. Then,  $[u,v]\backslash S^{M}=\{w,z\}$. Note that all boxes of $M$ have 1 or 2 elements, and now we consider $|F(u)|$ and $|F(v)|$. It is not possible that $|F(u)|=|F(v)|=2$, since $u,v$ differ in both the $k$ and $\ell$-coordinate and $M$ is elementary. If $|F(u)|=|F(v)|=1$ then $[u,v]\subseteq U$, implying $S$ was not weakly $(k, \ell)$-isometric. Thus, one of $F(u)$ or $F(v)$ must have size 2 and the other size 1. Because of shared coordinates, the same holds for $F(w)$ and $F(z)$. Assume w.l.o.g. $|F(u)|=|F(z)|=2$ and $|F(v)|=|F(w)|=1$. There must be some $z'\in F(z)$ that is not in $S$, since $z\notin S^{M}$. Likewise, there is a neighbor of $z'$ that is in $F(u)$, call it $u'$. We have $u'\in S$, since $u\in S^M$. Let $F(v)=\{v'\}$ and $F(w)=\{w'\}$. Then $u',v'\in S$, $z',w'\notin S$, and $[u', v']=\{u',v',w',z'\}$, which contradicts our assumption that $S$ was weakly $(k,\ell)$-isometric.

On the other hand, suppose $S_M$ is not weakly $(k,\ell)$-isometric. By Lemma  \ref{lem:projections-and-complements} we have $S_{M}=((S^*)^{M})^*$. Since $S^*$ is weakly $(k,\ell)$-isometric, by the previous case, $(S^*)^M$ is also weakly $(k,\ell)$-isometric, hence $((S^*)^{M})^*=S_M$ is weakly $(k,\ell)$-isometric, which is a contradiction. This completes the proof of assertion (3).
    
The final assertion of the lemma follows by applying the assertions (1),(2),(3) with $M'$ instead of $U$ and $M\vee M'$ instead of $M$. 
\end{proof}


\begin{lemma}\label{lem:superscriptample}
    Let $S\subseteq U$ be ample in $U$ and $M\in \MP(U)$. Then $S^M$ is ample in $M$.
\end{lemma}
\begin{proof}
    Due to \cref{lem:atomistic,lem:operatorsselfcommutative}, it suffices to prove the statement for elementary $M$, as the result for general $M$ would follow immediately by induction. By \cref{thm:ample=weaklyample}, we only need to show $\EMP(M)$-ampleness of $S^M$. Without loss of generality, let $M_1$ be the non-trivial factor of $M$, let $\alpha_1=\{P_1^1,\ldots,P_{\ell_1}^1\}$ be its related partition, and set $P_1^1=\{a,b\}$. Let $E\in \MP(U)$ such that $E\succeq M$ and in particular $E\in \EMP(M)$, and $E$ is shattered by $S^M$ in $M$.
    Let $\beta_1=(Q_1^1,\ldots, Q_{\ell_1'}^1)$ be the partition related to the first factor of $E$, such that $Q_2^1,\ldots,Q_{\ell_1'}^1$ are all trivial blocks. Since $E\succeq M$, $a$ and $b$ are either both in $Q_1^1$ or both not in $Q_1^1$. We distinguish two cases:

    \begin{case}  $a,b\notin Q_1^1$.
    \end{case}
    
    In this case, the partition corresponding to the first coordinate of $E$ consists of $Q_1^1$, $\{a,b\}$ and trivial blocks. Let $E'\in \MP(U)$ be the same as $E$, except that its generalized partition has blocks $\{a\}$ and $\{b\}$ instead of $\{a,b\}$. We argue that $S$ is shattered by $E'$. Pick any $t=(t_1,\ldots,t_m)$ in $E'$.  If the first coordinate of $t$ is not $\{a\}$ or $\{b\}$, then $t\in E$, and since $E$ is shattered by $S^M$ in $M$, that means $F_M(t)\cap S^M\neq \varnothing$. Moreover, then $F_M(t)=F(t)$, thus $F(t)\cap S\neq \varnothing$. 
    If on the other hand the first coordinate of $t$ is $\{a\}$ or $\{b\}$, then let $t'=(\{a,b\},t_2\ldots,t_m)\in E$. There exists $t''\in F_M(t')\cap S^M$, since $E$ is shattered by $S^M$. Then $|F(t'')|=2$ and $F(t'')\subseteq S$, say $F(t'')=\{t^a,t^b\}$, with $t^a,t^b$ having respective coordinates $a$ and $b$. If the first coordinate of $t$ was $\{a\}$, then $t^a\in F(t)\cap S$, and otherwise $t^b\in F(t)\cap S$. In all cases, every box of $E'$ contains an element of $S$, so $E'$ is shattered by $S$. It is then also strongly shattered by $S$. By ampleness, $S$ contains a copy $W=W_1\times \ldots \times W_m$ of $E'$. Let $W'=(W_1\backslash \{a,b\}\cup\{\{a,b\}\})\times W_2\times \ldots\times W_m$ be created from $W$ by merging $a$ and $b$ into a single element $\{a,b\}$ (note that $a,b\in W_1$ as $\{a\},\{b\}$ are blocks in the first partition of $E'$). Then $W'\subseteq M$, and $W'$ is a copy of $E$ in $M$. We claim that $W'\subseteq S^M$. For all the elements that do not have an $\{a,b\}$-coordinate this is trivial, since they are also elements of $S$. Now consider $w\in W'$ and assume that $w=(\{a,b\},w_2,\ldots, w_m)$. Then the tuples $(a,w_2,\ldots,w_m)$ and $(b,w_2,\ldots,w_m)$ are in $S$, since they were in $W$, and therefore $w$ is in $S^M$. In conclusion, $S^M$ is strongly shattered by $E$.
    
    \begin{case} $a,b\in Q_1^1$.
    \end{case} 

    Similarly to the previous case, we can  show that $S$ is shattered by $E$. Then, by ampleness, $S$ contains a copy $W=W_1\times \ldots\times W_m$ of $E$. If there exists such a copy $W$ of $E$ such that $a,b\notin W_1$,  then $W$ is also contained in $S^M$, which implies that $S^M$ is strongly shattered by $E$. Therefore, for the rest of the proof we can suppose that any copy $W$ of $E$ in $S$ contains either $a$ or $b$ in $W_1$. First, we show that $S$ contains two copies $W^a$ and $W^b$ of $E$, where $W^a$ contains $a$ in its first subfactor, and $W^b$ has $b$ in the first subfactor.
    
    Let $V=(U_1\backslash \{b\})\times U_2\times\ldots\times U_m$ and $E'\in \MP(V)$ be the restriction of $E$ to $V$.  We assert that $S\cap V$ is shattered by $E'$. Pick any $t\in E'$. If the first coordinate of $t$ is not $Q_1^1\backslash \{b\}$, then $t$ is an element of $E$. Since $F_{M}(t)$ contains an element of $S^M$ (as $S^M$ is shattered by $E$), and in this case $F_M(t)=F(t)$, the fiber $F(t)$ of $t$ also contains an element of $S\cap V$.

    On the other hand, if $t$ has first coordinate $Q_1^1\backslash \{b\}$, say $t=(Q_1^1\backslash \{b\}, t_2,\ldots, t_m)$, then let $t'\in E$ with $t'=\{Q_1^1,t_2,\ldots,t_m\}$. Since $E$ is shattered by $S^M$, there exists $t''\in F_M(t')\cap S^M$. If the first coordinate of $t''$ is not $\{a,b\}$, then $t''\in V$, and therefore $t''\in F(t)\cap(S\cap V)$. Otherwise, if the first coordinate of $t''$ is $\{a,b\}$, say $t''=(\{a,b\},t_2'',\ldots,t_m'')$. Since $t''\in S^M$, we deduce that $t''':=(a,t_2'',\ldots,t_m'')\in S$, hence $t'''\in F(t)\cap (S\cap V)$. So in all cases the boxes of $E'$ contain an element of $S\cap V$, thus $S\cap V$ is shattered by $E'$.

    By \cref{thm:ample=weaklyample} we know that $S\cap V$ is ample in $V$, and thus $E'$ is strongly shattered by $S\cap V$. In particular, $S\cap V$ contains a copy $W^a$ of $E'$, which is also a copy of $E$. Because of our assumption on copies of $E$, since the first subfactor of $W^a$ cannot contain $b$, but must have an element of $Q_1^1$, it follows that $W^a$ has $a$ in its first subfactor. With an analogous argument (by removing $a$ from $U_1$) we can find a copy $W^b$ of $E$ that contains $b$.

    Now we create $E(b)\in \MP(U)$ from $E$, by replacing the first partition $\beta_1$ by the partition $\{Q_1^1\backslash \{b\},\{b\},Q_2^1,\ldots,Q_{\ell_1'}^1\}$. We claim that $E(b)$ is shattered by $S$. Let $t\in E(b)$. If $t$ has as first coordinate $\{b\}$, then $F(t)$ contains an element of $W^b$. On the other hand, if the first coordinate of $t$ is $Q_1^1\backslash \{b\}$, then $F(t)$ contains an element of $W^a$. Finally, if $t$ has any other first coordinate, then $t$ is also an element of $E$, and $F(t)$ contains an element of $S$ because $S^M$ shatters $E$. So $E(b)$ is shattered by $S$. In particular, it is strongly shattered by $S$ since $S$ is ample, thus let $W'=W_1'\times\ldots\times W_m'$ be a copy of $E(b)$ in $S$. Note that, because of our assumption on $W$, $W'_1$ must contain $a$: if $W_1'$ had any other element from $Q_1^1\backslash \{b\}$, then $(W_1'\backslash \{b\})\times W_2'\times \ldots\times W_m'$ would give us a copy of $E$ in $S^M$, since $(W_1'\backslash \{b\})$ contains neither $a$ nor $b$. Thus $a,b\in W_1'$. Let $W''=W_1'\backslash \{a,b\}\cup \{\{a,b\}\}\times W_2'\times\ldots\times W_m'$. We claim that $W''\subseteq S^M$. Pick any $t\in W''$. If its first coordinate is not $\{a,b\}$, then $t\in W'\subseteq S$ and therefore $t\in S^M$. Otherwise, the two elements of $F(t)$ are both in $W'$, so they are both in $S$, and therefore $t\in S^M$. We conclude that $W''$ is a copy of $E$ in $S^M$, and therefore $E$ is strongly shattered by $S^M$ and $S^M$ is ample.
    \end{proof}

\section{Main characterizations of ampleness}\label{sec:maincharacterization}
In this section we characterize ampleness via commutativity, superconnectivity, superisometricity, and complement. We also characterize ample sets using elementary minor-subproducts. Our results generalize analogous characterizations  of ample sets in binary products given in  \cite{BaChDrKo}. We continue with the main definitions used in this section. Let $U=U_1\times \ldots\times U_m$. 

\begin{definition}[Commutativity]
    A set $S\subseteq U$ is \emph{commutative} if for all $M,M'\in \MP(U)$ with $\supp(M)\cap\supp(M')=\varnothing$ we have $(S_M)^{M'}=(S^{M'})_M$.
\end{definition}

 \begin{definition}[Superisometricity/superconnectivity]
        A set $S\subseteq U$ is  \emph{superisometric} (respectively, \emph{superconnected}) if $S^M$ is isometric (respectively, connected) for all $M\in \MP(U)$.
 \end{definition}


\subsection{Weak isometricity and commutativity} We continue with two auxiliary lemmas, linking weak isometricity of sets with weaker versions of cummutativity. We also establish some properties of commutative sets.

 \begin{lemma}\label{lem:weakisometricity}
    Let $M''\in MP(U)$ and $k,\ell\in X$ with $k\neq \ell$. For a set $S\subseteq M''$ the following conditions are equivalent: 
     \begin{enumerate}
         \item $S$ is weakly $(k,\ell)$-isometric; 
         \item $S^*$ is weakly $(k,\ell)$-isometric;
         \item $(S^M)_{M'}=(S_{M'})^M$ for all elementary $M,M'\in \MP(U)$ with $\supp(M)=k$ and $\supp(M')=\ell$;
         \item $(S^M)_{M'}=(S_{M'})^M$ for all one-dimensional $M,M'\in \MP(U)$ with $\supp(M)=k$ and $\supp(M')=\ell$.
     \end{enumerate}
 \end{lemma}
\begin{proof} The equivalence (1)$\Longleftrightarrow$(2) immediately follows from the definition of weak $(k,\ell)$-isometricity. Before proving the other implications, we adopt the following convention for elementary minor-subproducts. Let $M=M_1\times\ldots\times M_m$ be an elementary minor-subproduct. Then all its factors are trivial except for the $k$-th factor. Suppose without loss of generality that this factor is $M_1$, i.e., $k=1$. By definition of an elementary minor-subproduct, the partition $\alpha_1=\{P_1^1\ldots,P_{\ell_1}^1\}$ corresponding to $M_1$ 
is an atom of the lattice $\Part(U_1)$. This means that  all of its blocks are trivial except one, which contains two elements, i.e., without loss of generality, we can assume that $P_1^1=\{a,b\}$. Likewise, we assume that the elementary minor-subproduct $M'=M'_1\times M'_2\times\ldots\times M'_m$ has as non-trivial factor $M_2'$ (recall that $\ell\ne k$), whose partition $\beta_2=\{Q_1^2,\ldots,Q_{\ell_2}^2\}$ contains one non-trivial block $Q_1^2=\{c,d\}$.
     
Next  we establish the implication (1)$\Rightarrow$(3). Suppose $S$ is weakly $(k,\ell)$-isometric. Pick any $u=(u_1,u_2,\ldots,u_m)\in M\vee M'$. If $u_1$ corresponds to $P_1^1=\{ a,b\}$ and $u_2$ to $Q_1^2=\{c,d\}$, and $M,M'\npreceq M''$, then the $M''$-fiber $F_{M''}(u)$ of $u$ is a square of the Hamming graph $H(M'')$.  
Weak  $(k,\ell)$-isometricity of $S$ in $M''$ implies  that the intersection of $S$ and $F_{M''}(u)$ is not equal to two opposite corners of the square $F_{M''}(u)$. Therefore, the intersection $S \cap F_{M''}(u)$ is either empty, a single vertex, two adjacent vertices, three vertices of $F_{M''}(u)$, or the whole square $F_{M''}(u)$. This can be equivalently rewritten in the following form:  $u$ belongs to $(S^M)_{M'}$ if and only if $u$ belongs to $(S_{M'})^M$. In all other cases, $u\in M\vee M'$ have $M''$-fibers  $F_{M''}(u)$ that have size size 1 or 2 (they are vertices or edges of $H(M'')$), and one can easily verify that for them we also have that $u\in (S^M)_{M'}$ if and only if $u\in (S_{M'})^M$. This establishes the equality $(S^M)_{M'}=(S_{M'})^M$. 

To prove the converse implication (3)$\Rightarrow$(1), suppose that $S\subseteq M''$ satisfies (3) but $S$ is not weakly $(k,\ell)$-isometric. Then there are two vertices $u,v$ of $S$ at distance $2$ in $H(M'')$ such that $S\cap [u,v]=\{u,v\}$ and $u$ and $v$ differ in the coordinates $k$ and $\ell$. Suppose without loss of generality that $k=1, \ell=2$ and that $u=(a,c,u_3,u_4,\ldots,u_m)$ and $v=(b,d,u_3,u_4,\ldots,u_m)$. 
Then, let $M=M(\Lambda,U)$ be elementary, such that the two elements from its nontrivial block are in $a$ and $b$, respectively (then $M\vee M''$ is elementary in $M''$ with nontrivial block $\{a,b\}$). Likewise, let  $M'=M(\Lambda',U)$ be elementary, such that the unique nontrivial block of $M'\vee M''$ is $\{c,d\}$. 
Let $w\in M\vee M'$ be the element with $u\in F_{M''}(w)$, then $w\in (S_{M'})^M$ but $w\notin (S^M)_{M'}$, contrary to (3). This shows that $S$ is weakly $(k,\ell)$-isometric. 

     The direction (4)$\Rightarrow$(3) is trivial, and finally we show the implication (3)$\Rightarrow$(4). By the equivalence between the conditions (1), (2), and (3), if (3) holds, then $S$ and $S^*$ are weakly $(k,\ell)$-isometric. Let $M,M'\in \MP(U)$ be one-dimensional with disjoint supports. By \cref{lem:atomistic} we can write $M=\bigvee_{i=1}^{r} M^i$ and $M'=\bigvee_{j=1}^{r'}(M')^j$, where $M^i$ and $(M')^j$ are elementary minor-subproducts for all $i=1,\ldots,r$ and $j=1,\ldots,r'$. We clearly also have $\supp(M)=\supp(M^i)$ for all $i=1,\ldots,r$ and $\supp(M')=\supp((M')^j)$ for all $j=1,\ldots,r'$. 
     This means that, by \cref{lem:operatorsselfcommutative}, we have
     \begin{eqnarray}
            (S^{M})_{M'}&=&\left(\ldots\left(\left(\ldots\left(\left(S^{M^1}\right)\ldots\right)^{M^{r}}\right)_{(M')^1} \right)\ldots\right)_{(M')^{r'}}\label{eq:onedimcommute1}
     \end{eqnarray}   
    and
     \begin{eqnarray}
            (S_{M'})^M &=& \left(\ldots\left(\left(\ldots\left(\left(S_{(M')^1} \right)\ldots\right)_{(M')^{r'}}\right)^{M^1}\right)\ldots\right)^{M^r}.\label{eq:onedimcommute2}
        \end{eqnarray}
    It follows from \cref{lem:convexity-preimages} that all sets between the brackets in \eqref{eq:onedimcommute1} and \eqref{eq:onedimcommute2} are weakly $(k,\ell)$-isometric. That implies that all the operations in the equations \eqref{eq:onedimcommute1} and \eqref{eq:onedimcommute2} commute, since we already established the equivalence (1)$\Longleftrightarrow$(3). It then follows that $(S^M)_{M'}=(S_{M'})^M$.
 \end{proof}

\begin{lemma}\label{lem:isometricityclassified}
     Let $S\subseteq M''\in \MP(U)$, then the following conditions are equivalent:
     \begin{enumerate}
         \item $S$ is isometric in $M''$;
         \item $S_M$ is weakly isometric for all $M\in \MP(U)$;
         \item $(S^M)_{M'}=(S_{M'})^M$ for all $M'\in \MP(U)$ and all one-dimensional $M\in \MP(U)$ with $\supp(M)\cap\supp(M')=\varnothing$.
     \end{enumerate}
 \end{lemma}

 \begin{proof} First we prove (1)$\Rightarrow$(2). Since any isometric set $S$ is weakly isometric, we can apply Lemma \ref{lem:convexity-preimages} and conclude that $S_M$ is weakly isometric in $M\vee M''$ for all $M\in \MP(U)$.
     To show (2)$\Rightarrow$(1), suppose $S$ is not isometric in $M''$. This means that there are $u,v\in S$ with $d(u,v)\geq 2$ such that $[u,v]\cap S=\{u,v\}$. Suppose we have chosen $M''$, $S$, $U$, $u$ and $v$ satisfying (2) such that $d(u,v)=d_{H(M'')}(u,v)$ is minimal. Since $S_M$ is weakly isometric for all $M\in \MP(U)$ and $M''\in \MP(U)$,  clearly $S$ is weakly isometric, and thus $d(u,v)\geq 3$. Then, take an elementary $M'\in \MP(M'')$ with its only non-trivial block equal to $\{a,b\}$, where $a,b$ are two coordinates of $u,v$ where they differ. Then take $u',v'\in S_{M'}$ such that $u\in F_{M''}(u')$ and $v\in F_{M''}(v')$. We have $d_{H(M')}(u',v')=d(u,v)-1$ and $[u',v']\cap S_{M'}=\{u',v'\}$, so $S_{M'}$ is not isometric, while $(S_{M'})_M=S_{M'\vee M}$ is weakly isometric for all $M$. Thus we obtain a contradiction with the minimality choice of $M''$, $S$, $U$, $u$ and $v$, hence $S$ is isometric.

     We then prove (2)$\Rightarrow $(3) by induction on $|\supp(M')|$. For $|\supp(M')|=0$ this is trivial. Suppose we have shown the statement for $|\supp(M')|=k$, and pick an $M'$ with $|\supp(M')|=k+1$. We can find $M^1,M^2\in \MP(U)$ such that $M'=M^1\vee M^2$, $|\supp(M^1)|=k$ and $|\supp(M^2)|=1$: simply by changing $k$ partitions of $M'$ to the trivial partition to get $M^2$, and changing one partition of $M'$ to a trivial one to get $M^1$. Then, applying the induction hypothesis,  the fact that $S_{M^1}$ is weakly isometric, \cref{lem:operatorsselfcommutative}, and \cref{lem:weakisometricity}  we get
     \[
        (S^M)_{M'}\stackrel{\text{Lem.  \ref{lem:operatorsselfcommutative}}}{=}((S^M)_{M^1})_{M^2}\stackrel{\text{IH}}{=}((S_{M^1})^M)_{M^2}\stackrel{\text{Lem.  \ref{lem:weakisometricity}}}{=}((S_{M^1})_{M^2})^M=(S_{M'})^M.
     \]
     To show (3)$\Rightarrow$(2), suppose we have $M'''\in \MP(U)$, then we show that $S_{M'''}$ is weakly isometric. To achieve this, we show that $S_{M'''}$ is weakly $(k,\ell)$-isometric for all $k,\ell\in X$ with $k\neq \ell$. Suppose we have one-dimensional $M,M'\in \MP(U)$ such that $\supp(M)=k$ and $\supp(M')=\ell$. We distinguish two cases:
     \begin{itemize}
         \item Suppose $k,\ell\notin \supp(M''')$. In this case $\supp(M'\vee M''')\cap\supp(M)=\varnothing$, so applying \cref{lem:operatorsselfcommutative} we get
         \[
            ((S_{M'''})_{M'})^{M } = (S_{M'''\vee M'})^M \stackrel{(4)}{=} (S^M)_{M'''\vee M'} =((S^M)_{M'''})_{M'} \stackrel{(4)}{=} ((S_{M'''})^M)_{M'}.
         \]
         And by \cref{lem:weakisometricity} this means that $S_{M'''}$ is weakly $(k,\ell)$-isometric.
         \item Suppose that $\{k,\ell\}\cap\supp(M''')\neq\varnothing$. Then, we find $M^1,M^2$ such that $M'''=M^1\vee M^2$, $k,l\notin\supp(M^1)$ and $\supp(M^2)\subseteq \{k,\ell\}$. This can be done by changing the appropriate partitions to the trivial partition in $M'''$. Applying the previous case, we know that $S_{M^1}$ is weakly $(k,\ell)$-isometric. Then, by \cref{lem:convexity-preimages}, we know that $(S_{M^1})_{M^2}=S_{M'''}$ is also weakly $(k,\ell)$-isometric.
     \end{itemize}
     So in both cases $S_{M'''}$ is weakly $(k,\ell)$-isometric, and this implies assertion (2).
 \end{proof}

\begin{lemma}\label{lem:subproductcommutative}
    Let $S\subseteq U$ be such that $(S^M)_{M'}=(S_{M'})^M$ for all $M,M'\in \MP(U)$ with $\supp(M)\cap \supp(M')=\varnothing$. Let $V$ be a full-dimensional subproduct of $U$. Then $((S\cap V)^M)_{M'}=((S\cap V)_{M'})^M$ for all $M,M'\in \MP(V)$ with $\supp(M)\cap \supp(M')=\varnothing$.
\end{lemma}
\begin{proof}
    It suffices to show that commutativity is preserved if just one element is removed from one $U_i$ with $|U_i|\geq 2$, since then the result follows by induction. Assume without loss of generality that $|U_1|\geq 2$, and let $x\in U_1$ and $V=U_1\backslash\{x\}\times U_2\times\ldots\times U_m$. Consider any two minor-subproducts $M,M'\in \MP(V)$ with $\supp(M)\cap\supp(M)=\varnothing$. Let $N,N'\in \MP(U)$ be obtained from $M,M'$, respectively, by adding the singleton $\{x\}$ to the partition in the first dimension. Note that $\supp(N)\cap\supp(N')=\supp(M)\cap \supp(M')=\varnothing$, and therefore $(S^N)_{N'}=(S_{N'})^N$. The only difference between $M\vee M'$ and $N\vee N'$ is that the latter has extra boxes with first coordinate $\{x\}$. Adding the singleton $\{x\}$ has no effect on these extra boxes with regard to (strong-)projection, thus the only difference between $((S\cap V)^M)_{M'}$ and $(S^N)_{N'}$ is that the latter may contain some extra elements with first coordinate $\{x\}$. Likewise, $((S\cap V)_{M'})^M$ and $(S_{N'})^N$ only differ by some elements with first coordinate $\{x\}$. It follows that indeed $((S\cap V)^M)_{M'}=((S\cap V)_{M'})^M$.
\end{proof}

\subsection{Ampleness, commutativity, and superisometricity} We continue with the first main characterization of ample sets. 

\begin{theorem}\label{thm:ampleisometriccommutative}
Let $S\subseteq U$. Then the following conditions are equivalent:
\begin{enumerate}
    \item \label{item:ample} $S$ is ample;
    \item \label{item:superisometric} $S^M$ is isometric for all $M\in\MP(U)$ (superisometricity);
    \item \label{item:superconnected}  $S^M$ is connected for all $M\in\MP(U)$ (superconnectivity);
    \item \label{item:commutative} $(S^M)_{M'}=(S_{M'})^M$ for all $M,M'\in \MP(U)$ with $\supp(M)\cap\supp(M')=\varnothing$ (commutativity);
    \item \label{item:complement} $S^*$ is ample;
    \item \label{item:weaklyisometric} $(S^M)_{M'}$ is weakly isometric for all $M,M'\in \MP(U)$.
\end{enumerate}
\end{theorem}

\begin{proof} 
    First, we prove $(1)\Rightarrow (2)$. By \cref{lem:superscriptample}, $S^M$ is ample for all $M\in \MP(U)$. By \cref{prop:ample-is-isometric}, every ample set is isometric, thus $S^M$ is isometric for all $M\in \MP(U)$.

    The implication (2)$\Rightarrow$(3) is trivial. For the converse implication (3)$\Rightarrow$(2), suppose there exist $U$ and $S\subseteq U$ such that $S$ is superconnected but not superisometric. Assume we have picked $U$ with this property with the smallest cardinality. This implies that $S^M$ is isometric for all $M\in \MP(U)\backslash M^{\bot}$ (since $(S^M)^{M'}$ is connected for all $M'\in \MP(U)$ and $|M'|<|U|$). From the minimality choice we conclude that $S$ is not isometric. Therefore, there must be non-adjacent $u,v\in S$ such that $[u,v]\cap S=\{u,v\}$. Since $S$ is connected, $u$ and $v$ are connected by a path in $S$ in the Hamming graph $H(U)$. Let $P=(u, u_1,u_2,\ldots, u_k, v)$ be such a path of shortest length. Let $x_u,x_{u_1}$ be the coordinates by which $u$ and $u_1$ differ, and let $M$ be the elementary minor-subproduct whose only nontrivial block is $\{x_u,x_{u_1}\}$. Let $u'\in M$ such that $F(u')=\{u,u_1\}$, then we have $u'\in S^M$.
    We distinguish two cases:
   \begin{case}    Every vertex of the path $P$ has a coordinate in $\{x_u,x_{u_1}\}$. Then both $u$ and $v$ have coordinate $x_u$, since $v$ having coordinate $x_{u_1}$ would imply $u_1\in [u,v]$. This implies that there must be some $i$ such that $u_i$ has coordinate $x_{u_1}$ and $u_{i+1}$ has coordinate $x_u$ (if $i=k$ we say $u_{k+1}:=v$). Suppose that $i$ is the smallest such index. Let $u_i'\in M$ be such that $F(u_i')=\{u_i,u_{i+1}\}$. Since $S^M$ is isometric, there is a $(u',u_i')$-path in $S^M$ of length $d(u',u_i')=d(u,u_{i+1})$. If we take the element with coordinate $x_u$ from each box of the $u',u_i'$-path, we get a $(u,u_{i+1})$-path in $S$ that has length (in terms of edges) at most $i-1$, which contradicts our assumption that $P=(u,u_1,\ldots, u_k,v)$ was a  $(u,v)$-path of shortest length in $S$.
   \end{case}

   \begin{case} There is a vertex of the path $P$  that does not have a coordinate in $\{x_u,x_{u_1}\}$. Let $u_i$ be the first such vertex. Then $\{u_i\}$ is also a vertex of $S^M$. Let $P'=(u',u_1',u_2',\ldots,u_j',\{u_i\})$ be a path of length $d_M(u',\{u_i\})$ in $S^M$, which exists because $S^M$ is isometric. Note that $d_M(u',\{u_i\})=d(u,u_i)=d(u_1,u_i)$, since $u_i$ does not have an $x_u$ or $x_{u_1}$-coordinate. Now consider the $u,u_i$-path in $S$ created from $u',u_1',u_2',\ldots,u_j',\{u_i\}$ by taking either the element from the box with coordinate $x_u$ (if the element $u_\ell'$ has coordinate $\{x_u,x_{u_1}\}$), or taking the sole element in the box of $u_\ell'$. This forms a path in $S$: if $u_\ell'$ and $u_{\ell+1}'$ both have coordinate $\{x_u,x_{u_1}\}$, then the two $x_u$-coordinate elements of $S$ we pick are adjacent; if $u_\ell'$ and $u_{\ell+1}'$ both don't have coordinate $\{x_u,x_{u_1}\}$, then both boxes only contain one element of $S$, and the two are adjacent; and if exactly one of $u_\ell'$ and $u_{\ell+1}'$ has coordinate $\{x_u,x_{u_1}\}$, then the chosen elements of $S$ differ in only one coordinate, so they are adjacent. The new path has $d_M(u',\{u_i\})=d(u,u_i)=d(u_1,u_i)\leq i-1$ edges, hence $(u,u_1,\ldots,u_i)$ was not a shortest $u,u_i$-path in $S$, which contradicts our assumption that $P$ was a $(u,v)$-path of shortest length in $S$.
   \end{case}

    Since in both cases we get a contradiction,  we conclude $S$ is isometric and therefore (3)$\Rightarrow$(2).
    
    Now we prove $(2)\Rightarrow (4)$. 
    First, assume that $S^M$ is isometric for all minor-subproducts $M$. By \cref{lem:isometricityclassified}, a set being isometric means $(S^{M'})_{M''}=(S_{M''})^{M'}$ for all $M''\in \MP(U)$ and all one-dimensional $M'\in \MP(U)$ with $\supp(M')\cap\supp(M'')=\varnothing$. This means that
    \begin{equation}\label{eq:commutativelemma}
    ((S^M)^{M'})_{M''}=((S^M)_{M''})^{M'}
    \end{equation}
    for all $M,M',M''\in \MP(U)$ such that $M'$ is one-dimensional and $\supp(M')\cap\supp(M'')=\varnothing$.
    
    We prove statement (4) by induction on $|\supp(M)|$. For $|\supp(M)|=0$ we trivially have $(S_{M'})^M=S_{M'}=(S^M)_{M'}$. Now suppose that we have shown the statement for $|\supp(M)|\leq k$, then choose any $M$ with $|\supp(M)|=k+1$. Then by \cref{lem:atomistic} we can find $M^1,M^2\in \MP(U)$ with $M=M^1\vee M^2$, such that $|\supp(M^1)|=|\supp(M)|-1$, and $|\supp(M^2)|=1$. We have
    \[
    (S^M)_{M'}=((S^{M^1})^{M^2})_{M'}\stackrel{\text{\eqref{eq:commutativelemma}}}{=}(((S^{M^1})_{M'})^{M^2}\stackrel{\text{IH}}{=}(((S_{M'})^{M^1})^{M^2}=(S_{M'})^{M}.
    \]
    This completes the induction step.
    
    Then, we show $(4)\Rightarrow (1)$. 
    For contradiction, assume that $U$ is a set such that there exists $S\subseteq U$ that is commutative but not ample. 
    We may assume $S$ is not $\EMP(U)$-ample by \cref{thm:ample=weaklyample} Assume that we have picked  $U$ with this property with the least number of elements. Since $S$ is not $\EMP(U)$-ample, there exists an extended subproduct $M=M(\Lambda,U)\in \MP(U)$ that is shattered by $S$ but not strongly shattered by $S$. 

    Let $M=M_1\times\ldots\times M_m$, and suppose that there is a subfactor $M_i$ that is neither trivial nor co-trivial. We may assume that $M_1$ is such a subfactor. Let $\alpha_1=\{P_1^1,\ldots,P_{\ell_1}^1\}$ be the quasi-trivial partition related to $M_1$, with $1<|P_1^1|<|U_1|$. 

    Consider the subproduct $U'=P_1^1\times U_2\times \ldots \times U_m$. By \cref{lem:subproductcommutative}, $S\cap U'$ is commutative in $U'$. Since $|U'|<|U|$, by our initial assumption  $S\cap U'$ is  ample in $U'$. In particular, let $M'=M(\Lambda',U')\in \MP(U')$, with $\Lambda'=(\{P_1^1\},\alpha_2,\ldots,\alpha_m)$. The boxes of $M'$ form a subset of the boxes of $M$. Since $M$ is shattered by $S$, it follows that $M'$ is shattered by $S\cap U'$. By ampleness, this means that $M'$ is strongly shattered by $S\cap U'$, so $S\cap U'$ contains a copy $W=W_1\times \ldots \times W_m$ of $M'$. In particular, since $\{P_1^1\}$ is a co-trivial partition in $U'$, the set $W_1$ consists of a single element $u_1^1$.

    Now let $U''=(U_1\backslash P_1^1)\cup\{u_1^1\}\times U_2\times \ldots \times U_m$, and let $M''=(\Lambda'',U'')$, where $\Lambda''=(\alpha_1\backslash P_1^1\cup \{\{u_1^1\}\}, \alpha_2,\ldots,\alpha_m)$. Since  $|P_1^1|>1$, $U''$ is a subproduct of $U$ that is strictly smaller than $U$. By \cref{lem:subproductcommutative} it follows that $S\cap U''$ is commutative in $U''$. Since $|U''|<|U|$, $S\cap U''$ is then ample. We argue that $M''$ is shattered by $S\cap U''$. Consider any box of $M''$. If its first coordinate is $\{u_1^1\}$, then the box is a smaller version of a box from $M'$ in $U'$. Since by definition $W$ is a copy of $M'$, this means that the box contains an element of $W$, which is an element of $S$ with first coordinate $u_1^1$. This element is therefore also contained in the box of $M''$. On the other hand, if we have a box of $M''$ with any other first coordinate, then this is also a box of $M$, which then contains an element of $S$, since $M$ is shattered by $S$. We conclude that $M''$ is shattered by $S\cap U''$, and because of ampleness, $M''$  is strongly shattered by $S\cap U''$. This means that $S\cap U''$ contains a copy $W'$ of $M''$. However, since every box of $M$ contains a box of $M''$, this implies that every box of $M$ contains an element of $W'$, and thus $W'$ is a copy of $M$. So $M$ is strongly shattered by $S$. This is a contradiction with our assumptions. We conclude that our last assumption, that $M_1$ is neither trivial nor co-trivial, is false. Hence we can assume that all subfactors of $M$ are either trivial or co-trivial, which means that $M$ is mixed. 

    Mixed minor-subproducts $M$ have  unique complements $M^{\diamond}$: $M^{\diamond}$ is obtained from $M$ by replacing trivial partitions by co-trivial partitions and vice versa. Note that $M$ being  shattered by $S$ can be equivalently written as  $S_M=M$. Finally, note that for any $S'\subseteq M$, we have $(S')^{M^{\diamond}}\neq \varnothing$ if and only if $S'=M$. So $M$ being shattered by $S$ is equivalent to $(S_M)^{M^{\diamond}}\neq \varnothing$.

    Since $M$ is mixed, every box of $M^{\diamond}$ contains precisely one element from every box of $M$, and therefore any box of $M^{\diamond}$ forms a copy of $M$. Therefore, every element of $S^{M^{\diamond}}$ corresponds to a copy of $M$ in $S$. So $S^{M^{\diamond}}$ is nonempty if and only if $M$ is strongly shattered by $S$. Therefore $S^{M^{\diamond}}$ is nonempty if and only if $(S^{M^{\diamond}})_M\neq \varnothing$. 
    
    In conclusion, since $M$ is shattered by $S$, we have $(S_M)^{M^{\diamond}}\neq \varnothing$. By commutativity of $S$, we have $(S^{M^{\diamond}})_M\neq \varnothing$, which implies that  $M$ is strongly shattered by $S$. This contradicts our initial choice of $S$ and $U$. We conclude that there is no set $S$ that is commutative but not ample, and that completes the proof of the implication $(4)\Rightarrow (1)$. 

    To prove the implication $(4)\Rightarrow(5)$, let $M,M'\in \MP(U)$ with $\supp(M)\cap \supp(M')=\varnothing$. By Lemma \ref{lem:projections-and-complements} we get $(S^M)^*=(S^*)_M$ and $(S_M)^*=(S^*)^M$ for any $S$. Since $S$ is commutative we have
    \[
        ((S^*)_M)^{M'}=((S^M)^*)^{M'}=((S^M)_{M'})^*=((S_{M'})^M)^*=((S_{M'})^*)_M=((S^*)^{M'})_M.
    \]
    So $S^*$ is commutative, hence by the equivalence between (1) and (4), $S^*$ is ample, yielding $(4)\Rightarrow (5)$. The 
    implication $(5)\Rightarrow (1)$ follows by replacing $S$ by $S^*$, since we already know $(1)\Rightarrow (5)$. Finally, the equivalence (2)$\Leftrightarrow$(6) follows immediately from \cref{lem:isometricityclassified}. 
    \end{proof}

\begin{lemma}\label{lem:ample-projection} If $S\subseteq U$ is ample, then for any $M\in \MP(U)$, $S_M$ is ample in $M$. 
\end{lemma}

\begin{proof}
    By \cref{thm:ampleisometriccommutative} we know that $S$ is ample if and only if $S^*$ is ample. It follows from \cref{lem:superscriptample} that $(S^*)^M$ is ample, which is equal to $(S_M)^*$ by \cref{lem:projections-and-complements}. Taking the complement again, we deduce that $S_M$ is ample.
\end{proof}

\subsection{Ampleness and elementary minor-subproducts} In this subsection, we characterize ample sets in a more efficient way, using elementary minor-subproducts and intersections with intervals.  Before proving this result, we consider elementary minor-subproducts in more detail. 

Recall that the elementary minor-subproducts of $U$ correspond to the atoms  of $\GPart(U)$. The atoms of $\GPart(U)$  are the generalized partitions $\Lambda=(\alpha_1,\ldots,\alpha_m)$ in which all partitions of factors contain only singleton blocks, except one partition (say, the partition $\alpha_i$ of $U_i$), which contains singleton blocks and  precisely one block $e=\{a, b\}\subseteq U_i$ of size 2. As we noted before, the box-partitions $\cB(\Lambda,U)$ corresponding to $\Lambda$ consists only of singleton boxes, and of some boxes of size two, corresponding to the class $\Theta(e)$: the edges of $H(U)$ parallel to $e$. In view of this, we allow ourselves to represent an elementary minor-subproduct by such a block $e$. 

Pick any edge $e'=uv\in \Theta(e)$ and suppose without loss of generality that the $i$-coordinate of $u$ is $a$ and the $i$-coordinate of $v$ is $b$. Recall from \cref{thm:partial-Hamming} that $e'$ defines a partition of $U$ into three complementary halfspaces $W(u,v), W(u,v)$ and $W_=(u,v)$. 
This partition is independent of the choice of the edge $e'$ from $\Theta(e)$. Therefore for  $e=\{a, b\}\subseteq U_i$ we can canonically define 
a partition of $U$ into three halfspaces $W_a(e)=W(u,v), W_b(e)=W(v,u)$, and $W_=(e)$. Then $W(e)=W_a(e)\cup W_b(e)=U\setminus W_=(e)$ is also a halfspace of $U$, which we dub the \emph{main halfspace}, while  $W_=(e)$ will be called the \emph{residual halfspace} of $e$.  In case of binary products, for any choice of $e$, the residual halfspace $W_=(e)$ is empty.  This is also the case when the factor $U_i$ containing $e$ is binary.  
If $S\subseteq U$, then we set $S_a(e)=S\cap W_a(e), S_b(e)=S\cap W_b(e), S_=(e)=S\cap W_=(e),$  and $S(e)=S\cap W(e)=S_a(e)\cup S_b(e)$. 

Now we consider the projection and strong-projection $S_e$ and $S^e$, where we again interpret edge $e$ as an elementary minor-subproduct. For binary products $U$, the set $S$ is a subset of the hypercube $\{ -1,+1\}^X$. In that case the projection $S_e$ of an elementary minor-subproduct is, as the name suggests, the orthogonal projection of $S$ on the coordinate hyperplane defined by $X\setminus \{i\}$, with $i=\supp(e)$. Then, the sets $S_e$ and $S^e$ (often denoted by $S_i$ and $S^i$)
have a natural geometric and graph-theoretical interpretation (for more details, see \cite{BaChDrKo_geometry}). 
The subgraph $H(S_e)$ of the hypercube $\{ -1,+1\}^{X\setminus \{i\}}$ induced by $S_e$ 
is obtained from the subgraph $H(S)$ of $\{ -1,+1\}^{X}$ induced by $S$ by contracting all edges of $H(S)$ from the class $\Theta(e)$. On the other hand, $S^e$ can be viewed as the intersection of the coordinate hyperplane 
$x_i=0$ with the $e$-edges of $H(S)$ viewed as solid segments. Two tuples of $S^e$ are adjacent in $H(S^e)$ if and only if they are the middles of two $e$-edges defining a square of $H(S)$.  In the same way, each cube in $H(S)$ containing $i$ in its support gives raise  to a cube of $H(S^e)$ of one dimension less.  This is why $S^e$ can be dubbed the \emph{$i$-hyperplane} of $S$ (this terminology, coming from geometric group theory, was used in a more general setting in \cite{BaChKn}).

For subsets $S$ of arbitrary products $U$,  the projection $S_e$ can be interpreted in the same way as in the  binary case and the graph $H(S_e)$ is obtained from the graph $H(S)$ by contracting all $e$-edges. In particular, if $e=\{ a,b\}\subseteq U_i$, we will denote by $U_e$ the Cartesian product $U_1\times\ldots\times U_{i-1}\times U'_i\times U_{i+1}\times\ldots\times U_m$, where $U'_i=U_i\setminus \{ a,b\}\cup \{ w\}$ (in some proofs, instead of coordinate $w$ we will use coordinate $\{ a,b\}$). Then $S_e$ and $S^e$ are subsets of $U_e$ (to avoid saying that $S_e$ and $S^e$ are subsets of $e$).  By definition of $S^e$, each $e$-edge $s's''$ of $H(S)$  will be mapped to a tuple $t$ of $S^e$, because the fiber $F(t)$ of $t$ consists of the two extensions $s'$ and $s''$, and both $s',s''$ belong to $S$. If the factor $U_i$ containing $\{ a,b\}$ is binary, then the box-partition  $\cB(\Lambda,U)$  coincides with the perfect matching defined by the $e$-edges of $U$ and thus $S^e$, and $H(S^e)$ can be interpreted as in the binary case. Otherwise,  $\cB(\Lambda,U)$  will contain singleton boxes. By definition of $S^e$, each such box $\{ s\}\subset S$ will be bijectively mapped to the same tuple $s$ of $S^e$. Therefore,  $S^e$ is the disjoint union of two parts: the \emph{hyperplane} $(S(e))^e$ of the main halfspace $S(e)$ of $S$ and the \emph{residue} $(S_=(e))^e$, corresponding to the residual halfspace $S_=(e)$ of $S$. There will be an edge between two vertices $t,s\in S^e$, $t$ belonging to the hyperplane $(S(e))^e$ and $s$ belonging to the residue  $(S_=(e))^e$, if and only if $S$ contains a triangle $s's''s$ such that $s's''\in \Theta(e)$ and $s\in S_=(e)$. 

\begin{lemma} \label{lem:hyperplanes-ample} If $S\subseteq U$ is ample and $e=\{ a,b\}\subseteq U_1$ is an elementary product of $U$, then the hyperplane
$(S(e))^e$ and the residue $(S_=(e))^e)$ are ample in $U_e$. 
\end{lemma}

\begin{proof} If $S$ is ample, then each of the sets $S_a(e),S_b(e), S(e),$ and $S_=(e)$ are  also ample because the intersection of an ample set with a full-dimensional subproduct is ample by Theorem \ref{thm:ample=weaklyample}. 
Consequently, 
$(S(e))^e$ and $(S_=(e))^e)$ are ample in $U_e$ by  Lemma \ref{lem:superscriptample}. 
\end{proof}

We will use the following refinement of superisometricity: 

\begin{definition} [Box-superisometricity]
    A set $S\subseteq U=U_1\times\ldots\times U_m$ 
    is \emph{box-superisometric}  if for any pair of parallel boxes $B',B''\subseteq S$, there exists a geodesic gallery between $B'$ and $B''$ that is contained in $S$. 
\end{definition}

Here is our second main characterization of ampleness: 

\begin{theorem}\label{thm:ample-elementary}
\setcounter{claim}{0} 
Let $S\subseteq U$. Then the following conditions are equivalent:
\begin{enumerate}
    \item \label{item:ample} $S$ is ample;
    \item \label{item:contraction-strongcontr} $S$ is isometric and both $S^e$ and $S_e$ are ample for some 
    elementary minor-subproduct $e$; 
    \item \label{item:weakisometricity-contractions} $S$ is weakly isometric and both $S^e$ and $S_e$ are ample for some 
    elementary minor-subproduct $e$;
    \item \label{item:strongcontractions} $S$ is connected and $S^e$ is ample for every elementary minor-subproduct $e$;
    \item $S$ is box-superisometric;
    \item $S\cap[u,v]$ is ample in $[u,v]$ for all $u,v\in U$; 
    \item $S\cap[u,v]$ is ample in $[u,v]$ for all $u,v\in S$.
  \end{enumerate}
\end{theorem}

\begin{proof} To prove the equivalence of the conditions from (1) to (7), we establish three chains of
implications  $\rm(1)\Rightarrow \rm(2)\Rightarrow \rm(3)\Rightarrow \rm(4)\Rightarrow \rm(1)$, $\rm(1)\Rightarrow \rm(5)\Rightarrow \rm(6)\Rightarrow \rm(1)$, and $\rm(6)\Rightarrow\rm(7)\Rightarrow \rm(6)$. 

The implication (1)$\Rightarrow$(2) follows from \cref{prop:ample-is-isometric}, \cref{lem:superscriptample} and \cref{lem:ample-projection}. The 
implication (2)$\Rightarrow$(3) is trivial. To prove the implication (\ref{item:strongcontractions})$\Rightarrow$(1), by \cref{thm:ampleisometriccommutative} it suffices to show that $S^M$ is connected for any $M\in \MP(U)$. This is obviously so if $M=M_{\bot}$ since in this case we have $S^{M_\bot}=S$ and $S$ is connected. Otherwise, by \cref{lem:atomistic}, $M$ can be written as the join $M=\bigvee_{i=1}^k e_i$ of atoms (elementary minor-subproducts)  $e_1,\ldots,e_k$. Let $M'=\bigvee_{i=2}^k e_i\in \MP(U)$.  Then $S^M=(S^{e_1})^{M'}$. Since $S^{e_1}$ is ample by condition (4), the set $(S^{e_1})^{M'}$ is isometric by \cref{thm:ampleisometriccommutative}. Consequently,  we have the implication (\ref{item:strongcontractions})$\Rightarrow$(1).  We now show the final implication (3)$\Rightarrow$(4) of the first chain of implications. 
Its proof consists of several intermediate steps and extends the proof of the  implication (iii)$\Rightarrow$(iv) of Theorem 4 of \cite{BaChDrKo} (but the proof is much more involved than the binary case). 

\begin{claim}  $S$ is connected.
\end{claim}

\begin{proof}  Pick any $s,t\in S$, and let $s',t'\in S_e$ such that $s\in F(s')$ and $t\in F(t')$. Assume $s'\neq t'$ (otherwise $s=t$ or $s\sim t$), and, as $S_e$ is connected, let $P=(s',u_1',u_2',\ldots,u_k',t')$ be a path in $S^e$. Consider a sequence $s,u_1,u_2,\ldots, u_k,t$ in $S$, where each $u_i\in F(u_i')$. Every element $u_i$ in this sequence has distance at most 2 from the next element $u_{i+1}$, since each  pair $u_i,u_{i+1}$ either consists of adjacent vertices or of ends of two incident edges of $\Theta(e)$.  From weak isometricity of $S$, for each consecutive $u_i,u_{i+1}$ that are not adjacent, we can adjoin to $P$ their common neighbor $w_i\in S$ and obtain a path $P'$ connecting $s$ and $t$ in $S$. 
\end{proof}

\begin{claim}  $S^f$ is weakly $(k,\ell)$-isometric for all elementary $f\in\MP(U)$ and $k,\ell\in X$ with $k\neq \ell$.  
\end{claim}

\begin{proof}
 If $\supp(f)\in\{k,\ell\}$, then this follows immediately from \cref{lem:convexity-preimages}(3). Otherwise, let $s,t\in S^f\subseteq U_f$ such that $s,t$ have distance 2 in $U_f$ and only differ in the coordinates $k$ and $\ell$. Suppose for contradiction that there is no path of length 2 connecting $s,t$ in $S^f$. Therefore the interval $[s,t]$ in $U_f$ is a square $(s,p,t,q)$ with only two opposite vertices $s,t$ in $S^f$. 
Denote by $Q$ the 3-cube of $U$, which is the expansion of the square $[s,t]$. Set $C=S\cap Q$. Let $F(s)=\{ s_1,s_2\}, F(t)=\{t_1,t_2\}, F(p)=\{ p_1,p_2\}$, and $F(q)=\{ q_1,q_2\}$. For convenience of notation, let $a,0$ be the coordinates of $s$ (at index $k,\ell$, respectively), let $b,1$ be the coordinates of $t$, and let $A,B$ be the elements of the nontrivial block of the generalized partition of $f$. Suppose also that these three coordinates are defined by the three first factors $U_1,U_2,U_3$ of $U$. Then the four fibers $F(s),F(t),F(p)$ and $F(q)$  have the form \[F(s)=\{s_1=(a,A,0,z),s_2=(a,B,0,z)\}, F(t)=\{t_1=(b,A,1,z),t_2=(b,B,1,z)\},\]
\[F(p)=\{ p_1=(a,A,1,z), p_2=(a,B,1,z)\},F(q)=\{ q_1=(b,A,0,z), q_2=(b,B,0,z)\}\] (see \cref{fig:thm3proof}), where $z$ is a tuple of $U_4\times\ldots\times U_m$, common to all vertices of the 3-cube $Q$. By weak isometricity of $S$, $S$ must contain one of the vertices $p_1,q_1$, and one of the vertices $p_2,q_2$. The only way to do so without creating a connection between $s$ and $t$ in $S^f\subseteq U_f$, is if $C=S\cap Q$ is an isometric 6-cycle in the 3-cube $Q$. For the rest of the proof of the claim, assume without loss of generality that $Q\backslash S=\{p_2,q_1\}$. 
\begin{figure}[htbp]
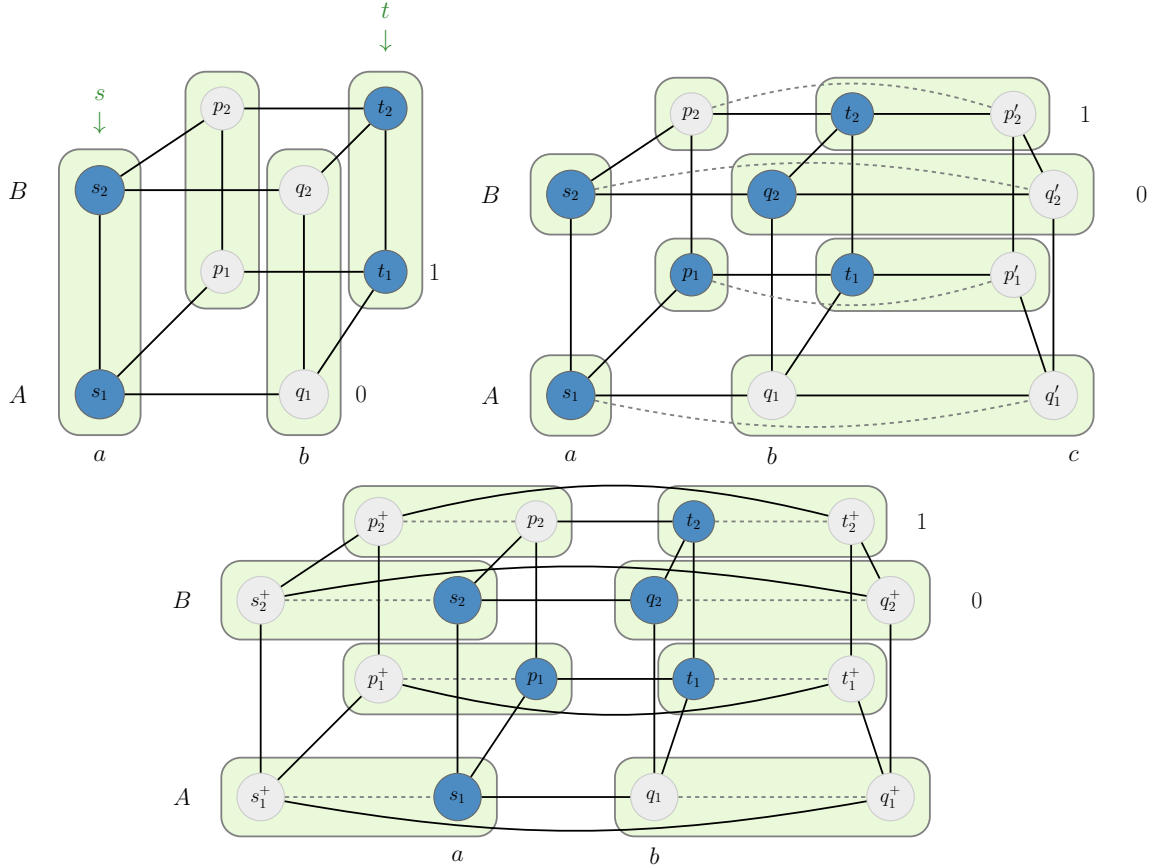

    \centering
    \includestandalone[height=0.4\linewidth]{img/thm3proof1}
    \includestandalone[height=0.33\linewidth]{img/thm3proof2}
    \includestandalone[height=0.33\linewidth]{img/thm3proof3}
    \caption{Top left: the boxes of $f$ related to the square $[s,t]$. Top right: the boxes of $e$ that correspond to the vertices of $Q'$ in case 3. Bottom: the boxes of $e$ with elements of $Q$ in case 4 (note: the fourth coordinate is not indicated, the vertices marked with $+$ are those with $4$-coordinate $\beta$)}
    
    \label{fig:thm3proof}
\end{figure}

Now, we will use ampleness of  $S^e$ and $S_e$  to derive a contradiction. 
Consider the elementary minor-subproduct $e$. We distinguish four cases. 

\begin{case} The nontrivial block of $e$ is one of the pairs $\{a,b\}, \{ A,B\},\{ 0,1\}$, say $e=\{a,b\}$. 
\end{case}
Then $Q$ is included in the main halfspace $W(e)$ of $e$, hence $C\subseteq S(e)$. Since $\{ s_2,q_2\}, \{ p_1,t_1\}\subseteq S$ and $\{ s_1,q_1\}, \{ p_1,t_1\}$ are not included in $S$, in $(S(e))^e\subseteq S^e$ 
we will obtain two opposite vertices of a square, whose two other vertices are not in $(S(e))^e$. Therefore $(S(e))^e$ is not weakly isometric  and thus $(S(e))^e$ is not ample, a contradiction with Lemma \ref{lem:hyperplanes-ample} and ampleness of $S^e$.

\begin{case} The nontrivial block of $e$ is included in $U_i$ for some $i\le 3$, but contains none of $a,b,A,B,0,1$, say $e=\{ c,d\}\subseteq U_1\setminus \{ a,b\}$.
\end{case}
Then $Q$ is included in the residual halfspace $W_{=}(e)$ of $e$, hence $C\subseteq S_=(e)$. Therefore $(S_=(e))^e$ includes a 3-cube, which intersected with $(S(e))^e$ is a 6-cycle. Thus $(S_=(e))^e$ is not ample, a contradiction with \cref{thm:ample=weaklyample} and ampleness of $S^e$. 

\begin{case} The nontrivial block of $e$ contains exactly one of the six coordinates $a,b,A,B,0,1$.
\end{case}

We can assume by symmetry that this block is $\{b,c\}$ for some $c\in U_1\setminus \{ a,b\}$. Let $p'_1=(c,A,1,z), p'_2=(c,B,1,z)\},q'_1=(c,A,0,z),$ and $q'_2=(c,B,0,z)$. Then $(q'_1,q'_2,p'_2,p'_1)$ is a square of $U$ and $q'_1\sim s_1,q_1$, and  $p'_1\sim p_1,t_1$, and $p'_2\sim p_2,t_2$, and $q'_2\sim s_2,q_2$ (see \cref{fig:thm3proof}). 
Let $Q'$ be the 3-cube of  $U_e$ whose elements are the boxes of $\BBox(e,U)$ that contain elements of the 3-cube $Q$  (that is, there are four singleton boxes $\{ s_1\},\{ s_2\}, \{p_1\}, \{ p_2\}$ containing one element and four boxes $\{ q_1,q'_1\}, \{ t_1,p'_1\}, \{ q_2,q'_2\}, \{ t_2,p'_2\}$ containing two elements, corresponding to the four edges from $\Theta(e)$ incident to $q_1,q_2,t_1,t_2$, see \cref{fig:thm3proof}).  
Since $S_e$ is obtained from $S$ by contracting the edges of $\Theta(e)$, the only way that $S_e$ does not form an isometric 6-cycle in $Q'$ is if $q'_1\in S$. Since $q'_1=(c,A,0,z)$ and $t_1=(b,A,1,z)$ are at distance 2, and $S$ is weakly isometric, they must share a neighbor in $S$, which can only be $p'_1=(c,A,1,z)$. Likewise, $q'_1=(c,A,0,z)$ and $q_2=(b,B,0,z)$ are at distance 2, yielding $q'_2=(c,B,0,z)\in S$. The two elements $p'_1=(c,A,1,r)$ and $q'_2=(c,B,0,r)$ are in $S$, and the two boxes that contain them are then in $S^e$ and have distance 2. Thus by ampleness of $S^e$, they must be connected by another box of $S^e$, for which the only option is that $p'_2=(c,B,1,z)\in S$. We have now determined for all elements in the expansion of elements of $Q'$ if they are in $S$. Namely, we deduced that among the vertices $s_1,s_2,p_1,p_2,t_1,t_2,q_1,q_2,q'_1,q'_2,p'_1,p'_2$, only $q_1$ and $p_1$ do not belong to $S$. Furthermore, 
$s_1,s_2,p_1$ belong to $S_=(e)$ and $\{q_2,q_2'\},\{t_1,p_1'\},\{t_2,p_2'\}\subseteq S(e)$. 
This implies that $S^e$ forms an isometric 6-cycle in $Q'$, so $S^e$ cannot be ample. This is a contradiction, and we conclude that $S^f$ is weakly $(k,\ell)$-isometric in this case.

\begin{case} The nontrivial block of $e$ is included in $U_i$ for some $i\ge 4$, say $e=\{ \alpha,\beta\}\subseteq U_4$. 
\end{case}

Note that all vertices of the 3-cube $Q$ have the same fourth coordinate $\gamma$. First suppose $\gamma\notin \{ \alpha,\beta\}$. Then $Q$ is included in the residual halfspace $W_=(e)$. The projection $*^e$ maps the cube $Q$ to a 3-cube $Q'$ of $U_e$. Consequently, the vertices of the 6-cycle $C=S\cap Q$ give rise to a 6-cycle $C'$, which will be equal to the intersection of $Q'$ with the residue $(S_=(e))^e$. Consequently,  $(S_=(e))^e$ is not ample. Since $S^e$ is ample, this contradicts \cref{thm:ample=weaklyample}. 

Now, suppose  $\gamma\in \{ \alpha,\beta\}$, say $\gamma=\alpha$. For each 
vertex $r$ of the 3-cube $Q$ we denote by $r^+$ the vertex of $U$ having the same coordinates as $r$, except the fourth coordinate which is $\beta$. These vertices define a 3-cube $Q^+$ and the union $Q\cup Q^+$ is a 4-cube in $U$, namely, $Q\cup Q^+\subseteq W(e)$ (see \cref{fig:thm3proof}). Since $S_e$ is ample, it cannot have a cube in which it forms a 6-cycle: therefore, as $p_2,q_1\notin S$, necessarily at least one of the vertices $p^+_2, q^+_1$, must belong to $S$, say $q^+_1\in S$. Since $S$ is weakly isometric, necessarily $s^+_1,t^+_1,q^+_2$ must also belong to $S$. Consequently, $\{ s_1,s^+_1\},\{ q_2,q^+_2\},\{ t_1,t^+_1\}\subseteq S$. By weak isometricity of $S^e$ we also obtain that $\{ s_2,s^+_2\}, \{ p_1,p^+_1\}, \{ t_2,t^+_2\}\subseteq S$. This, however, implies that $S^e$ intersects a 3-cube of $U_e$ in a 6-cycle, contradicting ampleness of $S^e$. 

This finishes the proof of all cases, and implies that $S^f$ is weakly $(k,\ell)$-isometric for all $k,\ell\in X$. 
\end{proof}

Note that, since $S$ is weakly isometric,  by \cref{lem:operatorsselfcommutative,lem:weakisometricity}, we have $(S^f)_e=(S_e)^f$ if $\supp(e)\neq \supp(f)$ and $(S^f)^e=(S^e)^f$. We know that $(S_e)^f$ and $(S^e)^f$ are ample by \cref{lem:superscriptample}, so $(S^f)^e$ and $(S^f)_e$ are ample if $\supp(e)\neq \supp(f)$.
We now prove by induction on $|U|$ that $(3)\Rightarrow(4)$. Since we already proved that $S$ is connected, 
we only need to show that $S^e$ is ample for all elementary $e$. If $|U|=1$ or $|U|=2$ this is trivial. Suppose we have shown $(3)\Rightarrow(4)$ for all universes with lower cardinality than $|U|$. For any elementary minor-subproduct $f$ we have $|f|<|U|$, so if $\supp(e)\neq \supp(f)$, then $(S^f)^e$ and $(S^f)_e$ are ample and $S^f$ is weakly isometric, so by the induction hypothesis (4) holds for $S^f$, and then $S^f$ is ample 
(since we already showed $(4)\Rightarrow (1)$). Now suppose $\supp(e)=\supp(f)$. By \cref{lem:trivial-factor} we may assume $|U_i|>1$ for all $i$. If $m=1$, then any subset of $U$ is ample. Otherwise, there is an elementary minor-subproduct $e'\in \MP(U)$ with $\supp(e')\neq \supp(e)$. By what we just showed, $S^{e'}$ is ample. It follows from (3) and \cref{thm:ampleisometriccommutative} that $(S^*)_e=(S^e)^*$ and $(S^*)^e=(S_e)^*$ are ample. Since $S^*$ is also weakly isometric, we can repeat the entire proof up to this point (replacing $S$ by $S^*$) to show that $(S^*)^{e'}$ is also ample. It follows that $(S_{e'})^*$ is ample, so $S_{e'}$ is ample. Since $S$ is weakly isometric and $S^{e'},S_{e'}$ are ample, and since $\supp(e')\neq \supp(f)$, it follows from the case above that $S^f$ is ample. 
That completes the induction proof and establishes $(3)\Rightarrow(4)$. 

To prove the implication $(1)\Rightarrow (5)$, suppose that $S$ is ample. By \cref{thm:ampleisometriccommutative},   $S$ is superisometric. Let $B',B''$ be parallel boxes in $S$, such that both $B'$ and $B''$ are copies of a subproduct $V$ as in \cref{def:parallel}. Let $M=M(\Lambda,U)\in \MP(U)$ be an extended subproduct, and $\Lambda=(\alpha_1,\ldots,\alpha_m)$ be defined as follows: $M_i$ is trivial when $i\notin\supp (V)$ or $|V_i|=1$, and for all other $i$ it holds that $V_i$ is the unique nontrivial block of $\alpha_i$. Note that every copy of $V$ by definition is a box of the minor-subproduct $M$. In particular, there are $b',b''\in M$ such that $F(b')=B'$ and $F(b'')=B''$. We then know that $b',b''\in S^M$, since $B',B''\subseteq S$. By superisometricity, there is a path $b',b_1,\ldots,b_k,b''$ in $S^M$ of length $d_M(b',b'')=d(B',B'')$. In particular, this path is contained in the interval $[b',b'']$, and the expansion of any element in this interval is a copy of $W$. It follows that $F(b'),F(b_1),\ldots,F(b_k),F(b'')$ is a geodesic gallery that is contained in $S$, thus $S$ is box-superisometric. 

To prove the implication $(5)\Rightarrow (6)$, suppose that $S$ is box-isometric. Consider any interval $[u,v]$, and take any two parallel boxes $B',B''\in S\cap [u,v]$. Then there exists a geodesic gallery connecting the two boxes, which must be contained within the interval $[u,v]$, as it is of minimal length. As noticed by \cite{ChChMoWa}, in a hypercube, any path in some $S^M$ corresponds to a gallery and vice versa. As a result, in a hypercube, superisometricity is equivalent to box-superisometricity. In particular, since $[u,v]$ is a hypercube and $S\cap [u,v]$ is box-isometric within $[u,v]$, $S\cap [u,v]$ is also superisometric within $[u,v]$. By \cite{BaChDrKo} (or by \cref{thm:ampleisometriccommutative}), $S\cap [u,v]$ is ample in $[u,v]$, concluding the proof of  $(5)\Rightarrow (6)$. 

To show the final implication $(6)\Rightarrow (1)$ of the second chain of implications, suppose for contradiction that we have $U$ and $S$ such that $S\cap[u,v]$ is ample for all $u,v\in U$, but $S$ is not ample in $U$. Suppose furthermore that we pick $U$ such that $|U|$ is minimal. Clearly $S$ is weakly isometric, since ampleness in each interval implies weak isometricity. Pick any elementary minor-subproduct $e\in\MP(U)$, and let $\{a,b\}$ be the unique nontrivial block of its partitions. Since we have shown the implication $(3)\Rightarrow (1)$ already, it follows that one of $S^e$ and $S_e$ is not ample, say $S^e$ is not ample (the proof when $S_e$ is not ample is analogous). Since $e$ has strictly less elements than $U$, this means that there exist $u,v\in e$  such that $S^e\cap [u,v]$ is not ample in $[u,v]$. There are three types of intervals $[u,v]$ in $U_e$ that could occur. 

First suppose that neither $u$ nor $v$ has a coordinate equal to $\{a,b\}$. In that case, $[u,v]$ is also an interval of $U$, and we have $S\cap [u,v]=S^e\cap [u,v]$. Since $S\cap [u,v]$ is ample in $[u,v]$, it follows that $S^e\cap [u,v]$ is ample in $[u,v]$, which is a contradiction with the choice of $u,v\in U_e$.

Now, suppose that both $u$ and $v$ have the coordinate $\{a,b\}$. In that case, the union $I=\bigcup_{x\in [u,v]}F(x)$ is an interval in $U$ of one dimension higher than $[u,v]$. Since $S\cap I$ is ample in $I$, it follows by \cref{lem:superscriptample} that $(S\cap I)^{e'}$ is ample, where $e'\in \MP(I)$ is the elementary minor-subproduct with unique nontrivial block $\{a,b\}$. And since $(S\cap I)^{e'}=S^e\cap[u,v]$, this gives us the same contradiction as in the first case.

Finally, suppose that exactly one of $u,v$ has a coordinate $\{a,b\}$, say $v_{\supp(e)}=\{a,b\}$. Let $V=\bigcup_{x\in[u,v]}F(x)$, and suppose that $V\neq U$. Then $S\cap V$ is ample in $V$, since $S\cap [u',v']$ is ample in $[u', v']$ for all $u',v'\in V$ and since $|V|<|U|$. But again, we can take an elementary $e'\in \MP(V)$ whose only block is $\{a,b\}$, which implies by \cref{lem:superscriptample} that $(S\cap V)^{e'}$ is ample, and we have $(S\cap V)^{e'}=S^e\cap [u,v]$, which is a contradiction. It follows that $V=U$.

Summarizing, the only possibility for $S$ to not be ample is if $U=V$, which means $|U_i|\le 3$ for the factor where $\{a,b\}\subseteq U_i$, and $|U_i|\in \{1,2\}$ for all other $i$. By \cref{lem:trivial-factor} we may discard the factors where $|U_i|=1$. Moreover, if $|U_i|=2$ for some $i$, we can take elementary $e$ such that $U_i$ is its only nontrivial block, and then $S$ must be ample by the previous case distinction. Then, the only possibility remaining is the case where $m=1$ and $U_1$ has at most 3 elements, but in that case every subset of $U$ is ample. Consequently, if $S\cap [u,v]$ is ample for all $u,v\in U$, then $S$ is ample. This establishes that  $(6)\Rightarrow (1)$.

Finally, we show the equivalence $(6)\Leftrightarrow (7)$. The direction $(6)\Rightarrow (7)$ is trivial. To show the direction $(7)\Rightarrow (6)$, take any interval $[u,v]$ with $u,v\in U$, and suppose for contradiction that $S'=S\cap[u,v]$ is not ample in the hypercube $[u,v]$. By \cite[Theorem 4]{La}, that implies that $S'$ is not \emph{totally asymmetric} in $[u,v]$, which means: there exists an interval $[u',v']\subseteq [u,v]$ such that $S'$ is invariant under the antipodal map of the hypercube $[u,v]$ (that maps each vertex of $[u',v']$ to its opposite vertex in $[u',v']$) and additionally $\varnothing \neq S'\cap [u',v']\neq [u',v']$. But this means that $[u',v']$ has at least one pair of antipodal vertices in $S'$, say $s,t\in S'$, and that means $[s,t]=[u',v']$. But then $S'\cap [s,t]$ is not totally asymmetric as well, which implies that $S'\cap[s,t]$ is not ample by \cite[Theorem  3]{La}. Since $s,t\in S'\subseteq S$ and $S\cap [s,t]=S'\cap [s,t]$, this contradicts item (7), so we conclude that $S\cap[u,v]$ must be ample for all $u,v\in U$. This concludes the proof of $(7)\Rightarrow (6)$, and the proof of the theorem. 
\end{proof}

\cref{thm:ample-elementary} has several interesting corollaries. The first one shows that ampleness of 
a set $S\subseteq U$ does not depend on the tuples defining  $S$ but only depends how the subgraph $H(S)$ of $H(U)$ induced by $S$ is embedded. 

\begin{corollary} \label{ample-embedding} Let $U=U_1\times\ldots\times U_m$ and $U'=U'_1\times\ldots\times U'_n$ be two Cartesian products and let $S$ be an ample subset of $U$. If the subgraph $H(S)$ induced by $S$ admits an isometric embedding $\varphi$ into the Hamming graph $H(U')$, then $\varphi(S)$ is an ample set of $U'$. 
\end{corollary}

\begin{proof} Let $S'=\varphi(S)$. Notice that the isometric embedding $\varphi$ maps boxes of $S$ to boxes of $S'$, furthermore $\varphi$ maps parallel boxes of $S$ to parallel boxes of $S$ and, vice-versa, any parallel boxes 
of $S'$ are images of parallel boxes of $S$. Consequently, the image of any gallery of $S$ is a gallery of $S'$. Now, pick any two parallel boxes $B_1',B'_2$ of $S$. By \cref{thm:ample-elementary}(5) it suffices to show that $B'_1$ and $B'_2$ can be connected in $S'$ by a geodesic gallery. Suppose that $B'_1=\varphi(B_1)$ and $B'_2=\varphi(B_2)$. Then $B_1,B_2$ are parallel boxes of $S$. By \cref{thm:ample-elementary}(5), $B_1$ and $B_2$ can be connected in $S$ by a geodesic gallery $\gamma(B_1,B_2)$. Since $\varphi$ is an isometric embedding, the image $\varphi(\gamma(B_1,B_2))$ is a geodesic gallery of $S'$ between $B'_1$ and $B'_2$, hence $\varphi(S)$ is ample in $U'$. 
\end{proof}

The second corollary allows to recognize ample sets in polynomial time.

\begin{corollary} \label{recognition} Given a set $S\subseteq U$, it can be decided in time polynomial in the size of the input  (i.e., in the size of $\vert S\vert$ and $m=|X|$) if $S$ is ample. 
\end{corollary}

\begin{proof} By \cref{thm:ample-elementary}(7), $S\subseteq U$ is ample if and only if for each pair $u,v\in S$ the intersection $S'=S\cap [u,v]$ is ample. This intersection can be computed in $O(m|S|)$ time using the  representation of all vertices as tuples. Namely, using this representation we compute the Hamming distance $d(u,v)$ and the Hammming distances from $u$ and $v$ to any $z\in S$. Then $z$ is included in $S'$ if and only if $d(u,v)=d(u,z)+d(z,y)$.  

Since $[u,v]$ is a hypercube of $U$, the intersection $S'=S\cap [u,v]$ is a subset of a  hypercube. Testing if a subset $S'$ of a hypercube is ample can be done in time polynomial in the size of $S'$ \cite{KnMa} by using Lawrence's total asymmetry \cite[Theorems 3\&4]{La}. Therefore testing if $S$ is ample can be done in time polynomial in the size of $S$ and the number of factors. 
\end{proof}

\cref{thm:ample-elementary}(4) characterizes ample sets as connected subsets of $U$ with ample strong-projections $S^e$ for all elementary minor-subproducts $e=\{ a,b\}\subseteq U_i$. As we noted, each strong-projection $S^e\subseteq U_e$ is the disjoint union of the hyperplane $(S(e))^e$ and the residue $(S_=(e))^e$. The hyperplane  $(S(e))^e$ is a thin part of $S^e$ because  $(S(e))^e$ can be viewed as a subset of the Cartesian product obtained from $U=U_1\times\ldots\times U_m$ by removing the factor $U_i$ (in fact by replacing $U_i$ by a single element corresponding to $\{a,b\}$ and neglecting $U_i\setminus \{ a,b\}$).  In case of binary products, $S_=(e)=\varnothing$, therefore  $S^e=(S(e))^e$. In the general case, the residue may have quite a general form. Therefore, it is worth to ask if ample sets can be characterized via their hyperplanes. This is the content of the following corollary:

\begin{corollary} \label{hyperplane} A set $S\subseteq U$ is ample if and only if $S$ is isometric and for every elementary minor-subproduct $e$ of $U$ the hyperplane $(S(e))^e$ is ample. 
\end{corollary}

\begin{proof} If $S$ is ample, then each hyperplane $(S(e))^e$ is ample by \cref{lem:hyperplanes-ample}.  Conversely, suppose that $S$ is isometric and each hyperplane of $S$ is ample, however $S$  is not ample. By Theorem  \ref{thm:ample-elementary}(6), for some $u,v\in U$, the set $S'=S\cap [u,v]$ is not ample in the hypercube $[u,v]$. Since $S$ is isometric in $U$, $S'$ is isometric in $[u,v]$. By \cite[Theorem 4]{BaChDrKo} (or by our \cref{thm:ample-elementary}(4)), there exists an elementary minor-subproduct $e=\{ a,b\}$ of $[u,v]$ such that the set $(S')^e$ is not ample in $[u,v]$. Necessarily $e\subseteq U_i$ for some $i\in X$,  whence $e$ is a minor-subproduct of $U$. Notice that the binary hyperplane  $(S')^e$ is a subset of the $e$-hyperplane $(S(e))^e$ of $S$ in $U$. Furthermore, $(S')^e$ is the intersection of $(S(e))^e$  with a full-dimensional subproduct of $U_e$. Since $(S(e))^e$ is ample in $U_e$ and $(S')^e$ is not ample, we obtain a contradiction with \cref{thm:ample=weaklyample}. 
\end{proof}

Recall now that $\BMP(U)$ denotes the set of all  minor-subproducts  $M=M(\Lambda,U)$ with $\Lambda=(\alpha_1,\ldots,\alpha_m)$ such that each $\alpha_i$ contains at most two blocks. For a tuple $a=(a_1,\ldots,a_m)\in U$, let $M(a)=M(\Lambda,U)\in \BMP(U)$, where $\Lambda=(\alpha_1,\ldots,\alpha_m)$ and $\alpha_i=\{ \{ a_i\}, U_i\setminus \{ a_i\}\}$ for each $i=1,\ldots,m$. Let $\BcMP(U)$ be the union of all $M(a)\in \BMP(U)$ over all $a\in U$. Analogously, define the set $\BcMP(V)$ for any subproduct $V=U_{i_1}\times\ldots\times U_{i_k}$ of $U$. Finally,  let $\BcMP^{**}(U)$ be the union of all $\BcMP(V)$ over all subproducts $V$ of $U$ and all $a\in V$. 

The 
minor-subproducts $M(a)$ from  $\BcMP^{**}(U)$ are extended minor-subproducts. Therefore one can ask if  $\BcMP^{**}(U)$-ample sets are ample. This is indeed the case due to the characterization of ample sets provided by \cref{thm:ample-elementary}: $S$ is ample if and only if $S\cap [u,v]$ is ample for any $u,v\in S$. Indeed, $V=[u,v]$ is a Boolean cube and a box of $U$, i.e., a full-dimensional subproduct of $U$. In Boolean cubes $V$, each minor-subproduct is binary and each binary minor-subproduct is a minor-subproduct from  $\BcMP^{**}(V)$ $C$. Thus $S\cap [u,v]$ is ample iff and only if $S\cap [u,v]$ is  $\BcMP^{**}(U)$-ample and if and only if $S\cap [u,v]$ is $\BMP^{**}(U)$-ample. Consequently, we obtain the following result:

\begin{corollary} \label{cor:binample}
\setcounter{claim}{0} 
For a set $S\subseteq U$, the following conditions are equivalent: 
\begin{enumerate}
    \item \label{item:bin-ample-bis} $S$ is $\BcMP^{**}(U)$-ample; 
    \item \label{item:bin-ample} $S$ is $\BMP^{**}(U)$-ample;
    \item \label{item:ample-bis} $S$ is ample. 
  \end{enumerate}
\end{corollary}

\section{Ampleness and push downs}\label{sec:pushdowns} 
In this section we characterize ample sets via push downs. 

\subsection{Push downs} In this section, we assume that there is a total order on the elements of each factor $U_i$ of $U$.  For convenience of notation, we assume $U_i=\{0,1,\ldots,|U_i|-1\}$ for each $i\in X$. We consider the tuple $\textbf{0}=(0,\ldots,0)$ as the \emph{basepoint} of $U$. For a tuple $u=(u_1,\ldots,u_m)$, we set  $p(u)=\sum_{i=1}^m u_i$. For a set $S\subseteq U$, we call the sum $\sum_{u\in S}p(u)$ the \emph{norm} of $S$. 

\begin{definition}[Push down]
    Let $S\subseteq U$, let $M=M(\Lambda,U)$ be a one-dimensional minor-subproduct, and let $\alpha_i$ be the unique nontrivial partition of $\Lambda$. Then every box $B\in \cB(\Lambda,U)$ is a clique whose elements only differ in their $i$-coordinate. The \emph{push down} operation $S[M\negthickspace\downarrow]$ 
    is defined boxwise: for each $B\in \cB(\Lambda,U)$, $S[M\negthickspace\downarrow ]\cap B$ is equal to the $|S\cap B|$ elements of $B$ with the smallest $i$-coordinates. Given a sequence of one-dimensional minor-subproducts $M^1,M^2,\ldots, M^k$ that have pairwise distinct supports, a \emph{serial push down} is the result of a sequence of push down operations. It is denoted by $S[M^1,M^2,\ldots, M^k\negthickspace\downarrow ]:=S[M^1\negthickspace\downarrow ][M^2\da]\ldots[M^k\negthickspace\downarrow ]$.

    A serial push down \emph{commutes} on $S$ if the result is invariant under permutations of the push down operations in the sequence.  Finally, a set $S\subseteq U$ is called \emph{stable by push downs} (or simply \emph{stable}) if for any one-dimensional minor-subproduct $M$,  $S[M\negthickspace\downarrow]=S$. 
\end{definition}

\begin{remark} Informally, the push down operation shifts the elements of $S$ as far down as possible, subject to the condition that the elements of every box $B$ of $M$ stay within $B$. Equivalently, $|S[M\da]\cap B|=|S\cap B|$ and, if $u,v\in B$ and $v_i<u_i$,  then 
$u\in S[M\negthickspace\downarrow]$ implies $v\in S[M\negthickspace\downarrow]$. In the particular case when the unique 
nontrivial partition $\alpha_i$ of $\Lambda$ consists of a single block $U_i$, each box $B$ is isomorphic to $U_i=\{0,1,\ldots,|U_i|-1\}$ and the push down will shift the elements of $S$ from $B$ to the left. 
\end{remark}

%
Notice that the operation of push down is idempotent:   $S[M,M\da]=S[M\da]$. 
Notice also that if $S[M\negthickspace\downarrow]\ne S$, then the norm of  $S[M\negthickspace\downarrow]$ is strictly smaller than the norm of $S$.  This implies that any set $S\subseteq U$ can be transformed by a sequence of push downs into a stable set. {When proving properties of ample sets and push downs, it will be convenient to decompose push down operations, as shown in the following lemma.}

\begin{lemma}\label{lem:pushdowndecomposition}
    Let $M\in \MP(U)$ be one-dimensional. Then there exists a sequence $e_1,e_2,\ldots,e_j$ of (not necessarily distinct) elementary subproducts such that
    \[
    S[M\da]=S[e_1,e_2,\ldots,e_j\da]
    \]
{for all $S\subseteq  U$.}
\end{lemma}
\begin{proof}
    Assume without loss of generality that $\alpha_1=\{P_1^1,P_2^1,\ldots,P_{\ell_1}^1\}$ is the unique nontrivial partition related to $M$. For $i=1,2,\ldots,\ell_1$, let $M^i$ be the one-dimensional extended subproduct whose only nontrivial block is $P_i^1$ (take $M^i=M_{\bot}$ if $|P_i^1|=1$). It follows directly from the definition that $S[M\da]=S[M^1,M^2,\ldots,M^k\da]$. 
    {If we can show that the push down operator of any extended subproduct can be decomposed into elementary subproduct push downs, then we are done. Take the extended subproduct $M^i$, and assume without loss of generality that $P_i^1=\{1,2,\ldots,p\}$. If $p=1$, then $[M^i\da]$ is the identity operator, so we can ignore the push down $[M^i\da]$. Otherwise, let $f_i\in \MP(U)$ be the elementary minor-subproduct with $\supp(f_i)=\supp(M)$ and with nontrivial block $\{i,i+1\}$ for $i=1,2,\ldots,p-1$. We claim that
    \begin{equation}\label{eq:elementarypushdowns}
        S'[M^i\da] = S'[f_1,f_2,\ldots,f_{p-1}\da][f_1,f_2,\ldots,f_{p-2}\da][f_1,f_2\ldots,f_{p-3}\da]\ldots[f_1,f_2\da][f_1\da].
    \end{equation}
    The reason for this is as follows: the only nontrivial boxes of the box partition related to $M^i$ are the one-dimensional boxes with $\supp(M)$-coordinates $1,2,\ldots,p$.}
    Applying the above sequence of push down operations is exactly the same as applying bubble sort within every box (considering elements of $S'$ smaller than elements of $U\backslash S'$). As a result, after the operations, the elements of $S'[M^k]$ are exactly the smallest elements in the box, which proves \eqref{eq:elementarypushdowns}. {Applying this to every extended subproduct $M^i$, we get
    \[
        S[M^1,M^2,\ldots,M^k\da]=S[e_1,\ldots,e_j\da]
    \]
    for some sequence $e_1,\ldots,e_j$ of elementary minor-subproducts, which completes the proof.}
\end{proof}

\begin{lemma}\label{lem:pushdownample}
    Let $S\subseteq U$ be ample. Then any serial push down of $S$ is ample.
\end{lemma}
\begin{proof}
    By \cref{lem:pushdowndecomposition} we only need to show that $S[e\da]$ is ample for any elementary $e\in\MP(U)$, since then the result follows by induction. Assume that the nontrivial block of $e$ is $\{a,b\}$ with $a<b$, and $\supp(e)=i$. For elementary minor-subproducts we can interpret the result of  a push down in terms of projections: let $u'\in e$ with $u'_i=\{a,b\}$, and let $F(u')=\{u^a,u^b\}$ with $u^a_i=a$ and $u^b_i=b$. We have $u^a\in S[e\da]$ if and only if $F(u')$ contains an element of $S$, which happens if $u'\in S_e$. Likewise, $u^b\in S[e\da]$ if and only if $u'\in S^e$. It follows that the set of elements of $S[e\da]$ with $i$-coordinate $a$ (respectively, $b$) is isomorphic to the set of elements of $S_e$ (respectively, $S^e$) with coordinate $\{a,b\}$.
    
    We show ampleness using \cref{thm:ample-elementary}(6). Let $[u,v]$ be an interval of $U$, and assume $u_j\leq v_j$ for all $j\in [m]$ ({this is without loss of generality, since we can always reassign $u,v$ to be the minimum, respectively maximum of the interval}). There are a few options:
    \begin{itemize}
        \item $u_i,v_i\notin \{a,b\}$. Then $[u,v]\cap S[e\da]=[u,v]\cap S$, which is ample.
        \item $\{u_i,v_i\}=\{a,b\}$. Then the push down operation $e$ within $[u,v]$ is equivalent to a push down operation as defined in \cite{BaChDrKo}. By \cite[Proposition 3]{BaChDrKo}, $S\cap [u,v]$ is also ample.
        \item $\{u_i,v_i\}\cap \{a,b\}=a$. Assume $u_i=a$, and let $u',v'\in e$ such that $u\in F(u')$ and $v\in F(v')$. Then $[u,v]\cap S[e\da]$ is isomorphic to $[u',v']\cap S_e$. The latter is ample by \cref{lem:ample-projection,thm:ample=weaklyample}.
        \item $\{u_i,v_i\}\cap \{a,b\}=b$. Assume $u_i=b$, and let $u',v'\in e$ such that $u\in F(u')$ and $v\in F(v')$. Then $[u,v]\cap S[e\da]$ is isomorphic to $[u',v']\cap S^e$. The latter is ample by \cref{lem:superscriptample,thm:ample=weaklyample}.
    \end{itemize}
    Since all intersections with intervals are ample, we conclude that $S[e\da]$ is ample.
\end{proof}

\begin{lemma}\label{lem:length2pushdown}
    Let $S\subseteq U$. Then $S$ is weakly isometric if and only if serial push downs of length 2 commute on $S$, i.e. $S[M^1,M^2\da]=S[M^2,M^1\da]$ {for all one-dimensional $M^1,M^2\in \MP(U)$}.
\end{lemma}
\begin{proof}
    Suppose $S$ is not weakly isometric. Then there are $u,v\in S$ at distance 2 with $[u,v]\cap S=\{u,v\}$. Since $u,v$ have distance 2, they have two coordinates where they differ, say $u$ has coordinates $u_1,u_2$ where $v$ has $v_1,v_2$. Let $M^1$ and $M^2$ be elementary minor-subproducts with nontrivial blocks $\{u_1,v_1\}$ and  $\{u_2,v_2\}$, respectively. {Then $S[M^1,M^2\da]$ contains the element in $[u,v]$ with coordinates $\min(u_1,v_1),\max(u_2,v_2)$}, while $S[M^2,M^1\da]$ does not contain that element, thus $[M^1\da]$ and $[M^2\da]$ do not commute.

    To prove the converse assertion we may assume that $m=2$ and that the nontrivial partitions of $M^1,M^2$ are equal to the co-trivial partitions $\{U_1\}$ and $\{U_2\}$, respectively, since in the general case the effects of the push down operators are restricted to two-dimensional boxes formed by the non-trivial partitions of $M^1$ and $M^2$. For $i=1,2,\ldots,|U_1|$, let $a_i^S$ be the number of elements of $S$ with first coordinate $i$, and let $\sigma_S:[|U_1|]\to[|U_1|]$ be a permutation such that $a_{\sigma_S(1)}^S\geq a_{\sigma_S(2)}^S\geq\ldots\geq a_{\sigma_S(|U_1|)}^S$. For any weakly isometric set $S$, denote the sequence $a^S_1,a^S_2,\ldots,a^S_{|U_i|}$ by $A(S)$ and the sequence $a^S_{\sigma_S(1)}, a^S_{\sigma_S(2)},\ldots, a^S_{\sigma_S(|U_1|)}$ by $A'(S)$. Note that $S[M^2\da]$ is uniquely determined by $A(S)$. We claim that for any weakly isometric $S$, $A(S[M^1\da])=A'(S[M^1\da])$. Consider any element $a^{S[M^1\downarrow]}_i$ from $A(S[M^1\da])$, and let $(\sigma_S(i),j_1),(\sigma_S(i),j_2),\ldots, (\sigma_S(i),j_{a^S_{\sigma_S(i)}})$ be the elements of $S$ with first coordinate $\sigma_S(i)$. Now, if there is any $k\in U_1$ with $a_k\geq a_{\sigma_s(i)}$, we claim that the elements $(k,j_1),\ldots,(k,j_{a^S_{\sigma_S(i)}})$ must be in $S$: if one of them is not, say $(k,r)\notin S$, then there must also be another index $s$ such that $(k,s)\in S$ and $(\sigma_S(i),s)\notin S$ {(since there are at least as many elements with coordinate $k$)}, but that violates weak isometricity of $S$: there is path of length 2 between $(\sigma_S,r)$ and $(k,s)$). {Summarizing, if a tuple $(\sigma_S(i),j)$ is in $S$, then there are at least $i$ tuples in $S$ with second coordinate $j$, namely $(\sigma_S(i'),j)$, $i'\leq i$. Therefore, if $(\sigma_S(i),j)\in S$, then $(i,j)\in S[M^1\da]$.} 
    This implies $a^{S[M^1\downarrow]}_i\geq a^S_{\sigma_S(i)}$. Since this inequality holds for all $i\in U_1$, and since $\sum a^S_i=|S|=|S[M^1\da]|=\sum a^{S[M^1\downarrow]}_i$, the inequality $a^{S[M^1\downarrow]}_i\geq a^S_{\sigma_S(i)}$ must hold with equality everywhere. This implies $A(S[M^1\da])$ is sorted already, so $A(S[M^1\da])=A'(S[M^1\da])=A'(S)$.

    {The operation $[M^2\da]$ has no effect on $A(S)$, so we have $A(S[M^1,M^2\da])=A(S[M^1\da])=A'(S)$. Moreover $A(S[M^2])=A(S)$, and as one can easily verify that $S[M^2]$ is weakly isometric, we can repeat the argument we did for $S$ to get $A(S[M^2, M^1\da])=A'(S[M^2\da])=A'(S)$. Combining these yields $A(S[M^1,M^2\da])=A(S[M^2,M^1\da])$.
    The last push down of $S[M^1,M^2\da]$ is $M^2$, so it follows from $A(S[M^1,M^2\da])=A'(S)$ that $S[M^1,M^2\da]$ contains precisely all the elements $(i,j)$ with $i\in U_1$ and $j\leq a^S_{\sigma_S(i)}$.} 

    {We saw that the number of elements of $S[M^1,M^2\da]$ per first coordinate (per `column') is non-increasing. With an analogous argument, the number of elements for each second coordinate (per `row') of $S[M^2,M^1\da]$ is non-increasing. If $(i,j)\in S[M^2,M^1\da]$, that means that the $j$-row of $S[M^2\da]$ has at least $i$ elements. But then, for any $j'\leq j$, the $j'$-row of $S[M^2\da]$ has at least as many elements, hence $(i,j')\in S[M^2,M^1\da]$. Since $A(S[M^2,M^1\da])=A'(S)$, we conclude that $A(S[M^2,M^1\da]$ contains exactly those $(i,j)$ with $j\leq a^S_{\sigma_S(i)}$. Since this is the same set as $S[M^1,M^2\da]$, this concludes the proof.}
\end{proof}

\begin{remark} In the previous proof, if we consider the bipartite graph with vertex set $U_1\cup U_2$, with an edge from $i\in U_1$ to $j\in U_2$ precisely when $(i,j)\in S$, then $S$ induces a Ferrers graph if and only if $S$ is weakly isometric. The {`only if' direction of} this lemma is then equivalent to saying that the shape of a Ferrers diagram is uniquely defined by its Ferrers graph up to isomorphism, by sorting both the rows and columns of the adjacency matrix by number of ones. 
\end{remark}

\subsection{Characterization of ampleness via serial push downs} The goal of this subsection is to prove the following result: 

\begin{theorem}\label{thm:pushdown}
    A set $S\subseteq U$ is ample if and only if every serial push down commutes on $S$.
\end{theorem}
\begin{proof}
    First, suppose that $S$ is ample. Consider the serial push down $S[M^1,M^2,\ldots,M^k\da]=S[M^1\da][M^2\da]\ldots[M^k\da]$, and let $i\leq k-1$. By \cref{lem:pushdownample} we know that $S[M^1,M^2,\ldots,M^{i-1}\da]$ is ample. Since any ample set is weakly isometric (\cref{prop:ample-is-isometric}), by \cref{lem:length2pushdown} it follows that 
    \[
        S[M^1,M^2,\ldots,M^{i-1}\da][M^i\da][M^{i+1}\da]=S[M^1,M^2,\ldots,M^{i-1}\da][M^{i+1}\da][M^{i}\da].
    \]
    So we can swap the order of any two adjacent minor-subproducts in the serial push down. Since we can create any permutation of the minor-subproducts in this manner, we conclude that every serial push down commutes on $S$.

    Conversely, suppose that any serial push down commutes on $S$. Pick any interval $[u,v]$, and let $I=\{i_1,i_2,\ldots,i_{|I|}\}$ be the set of indices $i$ for which $u_i\neq v_i$. For all $i\in I$, let $M^i\in \EEMP$ such that $\supp(M^i)=i$ and the unique nontrivial block of $M^i$ is $\{u_i,v_i\}$. Then within the interval $[u,v]$, the serial push down $S[M^{i_1},M^{i_2},\ldots,M^{i_{|I|}}\da]$ corresponds to a ``complete serial push down'' defined as in \cite{BaChDrKo}. Since the serial push down commutes, it follows from \cite[Corollary 3]{BaChDrKo} that $S\cap[u,v]$ is ample in $[u,v]$, hence by \cref{thm:ample-elementary}(6), $S$ is ample.
\end{proof}

Now, consider push downs with respect to the one-dimensional mixed minors. Such minors are uniquely determined by the unique factor of $U$ whose partition is co-trivial. For each $i=1,\ldots,m$, let $M^i_+$ be the one-dimensional mixed minor-subproduct defined by the generalized partition $\Lambda^i_+=(\alpha_1,\ldots,\alpha_{i-1},\alpha_i,\alpha_{i+1},\ldots,\alpha_m)$, where each of the partitions $\alpha_j, j\ne i$ is trivial and the partition $\alpha_i$ is co-trivial. The boxes of $\BPart(\Lambda^+_i,U)$  are parallel cliques of size $|U_i|$. In this case, we denote  $S[M^{i_1}_+,M^{i_2}_+,\ldots,M^{i_k}_+\da]$ by  $S[i_1,i_2,\ldots,i_k\da]$. If $S$ is ample, then  for any enumeration $1,\ldots,m$ of $X$, $S[1,\ldots,m\da]$ is called a \emph{complete (serial) push down} of $S$. {The resulting set is stable by push downs, since, using Theorem \ref{thm:pushdown} and idempotence of push downs, we have for any index $i$} 
\[ S[1,\ldots, m,i\da]= S[1,\ldots, i-1,i+1,\ldots,m,i,i\da]=S[1,\ldots, i-1,i+1\ldots,m,i\da]=S[1,\ldots, m\da].\] 
Consequently, we obtain the following corollary of Theorem \ref{thm:pushdown}: 

\begin{corollary}\label{stable} If $S\subseteq U$ is an ample set, then there exists a stable ample set $S_*$ such that $S_*=S[1,\ldots,m\da]$ for any complete serial push down $S[1,\ldots,m\da]$ of $S$ with respect to one-dimensional mixed minors.  
\end{corollary}

\section{Box complexes of ample sets}\label{sec:boxcomplexes} In this section we define box/prism complexes of ample sets and investigate their properties: $f$-vectors, the Euler characteristic, corner peelings, and contractibility.  

\begin{definition} [Box and prism complexes, $f$-vectors] For a set $S\subseteq U$, the \emph{box complex}  $\BBox(S)$ of $S$  {is the set of} all boxes $B\in \BBox(U)$ included in $S$. The \emph{dimension} of a box $B=B_1\times\ldots\times B_m$ is $\dim(B)=\sum_{i=1}^m \dim(B_i)=\sum_{i=1}^m (|B_i|-1)$. The \emph{prism complex} $||\BBox(S)||$ of $\BBox(S)$ is obtained by replacing each box $B=B_1\times\ldots\times B_m\in \BBox(S)$ by the prism $\Pi(B)$ of dimension $\dim(B)$, which is the Cartesian product of $m$ regular Euclidean simplices  of dimensions $\dim(B_1),\ldots,\dim(B_m)$, respectively. 

Let $f_i(S)$ denotes the number of boxes of $\BBox(S)$ of dimension $i$. The vector $f(S)=(f_0(S),f_1(S),\ldots,)$ is called the \emph{$f$-vector} of the box complex $\BBox(S)$ (and of $||\BBox(S)||$ or of $S$). 
\end{definition}

{Box and prism complexes and $f$-vectors are closely connected conceptually. }By Lemma \ref{convex-Hamming}, each box $B\in \BBox(S)$ is a Cartesian product of cliques. Therefore,  replacing each clique by a regular simplex we realize each box $B$ by a  Cartesian product of simplices, i.e., by a prism $\Pi(B)\in ||\BBox(S)||$. 
The dimension of a box $B$ is  the topological dimension of the prism $\Pi(B)$. Since the intersection of two boxes $B',B''\in \BBox(S)$ is empty or a box  of $\BBox(S)$ and the intersection of the prisms $\Pi(B')$ and $\Pi(B'')$ is empty or a prism, $||\BBox(S)||$ is a cell complex in which all cells are prisms, justifying the name ``prism complex''. The graph $H(S)$ is called the \emph{1-skeleton of the prism complex  $||\BBox(S)||$}: its vertices are the 0-dimensional prisms (i.e., the vertices corresponding to the tuples of $S$) and its edges are the 1-dimensional prisms.  If $B\in \BBox(S)$ and $B'$ is a box contained in $B$, then $B'$ is called a \emph{face} of $\BBox(S)$. The boxes of $\BBox(S)$ will be also called \emph{faces}. The faces which are maximal by inclusion are called \emph{facets} of $\BBox(S)$. A box complex  $\BBox(S)$  is called a \emph{bouquet of prisms} if there exists a vertex $v_0\in S$ belonging to all facets of   $\BBox(S)$; $v_0$ is called the \emph{origin} of the bouquet.

\subsection{Push downs and $f$-vectors}
In general, a push down preserves the number of elements  (i.e., $f_0(S[M\da])=f_0(S)$) but may increase the other coordinates of $f(S)$. Similarly to the binary case (see \cite{BaChDrKo}), the following result shows that the push down  preserves the $f$-vectors of ample sets. Furthermore, we show that a complete serial push down with respect to one-dimensional mixed minors  ends up with the same 
stable ample set, which  is a bouquet of prisms: 

\begin{proposition} \label{f-vector-push-down}  For any set $S\subseteq U$, any one-dimensional minor-subproduct $M$, and any $i$, $f_i(S)\le f_i(S[M\da])$. If $S$ is ample, then 
\begin{enumerate}
    \item $f_i(S)=f_i(S[M\da])$ for any $i$, thus $f(S[M\da])=f(S)$, and
    \item if $S_*$ is the stable ample set obtained by any complete serial push down  $S[1,\ldots,m\da]$ of $S$, then $S_*$ is a bouquet of prisms with origin $\bf{0}$.   
\end{enumerate}
\end{proposition}

\begin{proof} In view of Lemma \ref{lem:pushdowndecomposition}, to prove the inequality  $f_i(S)\le f_i(S[M\da])$ for any one-dimensional minor-subproduct $M$ and any $i$, it suffices to prove that 
$f_i(S)\le f_i(S[e\da])$ for any elementary minor-subproduct $e=\{ a,b\}$ and any $i$. Let $\{ a,b\}\subseteq \{ 0,1,\ldots, |U_{j}|-1\}=U_j$ and suppose that $a<b$.  A box $B=B_1\times\ldots\times B_m$ is called an \emph{$a$-box} if $B_j\subseteq U_j\setminus \{ b\}$ (and a \emph{$b$-box} if $B_j\subseteq U_j\setminus \{ a\}$). Equivalently, $B$ is an $a$-box if $B\subseteq W_a(e)\cup W_=(e)$. To each $a$-box $B'$ we can canonically associate a $b$-box $B''$ of the same dimension as $B'$, obtained by setting in all tuples $u'\in B'$ the $j$-coordinate to $b$ if they had $j$-coordinate equal to $a$. Then $B'$ and $B''$ are called \emph{twins}. Notice that $B'\cap B''=B'\cap W_=(e)$. If $B'\subseteq W_=(e)$, then $B'=B''$. Otherwise, if $B'\cap W_a(e)\ne\varnothing$, then  the union $B'\cup B''$ is a box of one dimension larger than $B'$ and $B''$.

We prove the inequality $f_i(S)\le f_i(S[e\da])$ by providing an injective function $g:\BBox(S)\to\BBox(S[e\da])$ that preserves the dimension $i$ of each box. The function $g$ is defined as follows. If $B$ is an $a$-box and the twin of $B$ is not contained in $S$, then $g(B)=B$. In all other cases, $g(B)$ is equal to the twin of $B$. We claim this is an injection: the preimage of any box $B$ can only consist of $B$ and/or its twin $B'$. {If $B=B'$, then there can only be one box in the preimage, and if $B\neq B'$, then $B$ can only be in the preimage if $B'\notin \BBox(S)$, and then the preimage only consists of one box.} Now we still need to show that $g$ is well-defined. If $B$ contains an edge from the parallelism class $\Theta(e)$, then for any vertex $u\in W_a(e)\cap B$ its unique neighbor $v\in W_b(e)$ also belongs to $B$. But that means $B$ is its own twin, and $B$ is not affected by the push down operation, hence $g(B)=B\in \BBox(S[e\da])$. Otherwise, $B$ is an $a$-box or a $b$-box. If $B$ is an $a$-box, let $B'$ be its twin. If $B'=B$, that means $B\subseteq W_=(e)$, so $B$ is not affected by the push down, and $g(B)=B\in\BBox(S[e\da])$. On the other hand, if $B'\neq B$ and $B'\notin \BBox(S)$, then $g(B)=B\in \BBox(S[e\da])$. If $B'\neq B$ and $B'\in \BBox(S)$, that implies $B\cup B'$ is not affected by the push down, and thus $B\cup B'\subseteq S[e\da]$, which implies $g(B)=B'\in \BBox(S[e\da])$. Finally, suppose that $B$ is a $b$-box. If $B\subseteq W_=(e)$, then again $g(B)=B\in \BBox(S[e\da])$. And if $B\cap W_b(e)\ne \varnothing$, then $B\neq B'$. However, for every vertex $v\in B$ whose $j$-coordinate is $b$, $S[e\da]$ contains its adjacent vertex $u\in U$, which has its $j$-coordinate equal to $a$. Thus $g(B)=B'\in \BBox(S[e\da])$, {and we conclude that $f_i(S)\leq f_i(S[e\da])$.}

Now, suppose that $S$ is ample. Then $S[e\da]$ is also ample by Lemma \ref{lem:pushdownample}. If $f_i(S)\ne f_i(S[e\da])$ for some $i$, then the function $g$ is not surjective (as a function $\BBox(S)\to\BBox(S[e\da])$). Hence there exists an $i$-dimensional box $B'$ of $S[e\da]$ that is not the image of an $i$-dimensional box of $S$. Let $B''$ be the twin of $B'$. If any vertex of $B'$ has $b$ as its $j$-coordinate, that would imply that each vertex with a $b$-coordinate has an $a$-neighbor (since it is in $S[e\da]$), and therefore $B''\subseteq S$ and $g(B'')=B'$. So $B'$ is an $a$-box. Let $B'_a=B'\cap W_a(e)$ and $B'_==B'\cap W_=(e)$. We have $B_=\subseteq S$, and $B'$ is not included in $S$ (otherwise, if $B''\subseteq S$ then $g(B'')=B'$, and if $B''\nsubseteq S$, $g(B')=B'$). We conclude that $B'_a$ must be nonempty and not be a subset of $S$. {Consider the $(i+1)$-dimensional box $B=B'\cup B''$ and let $S'=S\cap B$. We claim that the set $S'$ shatters in $B$ the elementary minor-subproduct $e=\{ a,b\}$: note that the boxes of this minor-subproduct consist of single elements of $B_='$, and of the edges of $\Theta(e)\cap B$. We have $B_='\subseteq S$, as $B'_=\subseteq S[e\da]$, and is unaffected by the push down; and at least one element of $S$ must be in every edge of $\Theta(e)\cap B$, because for every such edge the element with coordinate $a$ is in $S[e\da]$. By ampleness of $S'$ (Theorem \ref{thm:ample=weaklyample}(3)), $S'$ strongly shatters minor-subproduct $e$ in $B$. The only two copies of $e$ in $B$ are $B'$ and $B''$, so $B'\subseteq S$ or $B''\subseteq S$. But if $B''\subseteq S$ then $g(B'')=B'$, and otherwise $g(B')=B'$, giving us a contradiction. Hence $g$ is a bijection.}
Consequently, $f_i(S)=f_i(S[e\da])$ for all $i$, hence the ample sets $S$ and $S[e\da]$ have the same $f$-vector, establishing (1).

To prove (2), let  $S_*$ be the stable ample set obtained by any complete serial push down  $S[1,\ldots,m\da]$ of $S$. Let $B$ be a maximal box of $S_*$, say $B=B_1\times\ldots\times B_m$. Suppose by way of contradiction that $B$ does contain the origin $\bf{0}$. Since $B$ is a maximal box of $S_*$, necessarily there exists a factor of $B$, say $B_1$,   such that $0\in U_1\setminus B_1$ and the box $B'=(B_1\cup \{ 0\})\times B_2\times\ldots\times B_m$ is not included in $S_*$. Therefore there exists a tuple $(0,u_2,\ldots,u_m)\notin S_*$, where $u_2\in B_2,\ldots,u_m\in B_m$ and such that $(k,u_2,\ldots,u_m)\in S_*$ for some $k\in U_1$. But this implies that  $S_*[1\da]\ne S_*$, contrary to the choice of $S_*$. 
\end{proof}



We conclude this subsection with a property of stable ample sets $S_*\subseteq U$. On the elements of $U=U_1\times\ldots\times U_m$ with $U_i=\{ 0,1,\ldots,|U_i|-1\}, i=1,\ldots,m$, consider the partial order  $\leq$, where $x\leq y$ if $x_i\leq y_i$ for all $i=1,\ldots,m$.  For an element $t\in U$, let $B_\leq(t)=\{ t'\in U: t'\leq t\}$. 
Notice that if $t=(b_1,\ldots,b_m)$, then $B_\leq(t)$ coincides with the box $\{ 0,1,\ldots,b_1\}\times\ldots\times \{ 0,1,\ldots,b_m\}$. 

\begin{lemma} \label{l:boxBleq} If $S_*\subseteq U$ is a stable ample set and $t\in S_*$, then the box $B_\leq(t)$ belongs to $S_*$. 
\end{lemma}

\begin{proof} Let $t=(b_1,\ldots,b_m)$. We proceed by induction on the norm $p(t)=\sum_{i=1}^m b_i$ of $t$.  Pick any element $t'=(b'_1,\ldots,b'_m)\in B_\leq (t)$ different from $t$. Then $b'_i\le b_i$ and $b'_j<b_j$ for some coordinate $j$. Suppose without loss of generality that $b'_1<b_1$. Let $t''=(b'_1,b_2,\ldots,b_m)$. Then obviously $t'\le t''\leq t$ and $p(t'')<p(t)$. If $t''\notin S_*$, then performing a push-down along the first coordinate, we will deduce that the set $S_*$ is not stable by push-downs. Therefore, $t''\in S_*$. Since $p(t'')<p(t)$ and $t'\in B_\leq(t'')$, by induction hypothesis we obtain that 
$t'\in S_*$, as required. 
\end{proof}

\subsection{Euler characteristic}
For a set $S\subseteq U$, we will denote by $\chi(S)=\sum_{i=0}^{\infty}(-1)^if_i(S)$ the \emph{Euler characteristic} of its box-complex  $\BBox(S)$. A \emph{corner} of a set $S\subseteq U$ is a vertex $s\in S$ belonging to a unique maximal box $B$ of $\BBox(S)$. 

\begin{theorem} \label{Euler-characteristic}  A set $S\subseteq U$ is ample if and only if $\chi(S\cap V)=1$ for every full-dimensional subproduct $V$ with $S\cap V\neq \varnothing$.
\end{theorem}

\begin{proof} First, we show the `only if'-direction. Wo only need to show that if $S$ is ample and nonempty, then $\chi(S)=1$. The result then follows for all $V$ from \cref{thm:ample=weaklyample}. Let $S_*$ be the stable ample set obtained by any complete serial push down  $S[1,\ldots,m\da]$ of $S$. By Proposition \ref{f-vector-push-down}, $S_*$ is a bouquet of prisms with origin $\bf{0}$ and $S_*$ has the same $f$-vector as $S$. Therefore, it suffices to show that  $\chi(S_*)=1$, which we do by induction on $|S_*|$. For $|S_*|=1$ we have $S_*=\{\bf{0}\}$ and $\chi(S_*)=1$. Then take some $S$ with $|S_*|\geq 2$, and suppose that we have shown the statement for all ample sets $S$ of smaller size. On the elements of $S_*$, consider the partial order  $\leq$ defined above. Each maximal box has the form $B=B_1\times\ldots\times B_m$, where $B_i=\{ 0,b^{1}_i,\ldots b^{j_1}_i\}\subseteq U_i$ with $0\le b^{1}_i\le \ldots \le b^{j_1}_i$. Then clearly, $(B,\leq)$ has a unique maximal element $t^B=(b^{j_1}_1,\ldots,b^{j_m}_m)$. 

\begin{claim} \label{c:boxes} Each maximal box $B$ of $S_*$ coincides with the box $B_\leq(t^B)=\{t\in U: t\leq t^B\}$.  Furthermore, 
the unique maximal element $t^B$ of $B$ is a corner of $S_*$. 
\end{claim}

\begin{proof} Since $t^B$ is the unique maximal element of $(B,\leq)$, $B$ is included in the box $B_\leq(t^B)$. 
Since $t^B\in S_*$, by Lemma \ref{l:boxBleq}, the box $B_\leq(t^B)$ is included in $S_*$. Since $B$ may not be strictly contained in a larger box of $S_*$, 
necessarily $B=B_\leq(t^B)$. 

Secondly, we show that $B$ is the only maximal box of $S_*$ that contains $t^B$. Suppose that $t^B\in B'\neq B$ for some maximal box $B'\subseteq S_*$. That means $B'\nsubseteq B$, so $B'$ has elements with coordinates not occurring in $B$, and $t^B$ then has a neighbor $t'\in B'\backslash B$. Since neighboring elements are always comparable with $\leq$, either $t'\leq t^B$ or $t^B\le t'$. If $t'\leq t^B$, then $t'\in B_\leq(t^B)$ and by the equality $B=B_\leq(t^B)$ we deduce that $t'\in B$, contrary to the choice of $t'$. Therefore $t^B\leq t'$. {By what we found before, we have $B'=B_\le(t^{B'})$.} From $t'\in B'$ and $B'=B_\le(t^{B'})$ it follows that $t^B\le t'\leq t^{B'}$. This yields $B\subset B_{\leq}(t^{B'})=B'$, which contradicts maximality of $B$. Consequently, $t^B$ is a corner of $S_*$. 
\end{proof}

Now fix a maximal box $B$, and let $S'_*=S_*\backslash\{t^B\}$. {We claim that $S'_*$ is a stable bouquet of prisms with origin $\bf{0}$.} Let $B'$ be any maximal box of $S'_*$. If $B'$ is a maximal box of $S_*$ then $\mathbf{0}\in B'$. Otherwise, there is a maximal box $B''$ of $S_*$ containing $B'$, implying that $t^B\in B''$, which in turn implies $B''=B$ and thus $B'\subseteq B$. That means $B'$ is a maximal box of $B\backslash \{t^B\}$. The maximal boxes of $B\backslash \{t^B\}$ are all of the form $\{t\in B: t_i\neq t^B_i\}$ for some $i$, and in particular they all contain $\bf{0}$. {So $S'_*$ is a stable bouquet of prisms, allowing us to} use the induction hypothesis to get $\chi(S'_*)=1$. Since $|S_*|\geq 2$, every maximal box has size at least 2, so $B\backslash\{t^B\}$ is also a nonempty bouquet of prisms with origin $\bf{0}$. By induction hypothesis we have $\chi(B\backslash\{t^B\})=1$. Moreover, the elements $f_i(B)$ correspond to the \mbox{$f$-vector} of a product of simplices, and from Euler's formula for polytopes we obtain $\chi(B)=1$. {Now we compute $\chi(S_*)$.} Consider a box $B'$ of $S_*$. If $B'$ contains $t^B$, then $B'$ is a box of $S_*$ and $B$, but not of $S_*'$ and $B\backslash \{t^B\}$. If $B'$ does not contain $t^B$, then there are two options: if $B'\subseteq B$, then $B'$ is a box of $S_*$, $B$, $S_*'$, and $B\backslash \{t^B\}$; and if $B'\nsubseteq B$, then $B'$ is a box of $S_*$ and $S_*'$, but not of $B$ and $B\backslash \{t^B\}$. It follows that for each $i$ we have $f_i(S_*)=f_i(S'_*)+f_i(B)-f_i(B\backslash \{t^B\})$, which implies $\chi(S_*)=\chi(S'_*)+\chi(B)-\chi(B\backslash \{t^B\})=1+1-1=1$, thus $\chi(S)=1$. 

For the converse direction, suppose $\chi(S\cap V)=1$ for every full-dimensional subproduct $V$ with $S\cap V\neq\varnothing$. In particular for all $s,t\in S$, $\chi(S\cap V)=1$ for all subproducts $V\subseteq [s,t]$. By \cite{Wi} and \cite[Corollary 2]{BaChDrKo} 
this implies that $S\cap [s,t]$ is ample for all $s,t\in S$, which by \cref{thm:ample-elementary}(7) implies that $S$ is ample.
\end{proof}

\subsection{Corner peelings} A \emph{corner peeling} of a set $S\subseteq U$ is a total ordering $s_1,\ldots,s_n$ of the vertices of $S$ such that for any $i=n,\ldots,1$, $s_i$ is a corner of the set $S_i=\{s_1,\ldots,s_i\}$ (see the previous subsection for the definition of corner). An \emph{isometric dismantling} of a set $S\subseteq U$ is a total ordering $s_1,\ldots,s_n$ of $S$ such that for any $i=n,\ldots,1$, $S_i$ is an isometric set of $U$.  In the binary case, in \cite{ChChMoWa} it was shown that ampleness is preserved by corner peelings, and an example of an ample set without corners was provided. {We extend the result that ampleness is preserved to Cartesian products.}

\begin{lemma}\label{lem:corners}
    Suppose $S\subseteq U$ is ample, and let $s\in S^*$ such that $S\cup \{s\}$ is isometric. Then $s$ is a corner of $S\cup \{s\}$ and the sets $S\cup\{s\}$ and $S^*\setminus \{s\}$ are ample. 
\end{lemma}
\begin{proof}
    This proof is an adaptation of that of \cite[Lemma 4.1]{ChChMoWa}. Let $C[s]$ be the convex hull in $U$ of $s$ and of the neighbors of $s$ contained in $S$. That is, $C[s]$ is the subproduct whose factors contain all coordinates of $s$ and of its neighbors in $S$. We claim that $C[s]\subseteq S\cup \{s\}$. Suppose there is $t\in C[s]$ that is not in $S\cup \{s\}$. Then $t$ is not $s$, and not a neighbor of $s$: if $t$ were a neighbor, then the coordinate in which $t$ differs from $s$ does not occur in $s$ or any other neighbor of $S$, which would imply $t\notin C[s]$. Then $t$ has distance at least 2 from $s$. We have $[s,t]\subseteq C[s]$ by convexity of $C[s]$, and by isometricity of $S^*$, $s$ and $t$ must be connected within $[s,t]$ by vertices from $S^*$. However, every neighbor of $s$ in $[s,t]$ is in $S$, so we have a contradiction, hence $C[s]\subseteq S\cup \{s\}$. {Moreover, any box of $\BBox(S\cup \{s\})$ that contains $s$, is contained in $C[s]$, since all coordinates occurring in such a box would occur in one of the neighbors of $s$.} Hence $C[s]$ is the only maximal box of $\BBox(S\cup \{s\})$ containing $s$, thus $s$ is a corner of $S\cup \{s\}$. 

    We show ampleness of $S\cup \{s\}$ by \cref{thm:ample-elementary}(5). Let $B',B''$ be two parallel boxes of $S\cup\{s\}$. If $s\notin B'\cup B''$, then by ampleness of $S$, there is a geodesic gallery connecting $B'$ and $B''$. Otherwise, suppose without loss of generality that $s\in B'$. Clearly $B'\subseteq C[s]$. First, suppose that $B'=C[s]\cap \conv(B'\cup B'')$. Let $s'$ be the element of $B''$ with the smallest distance to $s$. We have $d(s,s')=d(B',B'')$. Since by assumption $S\cup \{s\}$ is isometric, there is a path $s,s_1,s_2,\ldots,s_{d(B,B')}=s'$ within $S\cup \{s\}$. However, $s_1\in C[s]$ as it is a neighbor of $S$, and since this is a shortest path, $s_1\in \conv(B'\cup B'')$, hence by our assumption on $B'$ we have $s_1\in B'$. But that means that we have a path of length $d(B',B'')-1$ between $s_1\in B'$ and $s'\in B''$, which is a contradiction, hence $B'\subsetneq C[s]\cap \conv(B'\cup B'')$.
    
    This implies that there must be $t=(t_1,t_2,\ldots,t_m)\in \left(C[s]\cap \conv(B'\cup B'')\right)\backslash B'$. Specifically, there must be one coordinate $t_i$ of $t$ that does not occur in $B'$. Since $B',B''$ are parallel boxes, {each of the two boxes can therefore not contain multiple $i$-coordinates, hence all elements of $B''$ have $i$-coordinate $t_i$.}
    Let $s^i$ be the neighbor of $s$ that has $i$-coordinate $t_i$. Since $s,t\in C[s]$, $s^i\in [s,t]$,  and $C[s]$ is convex, we have $s^i\in C[s]$. The box $B'$ only contains one $i$-coordinate, so there exists a box $B^i$ parallel to $B'$ and neighboring $B'$ such that $B^i$ has $i$-coordinate equal to $t_i$. Since $s^i\in C[s]$, $B^i$ must be contained in $C[s]\subseteq S\cup\{s\}$. We have $d(B^i,B'')'=d(B',B'')-1$, since the $i$-coordinate of $B^i$ matches that of $B''$. By ampleness of $S$, there is a geodesic gallery from $B^i$ to $B''$ of length $d(B^i,B'')$. Extending this gallery with $B'$ gives us a geodesic gallery from $B'$ to $B''$. So in all cases there is a geodesic gallery, hence $S\cup \{s\}$ is ample. Since the complement of an ample set is also ample, the set $S^*\setminus \{s\}$ is ample. 
\end{proof}

\begin{corollary}\label{cor:isometricdismantling}
    If $S\subseteq U$ has an isometric dismantling $s_1,\ldots, s_n$, then $S$ is ample and $s_1,\ldots, s_n$ is a corner peeling of $S$.
\end{corollary}
\begin{proof}
    This follows by induction using \cref{lem:corners}. Note that the converse of the statement is not true, not every ample set has an isometric dismantling due to \cite[Theorem 4.5]{ChChMoWa}.
\end{proof}

\subsection{Amalgams and contractibility} One can easily show that  Cartesian products of ample sets are ample. 
Amalgamation is another  operation preserving ampleness. 
To make this definition precise, we adapt the definition of isometric cover from \cite{Ch_Hamming} and AMP-amalgams from \cite{ChKnPh_CUOM} (which specifies the  notion of COM-amalgams from \cite{BaChKn}). 

\begin{definition}[AMP-amalgams]\label{def:amalgam}
A triplet $(S_1,S_0,S_2)$ is called an {\it isometric cover} of a connected set $S\subseteq U$, if the following conditions are satisfied:
\begin{itemize}
\item $S_1$ and $S_2$ are two isometric subsets of $S$;
\item $S_0=S_1\cap S_2\ne \varnothing$; 
\item $S=S_1\cup S_2$ and every edge of $S$ is an edge of $S_1$, of $S_2$, or of both $S_1$ and $S_2$. 
\end{itemize}
A set $S\subseteq U$ is an {\it AMP-amalgam} of two sets $S_1,S_2\subseteq U$ if $(S_1,S_1\cap S_2,S_2)$ is an isometric cover of $S$ and $S_1,S_2,S_0=S_1\cap S_2$ are ample sets of $U$. 
\end{definition}

Notice that $S_0$ is an $(S_1\setminus S_2,S_2\setminus S_1)$-separator in $H(S)$ in the sense that any path in $H(S)$ between a vertex of $S_1\setminus S_2$ and a vertex of $S_2\setminus S_1$ necessarily traverses $S_0$. {An example is shown in \cref{fig:amalgam}.} We continue with the following result, which generalizes a similar result of \cite{ChKnPh_CUOM} in the binary case: 

\begin{proposition} \label{AMPamalgam}  Let $S\subseteq U$ be an AMP-amalgam of two ample sets $S_1$ and $S_2$ of $U$.
If $S$ is isometric, then $S$ is ample. 
\end{proposition}

\begin{proof} By Theorem \ref{thm:ample-elementary}(7), to show that $S$ is ample, it suffices to show that $S'=S\cap [u,v]$ is ample for any $u,v\in S$. 
Since $S$ is isometric, $S'$ is also isometric in $[u,v]$. Since $S_1,S_0,S_2$ are ample sets of $U$ and $[u,v]$ is a full-dimensional subproduct, by Theorem 
\ref{thm:ample=weaklyample}(3), the sets $S'_1=S_1\cap [u,v], 
S'_2=S_2\cap [u,v]$, and $S'_0=S_0\cap [u,v]$ are ample in $[u,v]$. {If $S_0'=\varnothing$, that implies that either $S_1'=\varnothing$ or $S_2'=\varnothing$, since any shortest path between a vertex of $S_1'$ and a vertex of $S_2'$ would pass through a vertex of $S_0'$. In that case, $S'=S_1'$ or $S'=S_2'$, hence $S'$ is ample. Otherwise, since} $(S_1,S_0,S_2)$ is an isometric cover of $S$, $(S'_1,S'_0,S'_2)$ is an isometric 
cover of $S'$. By \cite[Proposition 5]{ChKnPh_CUOM}, $S'$ is an ample subset of $[u,v]$ and thus of $U$. 
\end{proof}

Now, we consider a partition of an ample set $S$ with respect to a factor $U_i$ of $U$. 
\begin{definition}[Sectors, cosectors, boundaries, etc.] 
\label{def:sector} 
For  $a\in U_i$, the \emph{sector}  $S(a)$ consists of all $u\in S$ such that the $i$th coordinate $u_i$ of $u$ is equal to $a$. The \emph{boundary} $\partial S(a)$ of $S(a)$ consists of all $u\in S(a)$ adjacent to a vertex {$v$ with $v_i\neq a$.}
The \emph{cosector} of $S(a)$ is the complement $S^c(a)=S\setminus S(a)$. The  \emph{neighborhood} 
$N(S(a))$ of $S(a)$ consists of all $v\in S^c(a)$ having a neighbor in $S(a)$. The union $S^+(a)=S(a)\cup N(S(a))$ is called an \emph{extended sector}.  Finally, the union $C_i(S)=\bigcup_{a\in U_i} \partial S(a)$ of all boundaries of sectors  is called the $i$-\emph{carrier} of $S$.
\end{definition}

The following result shows that for an ample set $S$ all such sets are also ample: 

\begin{theorem} \label{sector-boundary}  If $S\subseteq U=U_1\times\ldots\times U_m$ is ample, then for any factor $U_i$, the following sets are ample: 
\begin{itemize}
 \item[(1)] the sectors $S(a)$  and  the cosectors $S^c(a)$ for all $a\in U_i$;
 \item[(2)] the boundaries $\partial S(a)$ of all sectors $S(a), a\in U_i$;
 \item[(3)] the neighborhoods $N(S(a))$ of all sectors $S(a), a\in U_i$;
 \item[(4)] the extended sectors $S^+(a)$ for all $a\in U_i$;
 \item[(5)] the carrier $C_i(S)$. 
\end{itemize}
\end{theorem}

\begin{proof} To (1): Each sector $S(a)$ is the intersection of the ample set $S$ with the full-dimensional subproduct $U_1\times\ldots U_{i-1}\times \{ a\}\times U_{i+1}\times\ldots\times U_m$, thus $S(a)$ is ample by Theorem 
\ref{thm:ample=weaklyample}(3). By the same result, the cosector $S^c(a)=S\setminus S(a)=\bigcup_{b\in U_i, b\ne a} S(b)$ is ample as the intersection of $S$ with the full-dimensional subproduct $U_1\times\ldots\times U_{i-1}\times (U_i\setminus \{ a\})\times U_{i+1}\times\ldots\times U_m$. 

\smallskip
To (2): To prove that the boundary $\partial S(a)$ is ample, consider the one-dimensional minor-subproducts $M_1,M_2$, whose partitions in the $i$th coordinate are $\{\{a\},U_i\backslash \{a\}\}$ and $\{U_i\}$, respectively. Then $\partial S(a)$ is isomorphic to $(S_{M_1})^{M_2}$, because every element of $(S_{M_1})^{M_2}$ corresponds to an element of $S(a)$ with at least one neighbor that differs in the coordinate $i$. By \cref{lem:superscriptample,lem:ample-projection}, $(S_{M_1})^{M_2}$ is ample, thus $\partial S(a)$ is ample. 

\smallskip
To (3): To prove the ampleness of the  neighborhood $N(S(a))$, by Theorem \ref{thm:ample-elementary}(5), it suffices to show that $N(S(a))$ is box-isometric. Pick any two parallel boxes $B',B''$ in $N(S(a))$. Then there exists a subproduct $V=V_{i_1}\times\ldots\times V_{i_k}$ with support $A$ and two $(X\setminus A)$-tuples $t',t''$ such that $B'=t'\times V$ and $B''=t''\times V$. We distinguish the following cases. 

\begin{case} $i\in A$, say $i_1=i$. 
\end{case}

 Since $B'\subseteq N(S(a))$, then $a\notin V_{i_1}$ and any set of vertices of $B'$ that differs only in coordinate $i$ is adjacent to one and the same vertex of $S(a)$. Then we enlarge the box $B'$ by adding the neighbors in $S(a)$ of all vertices of $B'$, which gives us the box $B'_*=t'\times V'$, where $V'=(V_{i_1}\cup \{ a\})\times\ldots\times V_{i_k}$. This box $B'_*$ contains $B'$ and is contained in $S$. Performing the same operation with $B''$, we will get the box $B''_*=t''\times V'$ containing $B''$ and contained in $S$. Since $B'_*$ and $B''_*$ are parallel boxes of $S$ and $S$ is ample, by Theorem \ref{thm:ample-elementary}(5),  $B'_*$ and $B''_*$ are connected in $S$ by a geodesic gallery $\gamma(B'_*,B''_*)=(B'_*=B_*^0,B_*^1,\ldots,B_*^{k-1},B_*^k=B''_*)$. Removing from each box $B_*^j$ the vertices having $a$ as $i$-coordinate, we will get a box $B^j$ parallel to $B',B''$ such that  $\gamma(B',B'')=(B'=B^0,B^1,\ldots,B^{k-1},B^k=B'')$ is a geodesic gallery between $B',B''$ in $N(S(a))$. This concludes the analysis of Case 1. 

\smallskip
Now, suppose that $i\notin A$. This implies that all vertices of $B'$ have the same $i$-coordinate, say $b$, and all vertices of $B''$ have the same $i$-coordinate, say $c$ (where $b$ and $c$ are both different from $a$). Let $B'_a$ and $B''_a$ be the neighbors in $S(a)$ of all vertices of $B'$ and $B''$, respectively. Since $B'$ and $B''$ are included in $N(S(a))$, $B'_a$ and $B''_a$ are boxes of $S$ parallel to $B'$ and $B''$. Furthermore, the unions $B'_*=B'\cup B'_a$ and $B''_*=B''\cup B''(a)$ are boxes of $S$. {We distinguish two more cases.}

\begin{case} $b=c$. 
\end{case}

Then $B'_*$ and $B''_*$ are parallel boxes of $S$. As in Case 1,  by Theorem \ref{thm:ample-elementary}(5)
the boxes $B'_*$ and $B''_*$ are connected in $S$ by a geodesic gallery $\gamma(B'_*,B''_*)=(B'_*=B_*^0,B_*^1,\ldots,B_*^{k-1},B_*^k=B''_*)$. Removing from each box $B_*^j$ the vertices having $a$ as $i$-coordinate, we will get a box $B^j$ parallel to $B',B''$ such that  $\gamma(B',B'')=(B'=B^0,B^1,\ldots,B^{k-1},B^k=B'')$ is a geodesic gallery between $B',B''$ in $N(S(a))$ and we are done.

\begin{case} $b\ne c$. 
\end{case}

Since ampleness is preserved by taking intersections with full-dimensional subproducts, we can suppose without loss of generality that the factor $U_i$ is equal to $\{ a,b,c\}$. Consider the elementary minor-subproduct defined by $e=\{ b,c\}$. Then the set $S_e$ is ample by Lemma \ref{lem:ample-projection}. The boxes $B'_*$ and $B''_*$ are mapped by the projection function 
to two parallel boxes 
$(B'_*)_e$ and $(B''_*)_e$ of $S_e$ of the same dimension as $B'_*$ and $B''_*$.  By Theorem \ref{thm:ample-elementary}(5) applied to $S_e$, the boxes $(B'_*)_e$ and $(B''_*)_e$ can be connected in $S_e$ by a geodesic gallery 
$\gamma_e=\gamma((B'_*)_e,(B''_*)_e)=((B'_*)_e=B_e^0,B_e^1,\ldots,B_e^{k-1},B_e^k=(B''_*)_e)$. 
Each box $B_e^j$ of $\gamma_e$ has the form $V\times \{ a, \{ b,c\}\}\times t^j$, where $V$ is the same full-dimensional subproduct as in the definition of $B'$ and $B''$, $t^j$ is a $(X\setminus (A \cup \{ i\}))$-tuple, and every two consecutive tuples $t^{j-1}$ and $t^j$ differ in a single coordinate. The preimage in $U$ of the box 
$B_e^j$ is a box $B^j_+$ of the form $V\times \{ a,b,c\}\times t^j$. The box $B^j_+$ can be viewed as the disjoint union of three boxes $B_a^j=V\times \{ a\}\times t^j$ (the $a$-box), $B_b^j=V\times \{ b\}\times t^j$ (the $b$-box), and $B_c^j=V\times \{ c\}\times t^j$ (the $c$-box). Since $B^j_e\subseteq S_e$, this means that $S\cap B^j_{+}$ shatters the elementary subproduct $e$ in $B^j_+$. Since $S$ is ample in every subproduct, it strongly shatters $e$ in $B^j_+$. But there are only two copies of $e$ possible in $B^j_{+}$, namely $B_a^j\cup B_b^j$ and $B_a^j\cup B_c^j$. Thus, for each $j$, $S$ contains the $a$-box $B_a^j$ and must contain the $b$-box $B_b^j$ or the $c$-box $B_c^j$. 

Let $\gamma_+=(B_+^0,B_+^1,\ldots,B_+^{k-1},B_+^k)$. Then $\gamma_+$ is a geodesic gallery in $U$ (not necessarily in $S$) between the parallel boxes $B_+^0$ and $B_+^k$. This geodesic gallery $\gamma_+$ splits into three geodesic galleries $\gamma_+(a)=(B_a^0,B_a^1,\ldots,B_a^{k-1},B_a^k)$ (the $a$-gallery), $\gamma_+(b)=(B_b^0,B_b^1,\ldots,B_b^{k-1},B_b^k)$ (the $b$-gallery), and $\gamma_+(c)=(B_c^0,B_c^1,\ldots,B_c^{k-1},B_c^k)$ (the $c$-gallery). From what we noted above, the $a$-gallery $\gamma_+(a)$ is included in $S$ and for each $j$ the box $B_b^j$ or $B_c^j$ is also included in $S$. Notice also that the box $B_+^0$ contains the box $B'_*$ of $S$, namely, $B'_*$ consists of the $a$-box $B_+^0$ and the $b$-box of $B_b^0$ (which is $B'$). Analogously, the box $B_+^k$ contains the box $B''_*$ of $S$, namely, $B''_*$ consists of the $a$-box $B_a^k$ and the $c$-box  $B_c^k$ (which is $B''$).  

After all this, remember that our final goal is to prove that the parallel boxes $B'$ and $B''$ of $N(S(a))$ can be connected in $N(S(a))$ by a geodesic gallery. We prove this by induction on the distance $d(B',B'')=k+1$. Note that $d(B',B'')=d(B_+^0,B_+^k)+1=k+1$ because $B'$ and $B''$ additionally differ in coordinate $i$. {The basis case $k=0$ is trivial. Suppose then that we have shown the statement for distances up to $k$.} First suppose that the $c$-gallery $\gamma_+(c)$ contains a box $B_c^j$ with $j<k$ included in $S$. Since $B_a^j\subseteq S$, $B_c^j$ belongs to $N(S(a))$. Consequently, $B'=B_b^0$ and $B_c^j$ are two parallel boxes of $N(S(a))$ with $d(B',B_c^j)=j+1<k+1=d(B',B'')$. By induction hypothesis, 
$B'$ and $B_c^j$ can be connected in $N(S(a))$ by a geodesic gallery $\gamma'$. Since $B_c^i$ and $B_c^k=B''$ are two $c$-boxes, by Case 2, $B_c^i$ and $B''$ can be connected in $N(S(a))$ by a geodesic gallery $\gamma''$. Notice that $d(B_c^i,B'')=k-j$. Therefore the union of the geodesic galleries $\gamma'$ and $\gamma''$ is a gallery of length $j+1+k-j=k+1=d(B',B'')$ connecting $B'$ and $B''$ in $N(S(a))$. Since this is a geodesic gallery, we are done. 

Now suppose that $B''$ is the unique box of the $c$-gallery $\gamma_+(c)$  included in $S$. This implies that all boxes $B_b^i$ of the 
$b$-gallery $\gamma_+(b)$, except maybe the last box $B_b^k$, are boxes of  $S$. Since all boxes of the $a$-gallery $\gamma_+(a)$ belong to $S$ and therefore to $S(a)$, all boxes $B_b^i, i=0,\ldots,k-1$ are included in $N(S(a))$. If the box $B_b^k$ belongs to $S$, the same argument implies that $B_b^k$ belongs to $N(S(a))$ and the $b$-gallery $\gamma_+(b)$ followed by the $c$-box $B_c^k=B''$ is a geodesic gallery connecting $B'$ and $B''$ in $N(S(a))$. Now assume that $B_b^k$ is not included in $S$. Consider the parallel boxes $B_b^{k-1}$ and $B_c^k=B''$ of $N(S)$. They have distance 2 and their two  common neighbors are the boxes $B_b^k$ and $B_c^{k-1}$.  Since $S$ is ample, by Theorem \ref{thm:ample-elementary}(5),  $B_b^{k-1}$ and $B_c^k$ are connected in $S$ by a geodesic gallery. Since $B_b^k$ is not included in $S$, this implies that $B_c^{k-1}$ is included in $S$, {contradicting the assumption on $\gamma_+(c)$.}
This finishes the proof that the set $N(S(a))$ is ample. 

\smallskip
To (4): By Theorem \ref{thm:ample-elementary}(5) we only need to show that $S^{+}(a)$ is box-isometric. Consider the boxes $B'=t'\times V,B''=t''\times V$, that are contained in $S^+(a)$, where $V$ has support $A$.

First suppose that $i\in A$.  If $a\in V_i$, then any box parallel to $B',B''$ is contained in $S^+(a)$, since every element of such a box either has coordinate $a$ or has a neighbor with coordinate $a$. In that case, by ampleness of $S$ there is a geodesic gallery from $B'$ to $B''$ contained in $S$, which is then automatically contained in $S^+(a)$. On the other hand, if $a\notin V_i$, that means $B',B''\subseteq N(S(a))$, and we showed in (3) already that there a geodesic gallery between $B',B''$ contained in $N(S(a))$.

If $i\notin A$, then like for assertion (3), all elements of $B'$ have the same $i$-coordinate $b$, and all elements of $B''$ have the same $i$-coordinate $c$.  

Now suppose that  $i\notin A$ and $b=c$. If $b=c=a$, then $B',B''$ are contained in $S(a)$, and they are connected by a geodesic gallery in $S(a)$ by ampleness of $S(a)$. If $b=c\neq a$, then $B',B''$ are contained in $N(S(a))$, and there is a geodesic gallery between them by ampleness of $N(S(a))$.

Finally, suppose that $i\notin A$ and $b\neq c$. Again, if $b$ and $c$ are both unequal to $a$, then $B',B''$ are contained in $N(S(a))$ and existence of a gallery follows immediately. Otherwise, assume $b=a\neq c$. By ampleness of $S$ there exists a geodesic gallery $\gamma(B',B'')=(B'=B^0,B^1,\ldots,B^{k-1},B^k=B'')$. There exists an index $\ell\leq k-1$ such that $B^0,B^1,\ldots,B^{\ell}$ all have coordinate $a$, and $B^{\ell+1},\ldots,B^{k}$ have coordinate $c$. This means that $B^{\ell+1}\subseteq N(S(a))$, since every element has a neighbor in $B^{\ell}\subseteq S(a)$. Since $B^k\subseteq N(S(a))$ and since $N(S(a))$ is ample, there is a geodesic gallery $\gamma(B^{\ell+1},B^k)=(B^{\ell+1},B^{\ell+2}_{*},B^{\ell+3}_{*}\ldots,B^{k-1}_{*},B^{k})$ contained in $N(S(a))$. It follows that $(B'=B^0\ldots,B^{\ell},B^{\ell+1},B^{\ell+2}_{*},\ldots,B^{k-1}_{*},B^k=B'')$ is a geodesic gallery contained in $S^{+}(a)$. So in all cases a geodesic gallery exists, and we conclude that $S^+(a)$ is ample.

\smallskip
To (5): To prove that the $i$-carrier $C_i(S)$ of an ample set $S$ is ample, we proceed as in the proof of (3). By Theorem \ref{thm:ample-elementary}(5), it suffices to show that $C_i(S)$ is box-isometric. Pick any two parallel boxes $B',B''$ in $C_i(S)$. Then there exists a subproduct $V=V_{i_1}\times\ldots\times V_{i_k}$ with support $A$ and two $(X\setminus A)$-tuples $t',t''$ such that $B'=t'\times V$ and $B''=t''\times V$. We can suppose that $|V_{i_1}|>1,\ldots, |V_{i_k}|>1$. Since $B',B''$ are parallel boxes of $S$ and $S$ is ample, by Theorem \ref{thm:ample-elementary}(5)
the boxes $B'$ and $B''$ are connected in $S$ by a geodesic gallery $\gamma(B',B'')=(B'=B_0,B_1,\ldots,B_{k-1},B_k=B'')$. Then there exist $(X\setminus A)$-tuples $t_0=t',t_1,\ldots t_{k-1},t_k=t''$, such that $B_j=t_j\times V$ for all $j$, and $(t_0,t_1,\ldots,t_{k-1},t_k)$ is a $(t',t'')$-geodesic in $U_{j_1}\times\ldots\times U_{j_\ell}$, where $\{ j_1,\ldots,j_{\ell}\}=X\setminus A$. 

First suppose that $i\in A$, say $i_1=i$. Since $|V_{i_1}|>1$, each vertex $u$ of each box $B_j$ of $\gamma(B',B'')$  has a neighbor $v$ in $B_j$ with its first coordinate different from $u$. This implies that all vertices of $B_j$ belong to $C_i(S)$, thus the geodesic gallery $\gamma(B',B'')$ belongs to $C_i(S)$. 

Now, suppose that $i\notin A$. This implies that all vertices of $B'$ have the same $i$-coordinate, say $a$, and all vertices of $B''$ have the same $i$-coordinate, say $b$. Since $\gamma(B',B'')$ is a geodesic gallery, either $a=b$ and all vertices of all boxes of $\gamma(B',B'')$ have the $i$-coordinate $a$ or there exists an index $\ell$, such that all vertices of $B_0,\ldots B_\ell$ have the $i$-coordinate $a$ and all vertices of 
$B_{\ell+1},\ldots,B_k$ have the $i$-coordinate $b$. In the first case, this means that $B',B''$ are contained in $\partial S(a)$. Since we already showed that $\partial S(a)$ is ample, there is a geodesic gallery $\gamma'(B',B'')$ in $\partial S(a)\subseteq C_i(S)$ connecting $B'$ and $B''$ and we are done. In the second case, we deduce that the boxes $B_\ell$ and $B_{\ell+1}$ belong to $C_i(S)$, so we just need to show that there is a geodesic gallery between $B'=B_0 $ and $B_\ell$, and between $B_{\ell+1} $ and $B_k=B''$. That follows from ampleness of $\partial S(a)$ and of $\partial S(b)$, respectively. We conclude that in all cases there is a geodesic gallery, hence the carrier $C_i(S)$ is ample.
\end{proof}

\begin{figure}[htbp]
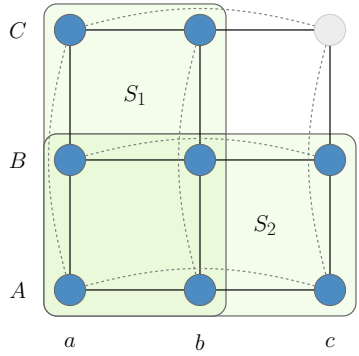

    \centering
    \includestandalone[width=0.3\linewidth]{img/amalgamation}
    \caption{The set $S\subseteq 
\{a,b,c\}\times \{A,B,C\}$ (colored vertices) is an AMP-amalgam of $S_1$ and $S_2$, where $S_1=S^{+}(C)$ and $S_2=S^{c}(C)$.}
    \label{fig:amalgam}
\end{figure}

A continuous map $F: T\times [0,1]\rightarrow T$ is a \emph{deformation retraction} of a topological space $T$ onto a subspace $A$ if, for every $x$ in $T$ and $a$ in $A$, $F(x,0)=x, F(x,1)\in A,$ and $F(a,1)=a$. A topological space $T$ is \emph{contractible} \cite{Hat} if there is a deformation retraction mapping of $T$ onto a point. The main consequence of Theorem \ref{sector-boundary} is the following result:

\begin{theorem}\label{ample-amalgams-pseudo-boxes} For each ample set 
$S\subseteq U$, the following holds:
\begin{itemize}
    \item[(1)] $S$ can be obtained from the set of its maximal boxes  by a sequence of AMP-amalgams;
    \item[(2)]  the prism complex  $||\BBox(S)||$  of $S$ is contractible. 
\end{itemize}
\end{theorem}

\begin{proof} To prove (1), we proceed by induction on the size of $S$.  If $S$ is a box, then we are done. Now, suppose that $S$ is not a box. By Lemma \ref{convex-Hamming}, $S$ is not convex in the Hamming graph $H(U)$. Since $S$ is connected, by the analogue of Tietze-Nakajima lemma for weakly modular graphs of \cite{Ch_metric} (Hamming graphs are weakly modular because the class of weakly modular graphs contains the complete graphs and is closed by taking Cartesian products), $S$ is not locally-convex. This means that $S$ contains an induced 3-path $(u,v,w)$ such that $u$ and $v$ have the second common neighbor $x$ which does not belong to $S$. Suppose without loss of generality that $u$ and $v$ differ only in the first coordinate, say $u_1$ is equal to $a$  and $v_1$ is equal to $b$ (and in all other coordinates they are equal). If $v$ and $w$ also differ  in the first coordinate, since $v\sim w$, they are equal in all other coordinates. But this implies that $u$ and $w$ also differ only in the first coordinate, yielding $u\sim w$, a contradiction. This shows that $w$ has $b$ as the first coordinate. 
Consider now the sector $S(a)$ of $S$. Then  $u\in S(a)$, $v,w\in S(b)\subseteq S^c(a)$, and $v\in N(S(a))$. We assert that $w$ does not belong to $N(S(a))$. Indeed, $w$ has a unique neighbor in $H(U)$ with coordinate $a$ and this vertex is $x$. Since $x\notin S$, necessarily $w\notin N(S(a))$. 

By Theorem \ref{sector-boundary}, the following subsets of $S$ are ample: the neighborhood $N(S(a))$ of the sector $S(a)$, the extended sector $S^+(a)=S(a)\cup N(S(a))$, and the cosector $S^c(a)$. By the definition of $N(S(a))$, for any edge $pq$ with $p\in S(a)$ and $q\in S^c(a)$, we necessarily have $q\in N(S(a))$, thus $N(S(a))$ separates the sector $S(a)$ from $S^c(a)\setminus N(S(a))$. By definition of sectors and cosectors, the sets $S^+(a)=S(a)\cup N(S(a))$ and $S^c(a)$ intersect in $N(S(a))$. Since $u\in S(a)$, {$v\in N(S(a))$}, and $w\in S^c(a)\setminus N(S(a))$, we conclude that $S$ is an AMP-amalgam of $S^+(a)$ and $S^c(a)$ along the ample set $N(S(a))$ (see also \cref{fig:amalgam}). Since $S^+(a)$ and $S^c(a)$ are ample sets containing less vertices than $S$, by induction hypothesis,  $S^+(a)$ and $S^c(a)$ are obtained from their maximal boxes by a sequence of AMP-amalgams. {Finally, we need to show that this sequence of amalgams originates from maximal boxes of $S$. We claim that all boxes of $S$ can be found in $S^+(a)$ or $S^c(a)$:} 

\begin{claim}\label{claim:box-amalgam} Each box $B$ of $\BBox(S)$ is either a box of $\BBox(S^+(a)))$ or a box of $\BBox(S^c(a))$ or a box of both 
$\BBox(S^+(a))$ and  $\BBox(S^c(a))$, in which case $B$ is a box of $\BBox(N(S(a)))$.
\end{claim}

\begin{proof} The sets $S(a)$ and $S^c(a)$ are complementary halfspaces of $S$ as the intersection of the isometric set $S$ with hafspaces $U_1\times\ldots U_{i-1}\times \{ a\}\times U_{i+1}\times\ldots\times U_m$ and $U_1\times\ldots\times U_{i-1}\times (U_i\setminus \{ a\})\times U_{i+1}\times\ldots\times U_m$ of $U$. Therefore, either $B$ is included in one of the sets $S(a)$ or $S^c(a)$, or  $S(a)\cap B$ and $S^c(a)\cap B$ define a partition of $B$ into two nonempty convex sets of $B$. Since these convex sets are full-dimensional subproducts of $B$, each vertex of $S^c(a)\cap B$ is adjacent to some vertex of $S(a)\cap B$. Consequently, $S^c(a)\cap B$ is included in the extended sector $S^+(a)$, and therefore $B\subseteq S^+(a)$. 
\end{proof} 
 By Claim \ref{claim:box-amalgam}, each maximal box of $S$ is either a maximal box of the extended sector $S^+(a)$ or of the cosector $S^c(a)$. This finishes the induction proof of the assertion (1) of the theorem.  

We prove the contractibility of the prism complex  $||\BBox(S)||$  of $S$  by induction on the size of $S$. If $S$ is a single box, then $||\BBox(S)||$ is a prism and thus is contractible. Now suppose that $S$ is not a box. By assertion (1),  $S$ is an AMP-amalgam of two smaller ample sets $S_1:=S^+(a)$ and $S_2:=S^c(a)$ along 
an ample set $S_0:=N(S(a))$. By induction assumption, the prism complexes $||\BBox(S_1)||, ||\BBox(S_2)||,$ and $||\BBox(S_0)||$ are contractible. As we noted above, each box of $S$ is either a box of $S_1$ or a box of $S_2$ or of both $S_1$ and $S_2$ (in which case, it is a box of $S_0$). By Whitehead's theorem \cite{Hat} (the gluing lemma \cite[Lemma 10.3]{Bj}),  $||\BBox(S)||$  is contractible. 
\end{proof}

\begin{remark} Let $S_1=S^+(a), S_2=S^c(a)$, and $S_0=N(S(a))$. From Claim \ref{claim:box-amalgam} it follows that for any $i$, the coordinates of the $f$-vectors of $S_1,S_2,$ and $S_0$ satisfy the equality $f_i(S)=f_i(S_1)+f_i(S_2)-f_i(S_0)$. By induction on the size of the sets, we get an alternative proof of the Euler formula $\chi(S)=1$. 
\end{remark}

Our Theorem \ref{ample-amalgams-pseudo-boxes} extends a similar result  of \cite{BaChKn}  in  the particular case of binary ample sets. By \cite{ChChMoWa}, cube complexes of binary ample sets  are collapsible.  Collapsibility is stronger than contractibility but is weaker than corner peeling. Informally, collapsibility provides a deformation  retraction of a cell complex into a point by a sequence of removals of free faces (a free face is a face of the complex belonging to a unique maximal face). Notice also that \cite{ChChMoWa} provides an example of a binary ample set without corners. 

\begin{conjecture}\label{conjecture-collapsibility} The prism complex $||\BBox(S)||$ of each ample set $S\subseteq U$ is collapsible. 
\end{conjecture}

\section{Examples of ample sets}\label{sec:examples} In this section, we present several examples of ample sets of Cartesian products. 

\subsection{Games on graphs} One of the main causes for considering ample sets in Cartesian products is because they arise in strategy sets in games on graphs. There are multiple classes of games on graphs for which one can show ampleness of the set of winning strategies in this context, including parity games and Markov decision processes (see \cite{maat_strategy_2025}), but for brevity we focus only on mean payoff games here.

A mean payoff game is a game played by two players called Maximizer and Minimizer, on a directed graph $G=(V_{\max}\cup V_{\min},E)$. Here the Maximizer owns the vertices of $V_{\max}$, and the Minimizer owns the vertices of $V_{\min}$. A pebble is placed on an initial vertex $v_0$, after which the player owning the initial node may choose an outgoing edge along which to send the pebble. After that, the owner of the next vertex the pebble is on chooses another outgoing edge, and so on, continuing indefinitely (assumed is that every vertex has at least one  outgoing edge). There is a weight function $w:E\to \mathbb{Z}$, and the outcome of the game is the long term average of the edge weights encountered in the course of the game. The Maximizer tries to maximize the long term average weight, while the Minimizer tries to minimize this.

We may assume that the Maximizer uses a positional strategy, meaning that he always picks the same outgoing edge at the same vertex (see \cite{fijalkow_games_2025} for more details). Denoting the set of successors of a vertex $v$ by $N^+(v)$, and assuming $V_{\max}=\{v_1,v_2,\ldots,v_m\}$, any positional strategy can be encoded by a vector $\sigma\in N^+(v_1)\times N^+(v_2)\times\ldots\times N^+(v_m)$, where the $i$-th entry $\sigma_i$ means that the Maximizer sends the pebble to $\sigma_i$ whenever it lands on $v_i$. Notice that the neighborhoods of vertices $v_1,v_2,\ldots,v_m$ may intersect. Let $\Sigma$ be the set of strategies $\sigma$ that guarantee a nonnegative long-term average outcome for every starting vertex. 
\begin{figure}[htbp]
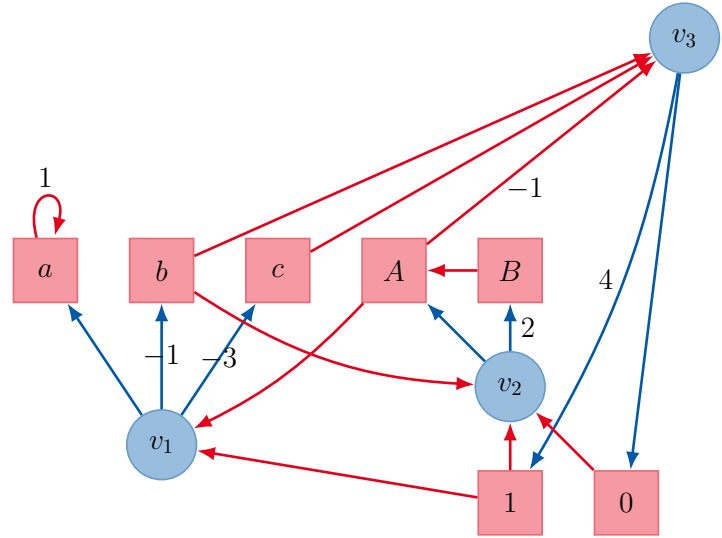

    \centering
    \includestandalone[width=0.6\linewidth]{img/MPGexample2}
    \caption{A mean payoff game, with circles representing Maximizer nodes and squares representing Minimizer nodes. All nonzero edge weights are shown on the edges. The set $\Sigma$ for this game is the ample set from \cref{fig:runningexample}.}
    \label{fig:MPGexample}
\end{figure}
For example, in the mean payoff game in \cref{fig:MPGexample}, we have $U=\{a,b,c\}\times\{A,B\}\times \{0,1\}$. If Maximizer uses the strategy encoded by $(a,A,1)$ then the outcome starting from node $v_2$ is $\frac{3}{4}>0$, since the Minimizer will follow cycle $(v_2,A,v_3,1)$ (they will never go to $v_1$, since that brings them to a cycle with mean payoff $1$). In fact, the outcome is nonnegative for every starting node. In total, there are seven positional Maximizer strategies that guarantee a nonnegative payoff from every starting node, and five strategies that do not (equivalently, they allow the Minimizer to move around a cycle of negative total weight). The set $\Sigma$ is given by the seven-element ample set from our running example (\cref{fig:runningexample}).

\begin{theorem}
    For any mean payoff game, $\Sigma$ forms an ample set. 
\end{theorem}
\begin{proof}
    It was shown in \cite{maat_strategy_2025} that, if $|N^+(v_i)|\leq 2$ for $i=1,2,\ldots, m$, then $\Sigma$ is an ample set. The condition $|N^+(v_i)|\leq 2$ implies that the set of strategies $U$ forms a hypercube. If we remove the requirement $|N^+(v_i)|\leq 2$, then $\Sigma$ is a subset of a Cartesian product. In particular, if we take any interval $[u,v]$, with $u,v\in U$, then this corresponds to the set of strategies of a subgame of the original mean payoff game, which is obtained by deleting the edges $(v_i,v')$ whenever $v'$ does not occur as an $i$-coordinate in $[u,v]$. This subgame has the property that any Maximizer vertex has one or two outgoing edges. By the before mentioned result, the set of strategies with nonnegative outcome in the subgame is ample, which is equivalent to saying that $\Sigma\cap [u,v]$ is ample. Since this holds for any interval, it follows from \cref{thm:ample-elementary} that $\Sigma$ is ample.
\end{proof}

\subsection{Prism-like polyhedra}
Now we describe a broad class of polyhedra whose vertex set can be described by an ample set in a Cartesian product.
\begin{definition}
    Let $P$ be a $d$-dimensional polyhedron (polyhedron meaning the intersection of a finite number of halfspaces in $\mathbb{R}^d$). We call $P$ a \emph{prism-like polyhedron} if there exists an order-preserving injective mapping $f$ from the face lattice of $P$ to the face lattice of a $d$-dimensional prism $P'$ (i.e. $P'$ is a Cartesian product of simplices).

   Moreover, suppose that $P'=\Delta_1\times \Delta_2\times\ldots\times \Delta_m$, where each $\Delta_i$ is a simplex. Let $U_i$ be the set of vertices of $\Delta_i$, then $U=U_1\times U_2\times\ldots\times U_m$ is the set of vertices of $P'$. Note that $f$ preserves any chains from the face lattice, hence it maps vertices to vertices. Thus, let $b(P)$ be the set of vertices $v'$ of $P'$ such that $v'=f(v)$ for some vertex $v$ of $P$ (note that $b(P)$ implicitly depends on the choice of $P'$ and $f$). 
\end{definition}
{For example, any quadrilateral is prism-like, since we can pick $f$ to be a bijection to a product of two line segments. Any simple polyhedral cone is prism-like, where $b(P)$ consists of any single element. The mean payoff game from \cref{fig:MPGexample} can be turned into a linear program (see \cite{maat_strategy_2025}) whose feasible region is a prism-like polyhedron $P$, and where $b(P)$ is the seven-element set from our running example.}

\begin{theorem}\label{thm:simplexproductample}
    If $P$ is a prism-like polyhedron, then for any choice of $b(P)$, the set $b(P)$ is an ample subset of $U$ and allows for a corner peeling.
\end{theorem}
\begin{proof}
    Analogously to the proof of \cite[Lemma 4.8]{maat_strategy_2025}, it can be shown that if $P$ is prism-like, then there exists another prism-like polyhedron $P''$ such that $b(P'')=b(P)\backslash \{c\}$  for some $c\in b(P)$. Moreover, $b(P)$ is isometric, since every interval of $U$ corresponds to a (possibly empty) face of $P$, and since the vertex-edge graph of any face of a polyhedron is connected (this can be shown similarly to Balinski's theorem). It follows from this that $b(P)$ has an isometric dismantling, thus by \cref{cor:isometricdismantling}, $b(P)$ is ample and has a corner peeling.
\end{proof}

\subsection{Quasi-median graphs} Quasi-median graphs are Hamming analogs of median graphs, which constitute the most important class of graph in metric graph theory and one of the basic examples of binary ample sets. Quasi-median graphs have been introduced in \cite{Mu} and have been investigated in many papers, in particular in \cite{BaMuWi}. 

Three vertices $v_1,v_2,v_3$ of a graph $G$ form a {\it metric triangle}
$v_1v_2v_3$ if the intervals $[v_1,v_2], [v_2,v_3],$ and
$[v_3,v_1]$ pairwise intersect only in the common end-vertices, i.e.,
$[v_i, v_j] \cap [v_i,v_k]= \{v_i\}$ for any $1 \leq i, j, k \leq 3$.
If $d(v_1,v_2)=d(v_2,v_3)=d(v_3,v_1)=k,$ then this metric triangle is
called {\it equilateral} of {\it size} $k.$ A \emph{quasi-median} of a triplet $u_1,u_2,u_3$ is a metric triangle $v_1v_2v_3$ such that $d(u_i,u_j)=d(u_i,v_i)+d(v_i,v_j)+d(v_j,u_j)$ for any $1 \leq i\ne j \leq 3$. 
A quasi-median $v_1v_2v_3$ of $u_1,u_2,u_3$ is called a \emph{median} if $v_1v_2v_3$ is a metric triangle of size 1, i.e., if $v_1=v_2=v_3.$ The \emph{kite} is the graph $K_4^-=K_{1,1,2}$ obtained by gluing two triangles along a common edge. A \emph{median graph} is a graph in which every triplet of vertices has a unique median. 
A \emph{quasi-median graph} is a graph $G$ satisfying the following three conditions:
\begin{itemize}
\item each triplet of vertices of $G$ ha a unique quasi-median and this quasi-median is an equilateral metric triangle;
\item $G$ does not contain induced kites $K_4^-$;
\item the convex hull of any isometric 6-cycle $C_6$ is a 3-cube. 
\end{itemize}
Quasi-median graphs are isometric subgraphs of Hamming graphs, furthermore they are exactly the retracts of Hamming graphs \cite{Wi-qm}. By \cite[Theorem 1]{BaMuWi}, finite quasi-median graphs which are obtained from Hamming graphs (boxes) by a sequence of gated amalgams (which are stronger that ample amalgams). By the same theorem of \cite{BaMuWi}, intervals of quasi-median graphs are median graphs, thus they are ample. Combining this result with the fact that quasi-median graphs are isometric subgraphs of Hamming graphs and our Theorem 
\ref{thm:ample-elementary}(7), we conclude that the vertex-sets of quasi-median graphs are ample:

\begin{proposition}\label{quasi-median} If $G=(V,E)$ is a quasi-median graph isometrically embedded into a Hamming graph $H(U)$, then $V$ is an ample subset of $V.$
\end{proposition}

\section{VC-dimension(s) versus VC-density} \label{sec:VC-dimension}

In this section, we consider the notion of VC-density of subsets $S$ of Cartesian products $U$, and introduce the notion of $\cM$-density based on shattering of specific subsets $\cM$ of minor-subproducts of $U$. Our goal is to show how to interpret the notions of VC-dimension existing in the literature in terms of $\cM$-density. We show that all these notions of VC-dimension are upper bounded by VC-density of $S$.  Analogously to the dimension of a subprooduct,  for a minor-subproduct $M=M(\Lambda,U)=M_1\times\ldots\times M_m$,  the \emph{dimension} $\dim(M)$ of $M$ is the dimension of $M$ as a product, i.e., $\dim(M)=\sum_{i=1}^m (|M_i|-1)$. 

\subsection{The VC-density} The VC-density of subgraphs of Cartesian products was introduced in \cite{ChLaRa} in order to generalize  the classical inequality $\frac{|E|}{|S|}\le \vcdim(S)\le \log|S|$ established in \cite{HaLiWa} for subgraphs $G=(S,E)$ of hypercubes. The \emph{density} $\dens(G)$ of a graph $G=(S,E)$ is the maximal ratio $|E(G')|/|V(G')|$ over all its subgraphs $G'$. The \emph{(minor) VC-density} $\vcdens(G)$ (or  $\vcdens(S)$) of a subgraph $G=(S,E)$  of a Cartesian product $\Gamma=G_1\times \ldots\times G_m$ of graphs $G_1,\ldots,G_m$ is the largest density $\dens(M)$ of a minor-subproduct $M$ of $\Gamma$  shattered by $G$. 

In the particular case when $\Gamma$ is the Hamming graph $H(U)$, each minor-subproduct $M$ is a Hamming graph itself and by \cite[Lemma 2]{ChLaRa}, a densest subgraph $M'$ of $M$ is again a Hamming graph. Since any Hamming graph is a regular graph, its density coincides with the degree of any of its vertices, which is nothing else than its dimension. Therefore the densest subgraph of a shattered minor-subproduct $M$ is $M$ itself, hence $\vcdens(S)$ is the largest dimension of a minor-subproduct $M$ shattered by $S$. 

The paper \cite{ChLaRa}  provides several properties of the VC-density for subgraphs of Cartesian products of general graphs, which are also interesting in the case of subgraphs of Hamming graphs. Notice also the Conjecture 1 of \cite{ChLaRa} asserting that for any subgraph $G=(S,E)$ of the Cartesian product $\Gamma=G_1\times \ldots\times G_m$ of graphs $G_1,\ldots,G_m$, 
the following inequality holds: $\frac{|E|}{|S|}\le \vcdens(S)$. This conjecture is open for subsets of Cartesian products.

\subsection{The $\cM$-density} 
Let  $\cM\subseteq \MP^{**}(U)$ (recall that   $\MP^{**}(U)$ denotes the set of all minor-subproducts $M\in \MP(V)$ over all subproducts $V$ of $U$).  The \emph{$\cM$-density} $\cM-\dim(S)$ of a set $S\subseteq U$ is the largest dimension $\dim(M)$ of a minor $M\in \cM$ shattered by $S$.  Notice that $\vcdens(S)$ coincides with $\MP^{**}(U)$-density of $S$.  Notice also that, unlike other notions of dimension (see below), 
the notion of $\cM$-density is \emph{bidimensional} because it takes into account not only the number of non-trivial factors in $M\in \cM$ but also their size.   The following monotonicity property of $\cM$-dimension follows from the definition:

\begin{lemma} \label{lem:dim-minor} If $\cM'\subseteq \cM\subseteq \MP^{**}(U)$ and $S\subseteq U$, then $\cM'-\dens(S)\le \cM-\dens(S)$. 
\end{lemma}

For subsets $S$ of binary products $U=\{ a,b\}^X$, the VC-density of $S$ coincides with the classical VC-dimension $\vcdim(S)$. Indeed, a subset $Y\subseteq X$ is shattered by $S$ if and only if $S$ shatters the following minor-subproduct $M$ of $U$: $M=M_1\times\ldots\times M_m=M(\Lambda)$, where $\Lambda=(\alpha_1,\ldots,\alpha_m)$ and each $M_i$ is defined by the partition $\alpha_i=\{ \{ a\}, \{ b\}\}$ if $i\in Y$ and $\alpha_i=\{ \{ a,b\}\}$ if $i\in X\setminus Y$.

\subsection{The Natarajan dimension} The Natarajan dimension was defined in \cite{Na} and thoroughly used in multiclass learnability. 
A \emph{cube} of dimension $d$ is a product $C=C_{i_1}\times\ldots\times C_{i_d}$, where each $C_{i_j}\subseteq U_{i_j}$ has size $2$. Equivalently, a cube is a subproduct of $U$ in which all factors have size $2$. For $S\subseteq U$, a set $Y\subseteq X=\{ 1,\ldots, m\}$ is \emph{Natarajan-shattered} by $S$ if the set $S_Y=\{ s_{|Y}: s\in S\}$ contains a cube of dimension $|Y|$. The \emph{Natarajan dimension} of $S$ is the maximum size of a set $Y\subseteq X$ that is Natarajan-shattered. 

The Natarajan dimension has the following interpretation in terms of minor-subproducts. We consider a cube 
$C=C_{i_1}\times\ldots\times C_{i_d}$ as a subproduct of $U$ and let $M_C=M(\Lambda,C)$, where $\Lambda=(\alpha_{i_1},\ldots,\alpha_{i_k})$ and $\alpha_{i_j}$ is the partition of $C_{i_j}$ into two blocks of size 1 (i.e., $M_C$ is the trivial minor-subproduct of $C$). Then $C$ is Natarajan-shattered by the set $S\subseteq U$ if and only if the minor-subproduct $M_C$ is shattered by $S$. Therefore, if $\cM_N(U)$ denotes the set of all minor-subproducts $M_C$ over all cubes  $C$, then the Natarajan dimension of $S$ coincides with $\cM_N(U)-\dens(S)$.

\subsection{The $\Psi$-dimension} This concept was introduced in \cite{BDCeHaLo} as a generalization of several other concepts of dimension.  Assumed is that $U_1=U_2=\ldots=U_m$. Let $\Psi$ be a class of functions $\overline{\psi}:U\to \{0,1,*\}^m$ of the form $\overline{\psi}(u)=\left(\psi_1(u_1),\psi_2(u_2),\ldots,\psi_m(u_m)\right)$, where $\psi_i:U_i\to\{0,1,*\}$ for $i=1,2,\ldots,m$. A set of indices $Y\subseteq X$ 
is \emph{$\Psi$-shattered} by a set $S\subseteq U$ if there exists a function $\overline{\psi}\in \Psi$ such that
\[
    \{\overline{\psi}(s)_{|Y}:s\in S\}\subseteq \{0,1\}^{|Y|}
\]
The \emph{$\Psi$-dimension} of  $S\subseteq $ is the largest cardinality $|Y|$ of a set $Y\subseteq X$ that is $\Psi$-shattered by $S$. Graph dimension \cite{Na}, Natarajan dimension \cite{Na}, and Pollard dimension \cite{Po} are obtained for particular 
choices of the class of functions $\Psi$, see \cite{BDCeHaLo} for details. 

There is a natural interpretation of $\Psi$-dimension in terms of minor-subproducts. Given  $Y=\{i_1,i_2,\ldots,i_k\}\subseteq X$ and some $\overline{\psi}=(\psi_1,\ldots,\psi_m)$, consider the subproduct $V=\psi_{i_1}^{-1}(\{0,1\})\times\ldots\times \psi_{i_k}^{-1}(\{0,1\})$. Let $M=M(\Lambda,V)$, where $\Lambda=(\alpha_{i_1},\ldots,\alpha_{i_k})$ and $\alpha_{i_j}=\{\psi_{i_j}^{-1}(0),\psi_{i_j}^{-1}(1)\}$ for $j=1,\ldots,k$. Then $S$ $\Psi$-shatters $Y$ by means of $\overline{\psi}$ if and only if the minor-subproduct $M$ is shattered by $S$. Thus the $\Psi$-dimension is the largest $d$ such that there is a subproduct $V$ and a binary minor-subproduct $M=M(\Lambda,V)$, such that $S$ shatters $M$, $M$ has $d$ partitions that are not co-trivial (i.e. have two blocks), and $M$ can be constructed from some element of $\Psi$ as described above. In particular, $\Psi$ corresponds to some set of binary minor-subproducts $\mathcal{M}_\Psi\subseteq \MP^{**}(U)$, and the $\Psi$-dimension of $S$ is equal to ${\cM}_\Psi-\dens(S)$.

\subsection{The Daniely-Shalev-Shwartz dimension} 

{Closely following the definition from the paper \cite{BrCaDiMoYe},} a subset $C\subseteq U_{i_1}\times\ldots\times U_{i_k}$ is called a \emph{pseudo-cube} if for every $u\in C$ and every coordinate $i_j$, there exists $v\in C$ that agrees with $u$ in all coordinates except $i_j$.  
For $S\subseteq U$, a set $Y\subseteq X=\{ 1,\ldots, m\}$ is \emph{DS-shattered} by $S$ if the set $S_Y=\{ s_{|Y}: s\in S\}$ contains a pseudo-cube. The \emph{DS-dimension} of $S$ is the maximum size of a set $Y\subseteq X$ that is DS-shattered by $S$ (DS-dimension was introduced in \cite{DaSS}). Each cube is a pseudo-cube, but the converse is not true, see \cite{BrCaDiMoYe}. By directly adapting the notion of $i$-carrier to arbitrary subsets of $U$, pseudo-cubes can be characterized as follows:

\begin{lemma} \label{pseudo-box} A set $S\subseteq U=U_1\times\ldots\times U_m$ is a pseudo-cube if and only if for each coordinate $i=1,\ldots,m$, $S$ coincides with its $i$-carrier $C_i(S)$. 
\end{lemma}

Now, we prove that pseudo-cubes of a product $U$ shatter binary mixed minor-subproducts from $\BcMP^{**}(U)$ {(defined as in \cref{cor:binample}).}

\begin{lemma} \label{lem:DS-dimension} If $C$ is a pseudo-cube of $U$ and $a^0=(a^0_1,\ldots,a^0_m)$ is any vertex of $C$, then 
$C$ shatters the minor-subproduct $M(a^0)\in \BcMP(U)$. 
Consequently, the DS-dimension of any set $S\subseteq U$ is smaller or equal than the $\BcMP^{**}(U)$-density of $S$.  If $S$ is ample, then the DS-dimension and the $\BcMP^{**}(U)$-density of $S$ coincide. 
\end{lemma}

\begin{proof}
Let $M(a)=M(\Lambda,U)$ and consider the box-partition  ${\mathcal B}(\Lambda,U)$ associated to the generalized partition $\Lambda$. Pick any box $B=B_1\times\ldots\times B_m$ of ${\mathcal B}(\Lambda,U)$. We show that $B\cap C\ne \varnothing$ by induction on the distance $d(\{a^0\},B)$ from $a^0$ to $B$. If $d(\{a^0\},B)=0$ then $B=\{a^0\}$, so $B\cap C=\{ a^0\}$. Now, suppose we have shown that $B\cap C\ne \varnothing$ for all $B\in {\mathcal B}(\Lambda,U)$ with $d(\{a^0\},B)\leq n$, and pick a box $B\in {\mathcal B}(\Lambda,U)$ such that $d(\{a^0\},B)=n+1\geq 1$. Let $b\in B$ be an element closest to $a^0$, and let $i\in X$ such that $b_i\neq a_i^0$. Let $B'=B_1\times \ldots \times B_{i-1}\times \{a_i^0\}\times B_{i+1}\times \ldots\times B_m$, then $B'\in {\mathcal B}(\Lambda,U)$ and $d(\{a^0\},B')=n$. By the induction hypothesis, there exists a vertex $b'\in B'\cap C$. Since $C$ is a pseudo-cube, there is a neighbor $b''$ of $b'$ contained in $C$ and differing from $b'$ only in coordinate $i$. 
Since the coordinate $i$ of $b'$ is $a^0_i$, from the definition of $B'$ it follows that $b''\in B\cap C$. In conclusion, the minor-subproduct $M(a^0)$ is shattered by $C$.

To prove the second assertion, recall that the DS-dimension of $S$ is the dimension $d$ of the largest pseudo-cube $C$ shattered by $S$. By the first assertion, if $C$ is a pseudo-cube of $V=U_{i_1}\times\ldots\times U_{i_d}$, then $C$ shatters a binary minor-subproduct $M(a)=M(\Lambda,V)$ with $a\in V$  of $\BcMP(V)$.  The dimension of this minor-subproduct $M(a)$ is also $d$. Since the set $S$ also shatters $M(a)$, we deduce that the DS-dimension of  $S$ is at most $\BcMP^{**}(U)-\dens(S)$. Additionally, if $S$ is ample, then $S$ contains a copy $R$ of the binary minor-subproduct $M(a)$. Since $M(a)$ is  a pseudo-cube, $R$ is also a pseudo-cube and we conclude that the DS-dimension of $S$ is at least $\BcMP^{**}(U)-\dens(S)$. 
\end{proof}

\begin{remark} There exist examples of sets $S\subseteq U$ with DS-dimension 1 and arbitrary large $\BcMP^{**}(U)$-density. 
\end{remark}

Summarizing, we obtain the following result:

\begin{proposition} Natarajan dimension and $\Psi$-dimension of a set $S\subseteq U$ are equal to $\cM_N(U)$-density and  $\cM_\Psi(U)$-density of $S$, respectively,  and DS-dimension of $S$ does not exceed the $\BcMP^{**}(U)$-density of $S$. Each of these dimensions does not exceed the VC-density of $S$. 
\end{proposition}

\section{Concluding remarks}

In this paper, we show how ample sets in Cartesian products  naturally arise as a generalization of ample sets in hypercubes. Furthermore, we develop a theory of ample sets  by giving several descriptions, constructions, and examples. This theory has been developed for finite sets. We believe that it is possible to extend the characterizations of ample sets to the infinite case, but we leave working out the details for future work. Let $\Pi_{i\in \Lambda} U_i$ be a Cartesian product of sets $U_i, i\in \Lambda$, where each $U_i$ is a subset of the set $\mathbb N$ of natural numbers, each $U_i$ contains $0$, and $\Lambda$ is an arbitrary set of indices. The elements of  $\Pi_{i\in \Lambda} U_i$ containing a finite number of coordinates different from $0$ are called \emph{tuples} and let $U$ denote the set of tuples of  $\Pi_{i\in \Lambda} U_i$. A  \emph{finite full-dimensional subproduct} of $U$ is any 
subroduct of the form  $V=\Pi_{i\in \Lambda} V_i$, where each $V_i$ is finite and $V_i=\{ 0\}$ for all $i\in \Lambda\setminus \Lambda'$ for a finite set $\Lambda'\subset \Lambda$. We say that a set $S\subseteq U$ is \emph{finitary ample} if $S\cap V$ for each finite full-dimensional subproduct $V$ of $S$. Then our proofs by induction can be performed by applying induction to the sets $S\cap V$ instead of $S$. 


Several important features of the theory of classical ample sets  still lack generalization in the Hamming setting:\emph{realizability, covector and cocircuit characterizations of ampleness, duality, and geometry of ample sets}. Also, it will be interesting to investigate the \emph{structure and the properties of lopsided/weakly ample sets}. 

By \cite{La}, any set $S\subseteq \{ -1,+1\}^E$ that encodes the intersection pattern of a convex set $K\subset {\mathbb R}^E$ with the orthants of ${\mathbb R}^E$ is ample. The ample sets $S\subseteq \{ -1,+1\}^E$ that can be represented in this way are called \emph{realizable}. Not every ample set is realizable \cite{La}, but it was shown in \cite{BaChDrKo_geometry} that any ample set $S\subseteq \{ -1,+1\}^E$ is realizable as the  intersection pattern of an $\ell_1$-Menger-convex set $K$ with the orthants of ${\mathbb R}^E$.  Furthermore, it was shown in  \cite{BaChDrKo_geometry} that a set $S\subseteq \{ -1,+1\}^E$ is ample iff the cube complex $||S||$ of $S$ endowed with the intrinsic $\ell_1$-metric is an isometric subspace of $({\mathbb R}^E,d_{\ell_1})$ (i.e., is  $\ell_1$-Menger-convex). The \emph{covectors} of $S$ are the barycenters of the cubes of the cube complex $||S||$ and the \emph{cocircuits} are the covectors corresponding to maximal cubes of $||S||$. At the difference of the vertices of $S$ which are $\{ -1,+1\}$-vectors, the covectors are $\{ -1,0,+1\}$-vectors. The complement $S^*$ can be viewed as the dual of $S$ (similar to duality in oriented matroids \cite{BjLVStWhZi}) and the baricenters of maximal cubes of $S^*$ are the \emph{circuits} of $S$ \cite{BaChDrKo}. 
 The subsets $J$ of $\{-1,0,+1\}^E$ that are the sets of covectors,   cocircuits, or circuits of an ample set of $\{-1,+1\}^E$ have been characterized in \cite{BaChDrKo_geometry} via a simple signed-circuit axiom (for the case of cocircuits, see also \cite{La}). 

To extend these results to ample sets in Cartesian products, we believe that it is necessary to define a geometric space hosting the Hamming graph $H(U)$ and the prism complex $||S||$ of $S\subseteq U$ (analogously to the solid cube $[-1,+1]^E$ hosting the geometric realization $||S||$ of $S\subseteq \{ -1,+1\}^E$). As such a space one can consider the prism complex $||U||$ of the product $U$ endowed with suitable metrics on factors. This will also require an appropriate encoding of covectors of $S$ (barycenters of prisms of $||S||$), i.e., an extension of the terminology used in the theory of oriented matroids \cite{BjLVStWhZi} and complexes of oriented matroids (COMs) \cite{BaChKn}. Ample sets are basic examples of COMs. Therefore one can ask if there is a \emph{suitable generalization of COMs},  capturing all ample sets of Cartesian products (for a question in the same vein, see \cite{ChGeKn}). 

Originally, classical ample sets  have been defined by the equality $|S|=|\oX(S)|$ in the sandwich lemma $|\uX(S)|\le |S|\le |\oX(S)|$. 
Since $|S|=|\oX(S)|$  implies the equality $|\uX(S)|=|S|$ and $\uX(S)\subseteq \oX(S)$, binary ampleness is equivalent to the equality $\uX(S)=\oX(S)$, which was used to define ample sets in Cartesian products.  It will be interesting and important to see if an analog of the sandwich lemma holds for all subsets of Cartesian products.

\section*{Acknowledgements} The authors would like to acknowledge Kolja Knauer for some useful discussions at the earlier stage of this work. Victor Chepoi was partially supported by the ANR project MIMETIQUE (ANR-25-CE48-4089-01). 

\bibliographystyle{abbrv}
\bibliography{references}

\end{document}

%% file: img/runningexample.tex
\definecolor{SetBlue}{RGB}{0, 90, 169} 
\begin{tikzpicture}[node distance = 3cm,
    def/.style={circle, draw=black!20, fill=black!7, thick, minimum size=12mm}, 
    every path/.style={draw=black,line width=1.3pt,{>=Latex}}, 
      S/.style={circle, draw=black!60, fill=SetBlue!70, thick, minimum size=12mm}, 
      phan/.style={draw=none, fill=none, inner sep=0pt, outer sep=0pt} 
      ] 
    {\huge
    \node[def] (a) at (-1,2) {};
    \node[def] (b) at (5,2) {};
    \node[def] (c) at (11,2) {};
    \node[S, minimum size=10.5mm] (d) at (2,5) {};
    \node[def, minimum size=10.5mm] (e) at (6,5) {};
    \node[S, minimum size=10.5mm] (f) at (10,5) {}; 
    \node[S] (g) at (-1,7) {};
    \node[S] (h) at (5,7) {};
    \node[def] (i) at (11,7) {};
    \node[S, minimum size=10.5mm] (j) at (2,9) {};
    \node[S, minimum size=10.5mm] (k) at (6,9) {};
    \node[S, minimum size=10.5mm] (l) at (10,9) {};
    \draw (a) edge[black] (b);
    \draw (b) edge[black] (c);
    \draw (d) edge[black] (e);
    \draw (e) edge[black] (f);

    \draw (g) edge[black] (h);
    \draw (h) edge[black] (i);
    \draw (j) edge[black] (k);
    \draw (k) edge[black] (l);
    
    \draw (a) edge[black] (d);
    \draw (b) edge[black] (e);
    \draw (c) edge[black] (f);
    \draw (g) edge[black] (j);
    \draw (h) edge[black] (k);
    \draw (i) edge[black] (l);

    \draw (a) edge[black] (g);
    \draw (b) edge[black] (h);
    \draw (c) edge[black] (i);
    \draw (d) edge[black] (j);
    \draw (e) edge[black] (k);
    \draw (f) edge[black] (l);

    \draw (j) edge[gray,dashed,bend left=17] (l);
    \draw (d) edge[gray,dashed,bend right=17] (f);
    \draw (g) edge[gray,dashed,bend left=12] (i);
    \draw (a) edge[gray,dashed,bend right=12] (c);
    }
    {\LARGE
    \node (coorda) at (-2.2,0.2) {$a$};
    \node (coordb) at (4.5,0.2) {$b$};
    \node (coordc) at (11.5,0.2) {$c$};

    \node (coordA) at (-3,2) {$A$};
    \node (coordB) at (-3,7) {$B$};

    \node (coordAA) at (13.2,7) {$0$};
    \node (coordBB) at (11.8,9) {$1$};
    }
\end{tikzpicture}

%% file: img/partitionlattice.tex
\definecolor{SetBlue}{RGB}{0, 90, 169} 
\definecolor{PartitionBlue}{RGB}{200,215,241}
\definecolor{BoxGreen}{RGB}{183,235,127}
\definecolor{PartYellow}{RGB}{220,220,100}
\definecolor{PartRed}{RGB}{230,110,110}
\pgfdeclarelayer{dominos}
\pgfdeclarelayer{parts}
\pgfdeclarelayer{points}
\pgfsetlayers{main,dominos,parts,points}

\begin{tikzpicture}[node distance = 3cm,
    def/.style={circle, draw=black!40, fill=black!7, thick, minimum size=12mm}, 
    every path/.style={draw=black,line width=1.3pt,{>=Latex}}, 
      S/.style={circle, draw=black!60, fill=SetBlue!70, thick, minimum size=12mm}, 
      phan/.style={draw=none, fill=none, inner sep=0pt, outer sep=0pt} 
      ] 

    \newcommand{\domino}[2]
    {
        \node (#1) at #2 {};
        \begin{scope}[shift={#2}]
            \begin{pgfonlayer}{dominos}
                \draw[fill=white, rounded corners=0.7cm] (-5.5,-3.5) rectangle (5.5,3.5);
            \end{pgfonlayer}
            \foreach \x in {-4,0,4}
            {
                \foreach \y in {-2,2}
                {
                    \begin{pgfonlayer}{points}
                        \node[def] at (\x,\y) {};
                    \end{pgfonlayer}
                }
            }
        \end{scope}
    }
    \newcommand{\bigdomino}[2]
    {
        \node (#1) at #2 {};
        \begin{scope}[shift={#2}]
            \begin{pgfonlayer}{dominos}
                \draw[fill=white, rounded corners=0.7cm] (-5.5,-3.5) rectangle (5.5,4.3);
            \end{pgfonlayer}
            \foreach \x in {-4,0,4}
            {
                \foreach \y in {-2,2}
                {
                    \begin{pgfonlayer}{points}
                        \node[def] at (\x,\y) {};
                    \end{pgfonlayer}
                }
            }
        \end{scope}
    }
    \newcommand{\partbox}[4] 
    {
        \begin{pgfonlayer}{parts}
            \begin{scope}[shift=(#1)]
                \draw[rounded corners=5mm, draw=black!50, fill=#4!50] #2 rectangle #3;
            \end{scope}
        \end{pgfonlayer}
    }
    {\huge

    \domino{triv}{(0,0)}
    \partbox{triv}{(-5,-3)}{(-3,-1)}{PartYellow}
    \partbox{triv}{(-5,1)}{(-3,3)}{PartYellow}
    \partbox{triv}{(-1,-3)}{(1,-1)}{PartYellow}
    \partbox{triv}{(-1,1)}{(1,3)}{PartYellow}
    \partbox{triv}{(3,-3)}{(5,-1)}{PartYellow}
    \partbox{triv}{(3,1)}{(5,3)}{PartYellow}

    \domino{el1}{(-22.5,15)}
    \partbox{el1}{(-5,-3)}{(1,-1)}{BoxGreen}
    \partbox{el1}{(-5,1)}{(1,3)}{BoxGreen}
    \partbox{el1}{(3,-3)}{(5,-1)}{PartYellow}
    \partbox{el1}{(3,1)}{(5,3)}{PartYellow}

    \domino{el2}{(-7.5,15)}
    \partbox{el2}{(-5,-3)}{(-3,-1)}{PartYellow}
    \partbox{el2}{(-5,1)}{(-3,3)}{PartYellow}
    \partbox{el2}{(-1,-3)}{(5,-1)}{BoxGreen}
    \partbox{el2}{(-1,1)}{(5,3)}{BoxGreen}

    \domino{el3}{(7.5,15)}
    \partbox{el3}{(-1,-3)}{(1,-1)}{PartYellow}
    \partbox{el3}{(-1,1)}{(1,3)}{PartYellow}
    \begin{pgfonlayer}{parts} 
            \begin{scope}[shift=(el3)]
                \draw[rounded corners=5mm, draw=black!50, fill=BoxGreen!50] (-5,-3) -- (-3,-3) -- (-1,-0.8) -- (1,-0.8) -- (3,-3) -- (5,-3) -- (5,-1) -- (1,-0.2) -- (-1,-0.2) -- (-5,-1) -- cycle; 
                \draw[rounded corners=5mm, draw=black!50, fill=BoxGreen!50] (-5,3) -- (-3,3) -- (-1,0.8) -- (1,0.8) -- (3,3) -- (5,3) -- (5,1) -- (1,0.2) -- (-1,0.2) -- (-5,1) -- cycle; 
            \end{scope}
        \end{pgfonlayer}

    \domino{el4}{(22.5,15)}
    \partbox{el4}{(-5,-3)}{(-3,3)}{BoxGreen}
    \partbox{el4}{(-1,-3)}{(1,3)}{BoxGreen}
    \partbox{el4}{(3,-3)}{(5,3)}{BoxGreen}

    \domino{coel1}{(-22.5,30)}
    \partbox{coel1}{(-5,-3)}{(5,-1)}{PartitionBlue}
    \partbox{coel1}{(-5,1)}{(5,3)}{PartitionBlue}

    \domino{coel2}{(-7.5,30)}
    \partbox{coel2}{(-5,-3)}{(1,3)}{PartitionBlue}
    \partbox{coel2}{(3,-3)}{(5,3)}{BoxGreen}

    \domino{coel3}{(7.5,30)}
    \partbox{coel3}{(-5,-3)}{(-3,3)}{BoxGreen}
     \partbox{coel3}{(-1,-3)}{(5,3)}{PartitionBlue}

    \bigdomino{coel4}{(22.5,30)}
    \partbox{coel4}{(-1,-3)}{(1,3)}{BoxGreen}
    \begin{pgfonlayer}{parts} 
        \begin{scope}[shift=(coel4)]
            \draw[rounded corners=5mm, draw=black!50, fill=PartitionBlue!50] (-5,-3) -- (-3,-3) -- (-3,3.2) -- (3,3.2) -- (3,-3) -- (5,-3) -- (5,3) -- (4,3.8) -- (-4,3.8) -- (-5,3) -- cycle;   
        \end{scope}
    \end{pgfonlayer}

    \domino{cotriv}{(0,45)}
    \partbox{cotriv}{(-5,-3)}{(5,3)}{PartRed}

    \foreach\i in {el1,el2,el3,el4}
    {
        \draw (triv) edge (\i);
    }
    \foreach\i in {el1,el2,el3}
    {
        \draw (\i) edge (coel1);
    }
    \draw (el1) edge (coel2);
    \draw (el2) edge (coel3);
    \draw (el3) edge (coel4);
    \foreach\i in {coel2,coel3,coel4}
    {
        \draw (el4) edge (\i);
    }
    \foreach\i in {coel1,coel2,coel3,coel4}
    {
        \draw (\i) edge (cotriv);
    }


    }
\end{tikzpicture}

%% file: img/boxpartition.tex
\definecolor{SetBlue}{RGB}{0, 90, 169} 
\definecolor{BoxGreen}{RGB}{183,235,127}

\begin{tikzpicture}[node distance = 3cm,
    def/.style={circle, draw=black!20, fill=black!7, thick, minimum size=12mm}, 
    every path/.style={draw=black,line width=1.3pt,{>=Latex}}, 
      S/.style={circle, draw=black!60, fill=SetBlue!70, thick, minimum size=12mm}, 
      phan/.style={draw=none, fill=none, inner sep=0pt, outer sep=0pt} 
      ] 
    {\huge
    \draw[rounded corners=5mm, draw=black!50, fill=BoxGreen!30] (-2,1) rectangle (6,3);
    \draw[rounded corners=5mm, draw=black!50, fill=SetBlue!5] (10,1) rectangle (12,3);
    \draw[rounded corners=5mm, draw=black!50, fill=BoxGreen!30] (1.1,4.1) rectangle (6.9,5.9);
    \draw[rounded corners=5mm, draw=black!50, fill=SetBlue!5] (9.1,4.1) rectangle (10.9,5.9);
    \draw[rounded corners=5mm, draw=black!50, fill=BoxGreen!30] (-2,6) rectangle (5.9,8);
    \draw[rounded corners=5mm, draw=black!50, fill=SetBlue!5] (10.1,6) rectangle (12,8);
    \draw[rounded corners=5mm, draw=black!50, fill=BoxGreen!30] (1.1,8.1) rectangle (6.9,9.9);
    \draw[rounded corners=5mm, draw=black!50, fill=SetBlue!5] (9.1,8.1) rectangle (10.9,9.9);

    \node[def] (a) at (-1,2) {};
    \node[def] (b) at (5,2) {};
    \node[def] (c) at (11,2) {};
    \node[def, minimum size=10.5mm] (d) at (2,5) {};
    \node[def, minimum size=10.5mm] (e) at (6,5) {};
    \node[def, minimum size=10.5mm] (f) at (10,5) {}; 
    \node[def] (g) at (-1,7) {};
    \node[def] (h) at (5,7) {};
    \node[def] (i) at (11,7) {};
    \node[def, minimum size=10.5mm] (j) at (2,9) {};
    \node[def, minimum size=10.5mm] (k) at (6,9) {};
    \node[def, minimum size=10.5mm] (l) at (10,9) {};
    \draw (a) edge[black] (b);
    \draw (b) edge[black] (c);
    \draw (d) edge[black] (e);
    \draw (e) edge[black] (f);

    \draw (g) edge[black] (h);
    \draw (h) edge[black] (i);
    \draw (j) edge[black] (k);
    \draw (k) edge[black] (l);
    
    \draw (a) edge[black] (d);
    \draw (b) edge[black] (e);
    \draw (c) edge[black] (f);
    \draw (g) edge[black] (j);
    \draw (h) edge[black] (k);
    \draw (i) edge[black] (l);

    \draw (a) edge[black] (g);
    \draw (b) edge[black] (h);
    \draw (c) edge[black] (i);
    \draw (d) edge[black] (j);
    \draw (e) edge[black] (k);
    \draw (f) edge[black] (l);

    \draw (j) edge[gray,dashed,bend left=17] (l);
    \draw (d) edge[gray,dashed,bend right=17] (f);
    \draw (g) edge[gray,dashed,bend left=12] (i);
    \draw (a) edge[gray,dashed,bend right=12] (c);

    }
    {\LARGE
    \node (coorda) at (-2.2,0.2) {$a$};
    \node (coordb) at (4.5,0.2) {$b$};
    \node (coordc) at (11.5,0.2) {$c$};

    \node (coordA) at (-3,2) {$A$};
    \node (coordB) at (-3,7) {$B$};

    \node (coordAA) at (13.2,7) {$0$};
    \node (coordBB) at (11.8,9) {$1$};
    }
\end{tikzpicture}

%% file: img/minorsubproduct.tex
\definecolor{SetBlue}{RGB}{0, 90, 169} 
\definecolor{BoxGreen}{RGB}{183,235,127}

\begin{tikzpicture}[node distance = 3cm,
    def/.style={circle, draw=black!20, fill=black!7, thick, minimum size=12mm}, 
    every path/.style={draw=black,line width=1.3pt,{>=Latex}}, 
      S/.style={circle, draw=black!60, fill=SetBlue!70, thick, minimum size=12mm}, 
      phan/.style={draw=none, fill=none, inner sep=0pt, outer sep=0pt} 
      ] 
    {\huge
    
    \node[def, fill=BoxGreen!30] (b) at (5,2) {};
    \node[def, fill=SetBlue!5] (c) at (11,2) {};
    \node[def, minimum size=10.5mm,, fill=BoxGreen!30] (e) at (6,5) {};
    \node[def, minimum size=10.5mm, fill=SetBlue!5] (f) at (10,5) {}; 
    \node[def, fill=BoxGreen!30] (h) at (5,7) {};
    \node[def, fill=SetBlue!5] (i) at (11,7) {};
    \node[def, minimum size=10.5mm, fill=BoxGreen!30] (k) at (6,9) {};
    \node[def, minimum size=10.5mm, fill=SetBlue!5] (l) at (10,9) {};
    \draw (b) edge[black] (c);
    \draw (e) edge[black] (f);

    \draw (h) edge[black] (i);
    \draw (k) edge[black] (l);
    
    \draw (b) edge[black] (e);
    \draw (c) edge[black] (f);
    \draw (h) edge[black] (k);
    \draw (i) edge[black] (l);

    \draw (b) edge[black] (h);
    \draw (c) edge[black] (i);
    \draw (e) edge[black] (k);
    \draw (f) edge[black] (l);

    }
    {\LARGE
    \node (coordb) at (4.5,0.2) {$\{a,b\}$};
    \node (coordc) at (11.5,0.2) {$c$};

    \node (coordA) at (3,2) {$A$};
    \node (coordB) at (3,7) {$B$};

    \node (coordAA) at (13.2,7) {$0$};
    \node (coordBB) at (11.8,9) {$1$};
    }
\end{tikzpicture}

%% file: img/shatterstrongshatter.tex
\definecolor{SetBlue}{RGB}{0, 90, 169} 
\definecolor{BoxGreen}{RGB}{183,235,127}

\begin{tikzpicture}[node distance = 3cm,
    def/.style={circle, draw=black!20, fill=black!7, thick, minimum size=12mm}, 
    every path/.style={draw=black,line width=1.3pt,{>=Latex}}, 
      S/.style={circle, draw=black!60, fill=SetBlue!70, thick, minimum size=12mm}, 
      phan/.style={draw=none, fill=none, inner sep=0pt, outer sep=0pt} 
      ] 
    {\Large
    \draw[rounded corners=5mm, draw=black!50, fill=BoxGreen!30] (-2,1) -- (6,1) -- (7,4) -- (7,5.9) -- (1.1,5.9) -- (-2,3) -- cycle;
    \draw[rounded corners=5mm, draw=black!50, fill=BoxGreen!30] (-2,6) -- (6,6) -- (7,8) -- (7,9.9) -- (1.1,9.9) -- (-2,8) -- cycle;
    \draw[rounded corners=5mm, draw=black!50, fill=BoxGreen!30] (10,1) -- (12,1) -- (12,3) -- (10.9,5.9) -- (9,5.9) -- (9,4) -- cycle;
    \draw[rounded corners=5mm, draw=black!50, fill=BoxGreen!30] (10,6) -- (12,6) -- (12,8) -- (10.9,9.9) -- (9,9.9) -- (9,7.9) -- cycle;  
    
    \node[def] (a) at (-1,2) {};
    \node[def] (b) at (5,2) {};
    \node[def] (c) at (11,2) {};
    \node[S, minimum size=10.5mm] (d) at (2,5) {$*$};
    \node[def, minimum size=10.5mm] (e) at (6,5) {};
    \node[S, minimum size=10.5mm] (f) at (10,5) {$*$}; 
    \node[S] (g) at (-1,7) {};
    \node[S] (h) at (5,7) {};
    \node[def] (i) at (11,7) {};
    \node[S, minimum size=10.5mm] (j) at (2,9) {$*$};
    \node[S, minimum size=10.5mm] (k) at (6,9) {};
    \node[S, minimum size=10.5mm] (l) at (10,9) {$*$};
    \draw (a) edge[black] (b);
    \draw (b) edge[black] (c);
    \draw (d) edge[black] (e);
    \draw (e) edge[black] (f);

    \draw (g) edge[black] (h);
    \draw (h) edge[black] (i);
    \draw (j) edge[black] (k);
    \draw (k) edge[black] (l);
    
    \draw (a) edge[black] (d);
    \draw (b) edge[black] (e);
    \draw (c) edge[black] (f);
    \draw (g) edge[black] (j);
    \draw (h) edge[black] (k);
    \draw (i) edge[black] (l);

    \draw (a) edge[black] (g);
    \draw (b) edge[black] (h);
    \draw (c) edge[black] (i);
    \draw (d) edge[black] (j);
    \draw (e) edge[black] (k);
    \draw (f) edge[black] (l);

    \draw (j) edge[gray,dashed,bend left=17] (l);
    \draw (d) edge[gray,dashed,bend right=17] (f);
    \draw (g) edge[gray,dashed,bend left=12] (i);
    \draw (a) edge[gray,dashed,bend right=12] (c);
    }
    {\LARGE
    \node (coorda) at (-2.2,0.2) {$a$};
    \node (coordb) at (4.5,0.2) {$b$};
    \node (coordc) at (11.5,0.2) {$c$};

    \node (coordA) at (-3,2) {$A$};
    \node (coordB) at (-3,7) {$B$};

    \node (coordAA) at (13.2,7) {$0$};
    \node (coordBB) at (11.8,9) {$1$};
    }
\end{tikzpicture}

%% file: img/amplenotisometric2.tex
\definecolor{SetBlue}{RGB}{0, 90, 169} 
\definecolor{BoxGreen}{RGB}{183,235,127}

\begin{tikzpicture}[node distance = 3cm,
    def/.style={circle, draw=black!20, fill=black!7, thick, minimum size=12mm}, 
    every path/.style={draw=black,line width=1.3pt,{>=Latex}}, 
      S/.style={circle, draw=black!60, fill=SetBlue!70, thick, minimum size=12mm}, 
      phan/.style={draw=none, fill=none, inner sep=0pt, outer sep=0pt} 
      ] 
    {\huge
    \draw[rounded corners=5mm, draw=black!50, fill=BoxGreen!30] (-1,-1) rectangle (1,6);
    \draw[rounded corners=5mm, draw=black!50, fill=BoxGreen!30] (4,-1) rectangle (6,6);
    \draw[rounded corners=5mm, draw=black!50, fill=BoxGreen!30] (9,-1) rectangle (11,6);
    \draw[rounded corners=5mm, draw=black!50, fill=BoxGreen!30] (-1,9) rectangle (1,11);
    \draw[rounded corners=5mm, draw=black!50, fill=BoxGreen!30] (4,9) rectangle (6,11);
    \draw[rounded corners=5mm, draw=black!50, fill=BoxGreen!30] (9,9) rectangle (11,11);
    
    \node[S] (a) at (0,0) {};
    \node[S] (b) at (5,0) {};
    \node[def] (c) at (10,0) {};
    \node[S] (d) at (0,5) {};
    \node[def] (e) at (5,5) {};
    \node[S] (f) at (10,5) {};
    \node[S] (g) at (0,10) {};
    \node[S] (h) at (5,10) {};
    \node[S] (i) at (10,10) {};
        
    \draw (a) edge[black] (b);
    \draw (b) edge[black] (c);
    \draw (d) edge[black] (e);
    \draw (e) edge[black] (f);
    \draw (g) edge[black] (h);
    \draw (h) edge[black] (i);
        
    \draw (a) edge[black] (d);
    \draw (b) edge[black] (e);
    \draw (c) edge[black] (f);
    \draw (g) edge[black] (d);
    \draw (h) edge[black] (e);
    \draw (i) edge[black] (f);

    \draw (a) edge[gray,dashed,bend left=15] (c);
    \draw (d) edge[gray,dashed,bend left=15] (f);
    \draw (g) edge[gray,dashed,bend left=15] (i);
    \draw (a) edge[gray,dashed,bend left=15] (g);
    \draw (b) edge[gray,dashed,bend left=15] (h);
    \draw (c) edge[gray,dashed,bend left=15] (i);
    
    }
    {\LARGE
    \node (coorda) at (0,-2) {$a$};
    \node (coordb) at (5,-2) {$b$};
    \node (coordc) at (10,-2) {$c$};

    \node (coordA) at (-2,0) {$A$};
    \node (coordB) at (-2,5) {$B$};
    \node (coordC) at (-2,10) {$C$};
    }
\end{tikzpicture}

%% file: img/projections1.tex
\definecolor{SetBlue}{RGB}{0, 90, 169} 
\definecolor{BoxGreen}{RGB}{183,235,127}

\begin{tikzpicture}[node distance = 3cm,
    def/.style={circle, draw=black!20, fill=black!7, thick, minimum size=12mm}, 
    every path/.style={draw=black,line width=1.3pt,{>=Latex}}, 
      S/.style={circle, draw=black!60, fill=SetBlue!70, thick, minimum size=12mm}, 
      phan/.style={draw=none, fill=none, inner sep=0pt, outer sep=0pt} 
      ] 
    {\Large
    \draw[rounded corners=5mm, draw=black!50, fill=BoxGreen!30] (-2,1) -- (6,1) -- (7,4) -- (7,5.9) -- (1.1,5.9) -- (-2,3) -- cycle;
    \draw[rounded corners=5mm, draw=black!50, fill=BoxGreen!30] (-2,6) -- (6,6) -- (7,8) -- (7,9.9) -- (1.1,9.9) -- (-2,8) -- cycle;
    \draw[rounded corners=5mm, draw=black!50, fill=BoxGreen!30] (10,1) -- (12,1) -- (12,3) -- (10.9,5.9) -- (9,5.9) -- (9,4) -- cycle;
    \draw[rounded corners=5mm, draw=black!50, fill=BoxGreen!30] (10,6) -- (12,6) -- (12,8) -- (10.9,9.9) -- (9,9.9) -- (9,7.9) -- cycle;  
    
    \node[def] (a) at (-1,2) {};
    \node[def] (b) at (5,2) {};
    \node[def] (c) at (11,2) {};
    \node[S, minimum size=10.5mm] (d) at (2,5) {};
    \node[def, minimum size=10.5mm] (e) at (6,5) {};
    \node[S, minimum size=10.5mm] (f) at (10,5) {}; 
    \node[S] (g) at (-1,7) {};
    \node[S] (h) at (5,7) {};
    \node[def] (i) at (11,7) {};
    \node[S, minimum size=10.5mm] (j) at (2,9) {};
    \node[S, minimum size=10.5mm] (k) at (6,9) {};
    \node[S, minimum size=10.5mm] (l) at (10,9) {};
    \draw (a) edge[black] (b);
    \draw (b) edge[black] (c);
    \draw (d) edge[black] (e);
    \draw (e) edge[black] (f);

    \draw (g) edge[black] (h);
    \draw (h) edge[black] (i);
    \draw (j) edge[black] (k);
    \draw (k) edge[black] (l);
    
    \draw (a) edge[black] (d);
    \draw (b) edge[black] (e);
    \draw (c) edge[black] (f);
    \draw (g) edge[black] (j);
    \draw (h) edge[black] (k);
    \draw (i) edge[black] (l);

    \draw (a) edge[black] (g);
    \draw (b) edge[black] (h);
    \draw (c) edge[black] (i);
    \draw (d) edge[black] (j);
    \draw (e) edge[black] (k);
    \draw (f) edge[black] (l);

     \draw (j) edge[gray,dashed,bend left=17] (l);
    \draw (d) edge[gray,dashed,bend right=17] (f);
    \draw (g) edge[gray,dashed,bend left=12] (i);
    \draw (a) edge[gray,dashed,bend right=12] (c);
    }
    {\LARGE
    \node (coorda) at (-2.2,0.2) {$a$};
    \node (coordb) at (4.5,0.2) {$b$};
    \node (coordc) at (11.5,0.2) {$c$};

    \node (coordA) at (-3,2) {$A$};
    \node (coordB) at (-3,7) {$B$};

    \node (coordAA) at (13.2,7) {$0$};
    \node (coordBB) at (11.8,9) {$1$};
    }
\end{tikzpicture}

%% file: img/projections2.tex
\definecolor{SetBlue}{RGB}{0, 90, 169} 
\definecolor{BoxGreen}{RGB}{183,235,127}

\begin{tikzpicture}[node distance = 3cm,
    def/.style={circle, draw=black!20, fill=black!7, thick, minimum size=12mm}, 
    every path/.style={draw=black,line width=1.3pt,{>=Latex}}, 
      S/.style={circle, draw=black!60, fill=SetBlue!70, thick, minimum size=12mm}, 
      phan/.style={draw=none, fill=none, inner sep=0pt, outer sep=0pt} 
      ] 
    {\Huge
    \node[def] (a) at (-1,2) {};
    \node[def] (b) at (4,2) {};
    \node[S] (g) at (-1,7) {};
    \node[def] (h) at (4,7) {};

    \draw (a) edge[black] (b);
    \draw (g) edge[black] (h);
    \draw (a) edge[black] (g);
    \draw (b) edge[black] (h);

    \node[S] (a2) at (-1,-7) {};
    \node[S] (b2) at (4,-7) {};
    \node[S] (g2) at (-1,-2) {};
    \node[S] (h2) at (4,-2) {};

    \draw (a2) edge[black] (b2);
    \draw (g2) edge[black] (h2);
    \draw (a2) edge[black] (g2);
    \draw (b2) edge[black] (h2);

    \node (upper) at (1.5,4.5) {$S^M$};
    \node (lower) at (1.5,-4.5) {$S_M$};
    }
    {\LARGE
    \node (coorda) at (-1,0.5) {$\{a,b\}$};
    \node (coordb) at (4,0.5) {$c$};
    \node (coordA) at (-2.5,2) {$A$};
    \node (coordB) at (-2.5,7) {$B$};
    
    \node (coorda2) at (-1,-8.5) {$\{a,b\}$};
    \node (coordb2) at (4,-8.5) {$c$};
    \node (coordA2) at (-2.5,-7) {$A$};
    \node (coordB2) at (-2.5,-2) {$B$};
    }
\end{tikzpicture}

%% file: img/thm3proof1.tex
\definecolor{SetBlue}{RGB}{0, 90, 169} 
\definecolor{BoxGreen}{RGB}{183,235,127}
\definecolor{DarkGreen}{RGB}{50,140,50}

\begin{tikzpicture}[node distance = 3cm,
    def/.style={circle, draw=black!20, fill=black!7, thick, minimum size=12mm}, 
    every path/.style={draw=black,line width=1.3pt,{>=Latex}}, 
      S/.style={circle, draw=black!60, fill=SetBlue!70, thick, minimum size=12mm}, 
      phan/.style={draw=none, fill=none, inner sep=0pt, outer sep=0pt} 
      ] 
    {\Large
    \draw[rounded corners=5mm, draw=black!50, fill=BoxGreen!30] (-2,1) rectangle (0,8);
    \draw[rounded corners=5mm, draw=black!50, fill=BoxGreen!30] (1.1,4.1) rectangle (2.9,9.9);
    \draw[rounded corners=5mm, draw=black!50, fill=BoxGreen!30] (5.1,4.1) rectangle (6.9,9.9);
    \draw[rounded corners=5mm, draw=black!50, fill=BoxGreen!30] (3.1,1) rectangle (4.9,8);
    
    \node[S] (a) at (-1,2) {$s_1$};
    \node[def] (b) at (4,2) {$q_1$};
    \node[def, minimum size=10.5mm] (d) at (2,5) {$p_1$};
    \node[S, minimum size=10.5mm] (e) at (6,5) {$t_1$};
    \node[S] (g) at (-1,7) {$s_2$};
    \node[def] (h) at (4,7) {$q_2$};
    \node[def, minimum size=10.5mm] (j) at (2,9) {$p_2$};
    \node[S, minimum size=10.5mm] (k) at (6,9) {$t_2$};
    \draw (a) edge[black] (b);
    \draw (d) edge[black] (e);

    \draw (g) edge[black] (h);
    \draw (j) edge[black] (k);
    
    \draw (a) edge[black] (d);
    \draw (b) edge[black] (e);
    \draw (g) edge[black] (j);
    \draw (h) edge[black] (k);

    \draw (a) edge[black] (g);
    \draw (b) edge[black] (h);
    \draw (d) edge[black] (j);
    \draw (e) edge[black] (k);
    }
    {\LARGE
    \node (coorda) at (-1,0.5) {$a$};
    \node (coordb) at (4,0.5) {$b$};

    \node (coordA) at (-3,2) {$A$};
    \node (coordB) at (-3,7) {$B$};

    \node (coordAA) at (5.4,2) {$0$};
    \node (coordBB) at (7.2,5) {$1$};

    {\color{DarkGreen}
    \node (coords) at (-1,9) {$\begin{array}{c} s\\\downarrow\end{array}$};
    \node (coordt) at (6,11) {$\begin{array}{c} t\\\downarrow\end{array}$};
    }
    }
\end{tikzpicture}

%% file: img/thm3proof2.tex
\definecolor{SetBlue}{RGB}{0, 90, 169} 
\definecolor{BoxGreen}{RGB}{183,235,127}
\definecolor{DarkGreen}{RGB}{50,140,50}

\begin{tikzpicture}[node distance = 3cm,
    def/.style={circle, draw=black!20, fill=black!7, thick, minimum size=12mm}, 
    every path/.style={draw=black,line width=1.3pt,{>=Latex}}, 
      S/.style={circle, draw=black!60, fill=SetBlue!70, thick, minimum size=12mm}, 
      phan/.style={draw=none, fill=none, inner sep=0pt, outer sep=0pt} 
      ] 
    {\Large
    \draw[rounded corners=5mm, draw=black!50, fill=BoxGreen!30] (-2,1) rectangle (0,3);
    \draw[rounded corners=5mm, draw=black!50, fill=BoxGreen!30] (-2,6) rectangle (0,8);
    \draw[rounded corners=5mm, draw=black!50, fill=BoxGreen!30] (1.1,4.1) rectangle (2.9,5.9);
    \draw[rounded corners=5mm, draw=black!50, fill=BoxGreen!30] (1.1,8.1) rectangle (2.9,9.9);

    \draw[rounded corners=5mm, draw=black!50, fill=BoxGreen!30] (3,1) rectangle (12,3);
    \draw[rounded corners=5mm, draw=black!50, fill=BoxGreen!30] (3,6) rectangle (12,8);
    \draw[rounded corners=5mm, draw=black!50, fill=BoxGreen!30] (5.1,4.1) rectangle (10.9,5.9);
    \draw[rounded corners=5mm, draw=black!50, fill=BoxGreen!30] (5.1,8.1) rectangle (10.9,9.9);
    
    \node[S] (a) at (-1,2) {$s_1$};
    \node[def] (b) at (4,2) {$q_1$};
    \node[def] (c) at (11,2) {$q_1'$};
    \node[S, minimum size=10.5mm] (d) at (2,5) {$p_1$};
    \node[S, minimum size=10.5mm] (e) at (6,5) {$t_1$};
    \node[def, minimum size=10.5mm] (f) at (10,5) {$p_1'$}; 
    \node[S] (g) at (-1,7) {$s_2$};
    \node[S] (h) at (4,7) {$q_2$};
    \node[def] (i) at (11,7) {$q_2'$};
    \node[def, minimum size=10.5mm] (j) at (2,9) {$p_2$};
    \node[S, minimum size=10.5mm] (k) at (6,9) {$t_2$};
    \node[def, minimum size=10.5mm] (l) at (10,9) {$p_2'$};
    \draw (a) edge[black] (b);
    \draw (b) edge[black] (c);
    \draw (d) edge[black] (e);
    \draw (e) edge[black] (f);

    \draw (g) edge[black] (h);
    \draw (h) edge[black] (i);
    \draw (j) edge[black] (k);
    \draw (k) edge[black] (l);
    
    \draw (a) edge[black] (d);
    \draw (b) edge[black] (e);
    \draw (c) edge[black] (f);
    \draw (g) edge[black] (j);
    \draw (h) edge[black] (k);
    \draw (i) edge[black] (l);

    \draw (a) edge[black] (g);
    \draw (b) edge[black] (h);
    \draw (c) edge[black] (i);
    \draw (d) edge[black] (j);
    \draw (e) edge[black] (k);
    \draw (f) edge[black] (l);

    \draw (j) edge[gray,dashed,bend left=17] (l);
    \draw (d) edge[gray,dashed,bend right=17] (f);
    \draw (g) edge[gray,dashed,bend left=12] (i);
    \draw (a) edge[gray,dashed,bend right=12] (c);
    }
    {\LARGE
    \node (coorda) at (-1,0.5) {$a$};
    \node (coordb) at (4,0.5) {$b$};
    \node (coordc) at (11.5,0.5) {$c$};

    \node (coordA) at (-3,2) {$A$};
    \node (coordB) at (-3,7) {$B$};

    \node (coordAA) at (13.2,7) {$0$};
    \node (coordBB) at (11.8,9) {$1$};

    }
\end{tikzpicture}

%% file: img/thm3proof3.tex
\definecolor{SetBlue}{RGB}{0, 90, 169} 
\definecolor{BoxGreen}{RGB}{183,235,127}
\definecolor{DarkGreen}{RGB}{50,140,50}

\begin{tikzpicture}[node distance = 3cm,
    def/.style={circle, draw=black!20, fill=black!7, thick, minimum size=12mm}, 
    every path/.style={draw=black,line width=1.3pt,{>=Latex}}, 
      S/.style={circle, draw=black!60, fill=SetBlue!70, thick, minimum size=12mm}, 
      phan/.style={draw=none, fill=none, inner sep=0pt, outer sep=0pt} 
      ] 
    {\Large
    \draw[rounded corners=5mm, draw=black!50, fill=BoxGreen!30] (-6,1) rectangle (1,3);
    \draw[rounded corners=5mm, draw=black!50, fill=BoxGreen!30] (-6,6) rectangle (1,8);
    \draw[rounded corners=5mm, draw=black!50, fill=BoxGreen!30] (-2.9,4.1) rectangle (2.9,5.9);
    \draw[rounded corners=5mm, draw=black!50, fill=BoxGreen!30] (-2.9,8.1) rectangle (2.9,9.9);

    \draw[rounded corners=5mm, draw=black!50, fill=BoxGreen!30] (4,1) rectangle (12,3);
    \draw[rounded corners=5mm, draw=black!50, fill=BoxGreen!30] (4,6) rectangle (12,8);
    \draw[rounded corners=5mm, draw=black!50, fill=BoxGreen!30] (5.1,4.1) rectangle (10.9,5.9);
    \draw[rounded corners=5mm, draw=black!50, fill=BoxGreen!30] (5.1,8.1) rectangle (10.9,9.9);
    
    \node[def] (aa) at (-5,2) {$s_1^+$};
    \node[S] (a) at (0,2) {$s_1$};
    \node[def] (b) at (5,2) {$q_1$};
    \node[def] (c) at (11,2) {$q_1^+$};
    \node[def] (dd) at (-2,5) {$p_1^+$};
    \node[S, minimum size=10.5mm] (d) at (2,5) {$p_1$};
    \node[S, minimum size=10.5mm] (e) at (6,5) {$t_1$};
    \node[def, minimum size=10.5mm] (f) at (10,5) {$t_1^+$}; 
    \node[def] (gg) at (-5,7) {$s_2^+$};
    \node[S] (g) at (0,7) {$s_2$};
    \node[S] (h) at (5,7) {$q_2$};
    \node[def] (i) at (11,7) {$q_2^+$};
    \node[def] (jj) at (-2,9) {$p_2^+$};
    \node[def, minimum size=10.5mm] (j) at (2,9) {$p_2$};
    \node[S, minimum size=10.5mm] (k) at (6,9) {$t_2$};
    \node[def, minimum size=10.5mm] (l) at (10,9) {$t_2^+$};
    \draw (aa) edge[gray, dashed] (a);
    \draw (a) edge[black] (b);
    \draw (b) edge[gray, dashed] (c);
    \draw (dd) edge[gray, dashed] (d);
    \draw (d) edge[black] (e);
    \draw (e) edge[gray, dashed] (f);

    \draw (gg) edge[gray, dashed] (g);
    \draw (g) edge[black] (h);
    \draw (h) edge[gray, dashed] (i);
    \draw (jj) edge[gray, dashed] (j);
    \draw (j) edge[black] (k);
    \draw (k) edge[gray, dashed] (l);

    \draw (aa) edge[black] (dd);
    \draw (a) edge[black] (d);
    \draw (b) edge[black] (e);
    \draw (c) edge[black] (f);
    \draw (gg) edge[black] (jj);
    \draw (g) edge[black] (j);
    \draw (h) edge[black] (k);
    \draw (i) edge[black] (l);

    \draw (aa) edge[black] (gg);
    \draw (a) edge[black] (g);
    \draw (b) edge[black] (h);
    \draw (c) edge[black] (i);
    \draw (dd) edge[black] (jj);
    \draw (d) edge[black] (j);
    \draw (e) edge[black] (k);
    \draw (f) edge[black] (l);

    \draw (jj) edge[black,bend left=14] (l);
    \draw (dd) edge[black,bend right=14] (f);
    \draw (gg) edge[black,bend left=10] (i);
    \draw (aa) edge[black,bend right=10] (c);
    }
    {\LARGE
    \node (coorda) at (0,0.5) {$a$};
    \node (coordb) at (5,0.5) {$b$};

    \node (coordA) at (-7,2) {$A$};
    \node (coordB) at (-7,7) {$B$};

    \node (coordAA) at (13.2,7) {$0$};
    \node (coordBB) at (11.8,9) {$1$};

    }
\end{tikzpicture}

%% file: img/amalgamation.tex
\definecolor{SetBlue}{RGB}{0, 90, 169} 
\definecolor{BoxGreen}{RGB}{183,235,127}

\begin{tikzpicture}[node distance = 3cm,
    def/.style={circle, draw=black!20, fill=black!7, thick, minimum size=12mm}, 
    every path/.style={draw=black,line width=1.3pt,{>=Latex}}, 
      S/.style={circle, draw=black!60, fill=SetBlue!70, thick, minimum size=12mm}, 
      phan/.style={draw=none, fill=none, inner sep=0pt, outer sep=0pt} 
      ] 
    {\huge

    \draw[rounded corners=5mm, draw=black!60, fill=BoxGreen!15] (-1,-1) rectangle (6,11);
    \draw[rounded corners=5mm, draw=black!60,fill=BoxGreen!15] (-1,-1) rectangle (11,6);
    \draw[rounded corners=5mm, draw=black!60, fill=BoxGreen!30] (6,6) -- (-1,6) -- (-1,-1) -- (6,-1) -- (6,6); 
    \node[S] (a) at (0,0) {};
    \node[S] (b) at (5,0) {};
    \node[S] (c) at (10,0) {};
    \node[S] (d) at (0,5) {};
    \node[S] (e) at (5,5) {};
    \node[S] (f) at (10,5) {};
    \node[S] (g) at (0,10) {};
    \node[S] (h) at (5,10) {};
    \node[def] (i) at (10,10) {};
        
    \draw (a) edge[black] (b);
    \draw (b) edge[black] (c);
    \draw (d) edge[black] (e);
    \draw (e) edge[black] (f);
    \draw (g) edge[black] (h);
    \draw (h) edge[black] (i);
        
    \draw (a) edge[black] (d);
    \draw (b) edge[black] (e);
    \draw (c) edge[black] (f);
    \draw (g) edge[black] (d);
    \draw (h) edge[black] (e);
    \draw (i) edge[black] (f);

    \draw (a) edge[gray,dashed,bend left=15] (c);
    \draw (d) edge[gray,dashed,bend left=15] (f);
    \draw (g) edge[gray,dashed,bend left=15] (i);
    \draw (a) edge[gray,dashed,bend left=15] (g);
    \draw (b) edge[gray,dashed,bend left=15] (h);
    \draw (c) edge[gray,dashed,bend left=15] (i);
    
    }
    {\Huge
    \node (coorda) at (0,-2) {$a$};
    \node (coordb) at (5,-2) {$b$};
    \node (coordc) at (10,-2) {$c$};

    \node (coordA) at (-2,0) {$A$};
    \node (coordB) at (-2,5) {$B$};
    \node (coordC) at (-2,10) {$C$};
    \node () at (2.5,7.5) {$S_1$};
    \node () at (7.5,2.5) {$S_2$};
    }
\end{tikzpicture}

%% file: img/MPGexample2.tex
\definecolor{p0Blue}{RGB}{0, 90, 169} 
\definecolor{p1Red}{RGB}{230, 0, 26} 
\begin{tikzpicture}[node distance = 3cm,
    p0/.style={circle, draw=p0Blue!60, fill=p0Blue!40, thick, minimum size=12mm}, 
    gen/.style={circle, draw=black, fill=black!20, thick, minimum size=12mm}, 
    p0small/.style={circle, draw=p0Blue!60, fill=p0Blue!40, thick, minimum size=6mm}, 
    every path/.style={->,draw=p0Blue,line width=1.3pt,{>=Latex}}, 
      p1/.style={rectangle, draw=p1Red!60, fill=p1Red!40, thick, minimum size=11mm}, 
      phan/.style={draw=none, fill=none, inner sep=0pt, outer sep=0pt}] 
    {\Large
    \node[p0] (v1) at (0,-1) {$v_1$};
    \node[p1] (a) at (-2,2) {$a$};
    \node[p1] (b) at (0,2) {$b$};
    \node[p1] (c) at (2,2) {$c$};
    \node[p0] (v2) at (6,0) {$v_2$};
    \node[p1] (A) at (4,2) {$A$};
    \node[p1] (B) at (6,2) {$B$};
    \node[p0] (v3) at (9,6) {$v_3$};
    \node[p1] (n0) at (8,-2) {$0$};
    \node[p1] (n1) at (6,-2) {$1$};

    \draw (v1) edge (a);
    \draw (v1) edge node[pos=.5,color=black] {$-1$} (b);
    \draw (v1) edge node[pos=.5,color=black] {$-3$} (c);
    \draw (v2) edge (A);
    \draw (v2) edge node[pos=.5,right,color=black] {$2$} (B);
    \draw (v3) edge[] (n0);
    \draw (v3) edge[bend left=10] node[pos=.5, left, color=black]{$4$} (n1);

    \draw (a) edge[loop above, color=p1Red] node[pos=.5,above,color=black] {$1$} (a);
    \draw (b) edge[p1Red, bend right=15] (v2);
    \draw (b) edge[p1Red] (v3);
    \draw (c) edge[p1Red] (v3);
    \draw (A) edge[p1Red] node[pos=.3,right,color=black] {$-1$} (v3);
    \draw (A) edge[p1Red, bend left=10] (v1);
    \draw (B) edge[p1Red] (A);
    \draw (n0) edge[p1Red] (v2);
    \draw (n1) edge[p1Red] (v1);
    \draw (n1) edge[p1Red] (v2);
    
}
\end{tikzpicture}